\documentclass{amsart}
\usepackage{amssymb}
\usepackage{amsfonts}
\usepackage{amssymb}
\usepackage{amsmath, accents}
\usepackage{amsthm}
\usepackage{pdfpages}
\usepackage{hyperref}
\usepackage{enumerate}
\usepackage{tabularx}
\usepackage{centernot}
\usepackage{mathtools}
\usepackage{stmaryrd}
\usepackage{amsthm,amssymb}
\usepackage{etoolbox}
\usepackage{tikz}
\usepackage{amssymb}
\usetikzlibrary{matrix}
\usepackage{tikz-cd}
\usepackage{tikz}
\usepackage{marginnote}
\definecolor{mygray}{gray}{0.85}
\usepackage[backgroundcolor=mygray,colorinlistoftodos,prependcaption,textsize=small]{todonotes}
\usepackage{xcolor}

\renewcommand{\leq}{\leqslant}
\renewcommand{\geq}{\geqslant}

\makeatletter
\def\subsection{\@startsection{subsection}{3}%
  \z@{.5\linespacing\@plus.7\linespacing}{.3\linespacing}%
  {\bfseries\centering}}
\makeatother

\makeatletter
\def\subsubsection{\@startsection{subsubsection}{3}%
  \z@{.5\linespacing\@plus.7\linespacing}{.3\linespacing}%
  {\centering}}
\makeatother

\makeatletter
\def\myfnt{\ifx\protect\@typeset@protect\expandafter\footnote\else\expandafter\@gobble\fi}
\makeatother

\newtheorem{theorem}{Theorem}[section]

\newtheorem{corollary}[theorem]{Corollary}
\newtheorem{definition}[theorem]{Definition}
\newtheorem{lemma}[theorem]{Lemma}

\newtheorem{observation}[theorem]{Observation}
\newtheorem{fact}[theorem]{Fact}
\newtheorem{conclusion}[theorem]{Conclusion}

\newtheorem{remark}[theorem]{Remark}

\newtheorem{notation}[theorem]{Notation}

\newtheorem{hypothesis}[theorem]{Hypothesis}
\newtheorem{cclaim}[theorem]{Claim}

\newtheorem{definition-proposition}[theorem]{Definition/Proposition}

\newtheorem{convention}[theorem]{Convention}

\newtheorem*{maintheorem}{Main Theorem}

\newcounter{claimcounter}
\numberwithin{claimcounter}{theorem}

\usepackage{eso-pic}
\newcommand{\mrm}[1]{\mathrm{#1}}

\begin{document}
%%%%%%%%%%%%%%%%%%

\begin{abstract} We prove that the Borel space of torsion-free Abelian groups with domain $\omega$ is Borel complete, i.e., the isomorphism relation on this Borel space is as complicated as possible, as an isomorphism relation. This solves a long-standing open problem in descriptive set theory, which dates back to the seminal paper on Borel reducibility of Friedman and Stanley from 1989. 
%We then move to some 
%\mbox{(co-)Hopfianity} questions for abelian groups. Firstly, we prove that the subset of $\mrm{TFAB}_\omega$ consisting of the Hopfian groups is a complete co-analytic subset of $\mrm{TFAB}_\omega$, %Lastly, we move from $\mrm{TFAB}_\omega$ to $3$-nilponent groups with domain $\omega$ ($\mathrm{NGp(3)}_\omega$) proving that the subset of $\mathrm{NGp(3)}_\omega$ consisting of the co-Hopfian groups is a complete analytic subset of $\mathrm{NGp(3)}_\omega$. 
%thus solving an open problem of Thomas, which posed the question for the space of all groups with domain $\omega$. 
%Secondly, we prove that despite the subsets of $\mrm{TFAB}_\omega$ consisting of the co-Hopfian groups is a Borel subset of $\mrm{TFAB}_\omega$, if we search for a ``better notion'' of co-Hopfianity for $\mrm{TFAB}_\omega$, then we have that the resulting subset is a complete co-analytic subset of $\mrm{TFAB}_\omega$.
\end{abstract}

\title[Torsion-Free Abelian Groups are Borel Complete]{Torsion-Free Abelian Groups are Borel Complete}

%\todog{Notes in gray are the standard (usually not major) things.}
%\todor{Notes in red are mathematical points which I do not understand. It does NOT mean that I believe that there are false claims, but these are more substantial points.}
%\todoy{Notes in yellow are somewhat general points (suggestions/questions/comments).}

\thanks{No. 1205 on Shelah's publication list. Research of both authors partially supported by NSF grant no: DMS 1833363. Research of the first author partially supported by project PRIN 2017 ``Mathematical Logic: models, sets, computability", prot. 2017NWTM8R. Research of the second author partially supported by Israel Science Foundation (ISF) grant no: 1838/19. We would like to thank C. Laskowski, D. Ulrich, and the anonymous referee for useful comments that lead to a simplification of the main construction of this paper.}

\author{Gianluca Paolini}
\address{Department of Mathematics ``Giuseppe Peano'', University of Torino, Via Carlo Alberto 10, 10123, Italy.}
%\email{gianluca.paolini@unito.it}

\author{Saharon Shelah}
\address{Einstein Institute of Mathematics,  The Hebrew University of Jerusalem, Israel \and Department of Mathematics,  Rutgers University, U.S.A.}

\date{\today}
\maketitle

%\tableofcontents
%\includepdf[pages=-]{file.pdf}

\section{Introduction}

	Since the seminal paper of Friedman and Stanley on Borel complexity \cite{friedman_and_stanley}, descriptive set theory has proved itself to be a decisive tool in the analysis of complexity problems for classes of countable structures. A canonical example of this phenomenon is the famous result of Thomas from \cite{thomas_JAMS} which shows that the complexity of the isomorphism relation for torsion-free abelian groups of rank $1 \leq n < \omega$ (denoted as $\cong_n$) is strictly increasing with $n$, thus, on one hand, finally providing a satisfactory reason for the difficulties found by many eminent mathematicians in finding systems of invariants for torsion-free abelian groups of rank $2 \leq n < \omega$ which were as simple as the one provided by Baer for $n = 1$ (see \cite{baer}), and, on the other hand, showing that for no $1 \leq n < \omega$ the relation $\cong_n$ is universal among countable Borel equivalence relations.
	As a matter of facts, abelian group theory has been one of the most important fields of mathematics from which taking inspiration for forging the general theory of Borel complexity as well as for finding some of the most striking applications thereof. The present paper continues this tradition solving one of the most important problems in the area, a problem open since the above mentioned paper of Friedman and Stanley from 1989. In technical terms, we prove that the space of countable torsion-free abelian groups with domain $\omega$ is {\em Borel complete}.

	As we will see in detail below, saying that a class of countable structures is Borel complete means that the isomorphism relation on this class is as complicated as possible, as an isomorphism relation. The Borel completeness of countable abelian group theory is particularly interesting from the perspective of model theory, as this class is model theoretically ``low'', i.e., stable (in the terminology of \cite{classification}). In fact, as already observed in \cite{friedman_and_stanley}, Borel reducibility can be thought of as a weak version of $\mathfrak{L}_{\omega_1, \omega}$-interpretability, and for other classes of countable structures such as groups or fields much stronger results than Borel completeness exist, as in such cases we can first-order interpret graph theory, but such classes are  unstable, while abelian group theory is stable. Reference \cite{las} starts a systematic study of the relations between Borel reducibility and classification theory in the context of $\aleph_0$-stable theories. 

\smallskip

	Coming back to us, we now introduce the notions from descriptive set theory which are necessary to understand our results, and we try to make a complete historical account of the problems which we tackle in this paper.
		The starting point of the descriptive set theory of countable structures is the following fact:
	
	\begin{fact}\label{borel_structure} The set $\mathrm{K}_{\omega}^L$ of structures with domain $\omega$ in a given countable language $L$ is endowed with a standard Borel space structure $(\mathrm{K}_{\omega}^L, \mathcal{B})$. Every Borel subset of this space $(\mathrm{K}_{\omega}^L, \mathcal{B})$ is naturally endowed with the Borel structure induced by $(\mathrm{K}_{\omega}^L, \mathcal{B})$.%, we will not distinguish between $\mathrm{K}$ and the naturally associated Borel space. 
\end{fact}

	For example, if take $L = \{ e, \cdot, ()^{-1} \}$, and we let $K'$ to be one of the following:
	\begin{enumerate}[(a)]
	\item the set of elements of $\mathrm{K}_{\omega}^L$ which are groups;
	\item the set of elements of $\mathrm{K}_{\omega}^L$ which are abelian groups;
	\item the set of elements of $\mathrm{K}_{\omega}^L$ which are torsion-free abelian groups;
	\item the set of elements of $\mathrm{K}_{\omega}^L$ which are $n$-nilpotent groups, for some $n < \omega$;
	\end{enumerate}
then we have that $\mathrm{K}'$ is a Borel subset of $(\mathrm{K}_{\omega}^L, \mathcal{B})$, and so Fact~\ref{borel_structure} applies.

Thus, given a class $K'$ as in Fact~\ref{borel_structure}, we can consider $K'$ as a standard Borel space, and so we can analyze the complexity of certain subsets of this space or of certain relations on it (i.e., subsets of $K' \times K'$ with the product Borel space structure). 
%For example we can ask the following questions:
%
%	\begin{problem}[Thomas]\label{intro_prob1} How complex is the set of (co-)Hopfian groups as a subset of the space of countable groups with domain $\omega$? 
%\end{problem}
%
%	\begin{problem}\label{intro_prob2} How complex is the isomorphism relation as a subset of $K' \times K'$, where $K'$ is the space of torsion-free abelian groups with domain $\omega$?
%\end{problem}
	Further, this technology allows us to compare pairs of classes of structures or, in another direction, pairs of relations defined on pairs of classes of structures.
	
	\begin{definition}\label{Borel_reducibility} Let $X_1$ and $X_2$ be two standard Borel spaces, and let also $Y_1 \subseteq X_1$ and $Y_2 \subseteq X_2$. We say that $Y_1$ is reducible to $Y_2$, denoted as $Y_1 \leq_R Y_2$, when there is a Borel map $\mathbf{B}: X_1 \rightarrow X_2$ such that for every $x \in X_1$ we have:
	$$x \in Y_1 \Leftrightarrow \mathbf{B}(x) \in Y_2.$$
\end{definition}

	Notice that Definition~\ref{Borel_reducibility} covers in particular the case $X_1 = K' \times K'$ for $K'$ as in Fact~\ref{borel_structure}, and so for example $Y_1$ could be the isomorphism relation on $K'$. %Thus, as result we could ask: is the isomorphism relation on countable graphs more complex than the isomorphism relation on countable torsion-free abelian groups?
	Also, given a Borel space $X$, we can ask if there are subsets of $X$ which are $\leq_R$-maxima with respect to a fixed family of subsets of an arbitrary Borel space (e.g., Borel sets, analytic sets, co-analytic sets, etc). In particular we can define:
	
	\begin{definition}\label{complete_analytic} Let $X_1$ be a Borel space and $Y_1 \subseteq X_1$. We say that $Y_1$ is {\em complete analytic} (resp. complete co-analytic) if for every Borel space $X_2$ and analytic subset (resp. co-analytic subset) $Y_2$ of $X_2$ we have that $Y_2 \leq_R Y_1$. 
\end{definition}

%	Thus, in light of the terminology from Definition~\ref{complete_analytic}, we can ask further in Problems~\ref{intro_prob1}~and~\ref{intro_prob2} if the subsets from there are {\em complete analytic}. 

	We now introduce the notion of Borel reducibility among equivalence relations.
	
	\begin{definition}\label{def_Borel_red} Let $X_1$ and $X_2$ be two standard Borel spaces, and let also $E_1$ be an equivalence relation defined on $X_1$ and $E_2$ be an equivalence relation defined on $X_2$. We say that $E_1$ is Borel reducible to $E_2$, denoted as $E_1 \leq_B E_2$, when there is a Borel map $\mathbf{B}: X_1 \rightarrow X_2$ such that for every $x, y \in X_1$ we have:
	$$x E_1 y \; \Leftrightarrow \mathbf{B}(x) E_2 \mathbf{B}(y).$$
\end{definition} 

	\begin{remark} Notice that in the context of Definitions~\ref{Borel_reducibility} and \ref{def_Borel_red}, $E_1 \leq_R E_2$ and $E_1 \leq_B E_2$ have two different meaning, as in the first case the witnessing Borel function has domain $X \times X$, while in the second case it has domain~$X$. Furthermore, notice that $E_1 \leq_B E_2$ implies  $E_1 \leq_R E_2$ (but the converse need not hold, see \ref{complete_implies_analytic}).
\end{remark}

	We now define {\em Borel completeness}, the notion at the heart of our paper.
	
	\begin{definition} Let $\mathrm{K}_1$ be a Borel class of structures with domain $\omega$ and let $\cong_1$ be the isomorphism relation on $\mathrm{K}_1$. We say that $\mathrm{K}_1$ is {\em Borel complete} (or, in more modern terminology, $\cong_1$ is $\mrm{S}_\infty$-complete) if for every Borel class $\mathrm{K}_2$ of structures with domain $\omega$ there is a Borel map $\mathbf{B}: \mathrm{K}_2 \rightarrow \mathrm{K}_1$ such that for every $A, B \in \mathrm{K}_2$:
$$A \cong B \; \Leftrightarrow \mathbf{B}(A) \cong_1 \mathbf{B}(B),$$
that is, the isomorphism relation on the space $\mathrm{K}_2$ is Borel reducible (in the sense of Definition~\ref{def_Borel_red}) to the the isomorphism relation on the space $\mathrm{K}_1$.
\end{definition}

	The following fact will be relevant for our subsequent historical account.

	\begin{fact}[\cite{friedman_and_stanley}]\label{complete_implies_analytic} Let $\mathrm{K}$ be a Borel class of structures with domain $\omega$. If $\mathrm{K}$ is Borel complete, then its isomorphism relation is a complete analytic subset of $\mrm{K} \times \mrm{K}$, but the converse need not hold, as for example abelian $p$-groups with domain $\omega$ have complete analytic isomorphism relation but they are {\em not} a Borel complete space.
\end{fact}

	We now have all the ingredients necessary to be able to understand the problems that we solve in this paper and to introduce the state of the art concerning them. But first a useful piece of notation which we will use throughout the paper.
	
	\begin{notation}\label{graph_not} \begin{enumerate}[(1)]
	\item We denote by $\mathrm{Graph}$ the class of graphs.
	\item We denote by $\mathrm{Gp}$ the class of groups.
	\item We denote by $\mathrm{AB}$ the class of abelian groups.
	\item We denote by $\mathrm{TFAB}$ the class of torsion-free abelian groups.
	%\item We denote by $\mrm{NGp}(3)$ the class of $3$-nilpontent groups.
	\item\label{graph_not_5} Given a class $\mathrm{K}$ we denote by $\mathrm{K}_{\omega}$ the set of structures in $\mathrm{K}$ with domain $\omega$.
	\end{enumerate}
	\end{notation}
	
	\begin{convention} To simplify statements, we use the following convention: when we say that a class $\mathrm{K}$ of countable structures  is Borel complete we mean that $\mathrm{K}_\omega$ is Borel complete. Similarly, when we say that a class $\mathrm{K}$ of countable groups is complete co-analytic we mean that $\mathrm{K}_\omega$ is a complete co-analytic subset of $\mathrm{Gp}_\omega$. Finally, when we say that the isomorphism relation on a class of countable groups is analytic, we mean that restriction of the isomorphism relation on $\mathrm{K}$ to $\mathrm{K}_\omega \times \mathrm{K}_\omega$ is an analytic subset of the Borel space $\mathrm{Gp}_\omega \times \mathrm{Gp}_\omega$ (as a product space).
\end{convention}
	
	In \cite{friedman_and_stanley}, together with the general notions just defined, the authors studied some Borel complexity problems for specific classes of countable structures of interest. Among other things they showed (we mention only the results relevant to us):
	\begin{enumerate}[(i)]
	\item countable graphs, linear orders and trees are Borel complete;
	\item torsion abelian groups have complete analytic $\cong$ but are {\em not} Borel complete;
	\item nilpotent groups of class $2$ and exponent $p$ ($p$ a prime) are Borel complete\footnote{As already mentioned in \cite{friedman_and_stanley}, this result is actually a straightforward adaptation of a model theoretic construction due to Mekler \cite{mekler}.};
	\item the isomorphism relation on finite rank torsion-free abelian groups is Borel.
\end{enumerate}	 

	In \cite{friedman_and_stanley} Friedman and Stanley state explicitly:
	
\begin{quote} There is, alas, a missing piece to the puzzle, namely our conjecture that torsion-free abelian groups are complete. [...] We have not even been able to show that the isomorphism relation on torsion-free abelian groups is complete analytic, nor, in another direction, that the class of all abelian groups is Borel complete. We consider these problems to be among the most important in the subject.
\end{quote}

	The challenge was taken by several mathematicians. The first to work on this problem was Hjorth, which in \cite{hjorth} proved that any Borel isomorphism relation is Borel reducible (in the sense of Definition~\ref{def_Borel_red}) to the isomorphism relation on countable torsion-free abelian groups, and that in particular the isomorphism relation on $\mrm{TFAB}_\omega$ is not Borel (as there is no such Borel equivalence relation), leaving though open the question whether $\mrm{TFAB}_\omega$ is a Borel complete class, or even whether the isomorphism relation on $\mrm{TFAB}_\omega$ is complete analytic (cf. Def. \ref{complete_analytic} and Fact~\ref{complete_implies_analytic}).	

	 The problem resisted further attempts of the time and the interest moved to another very interesting problem on torsion-free abelian groups: for $1 \leq n < m < \omega$, is the isomorphism relation $\cong_n$ on torsion-free abelian groups of rank $n$ strictly less complex (in the sense of Definition~\ref{def_Borel_red}) than the isomorphism relation on torsion-free abelian groups of rank $m$? As mentioned above, the isomorphism relation on torsion-free abelian groups of finite rank is Borel while, as just mentioned, the isomorphim relation on countable torsion-free abelian groups is not, and so the two problems are quite different, but obviously related. Also this problem proved to be very challenging, until Thomas finally gave a positive solution to the problem, in a series of two fundamental papers \cite{thomas_ACTA, thomas_JAMS}, proving in particular that, for every $n < \omega$, $\cong_n$ is not universal among countable Borel equivalence relations. %In \cite{thomas_ACTA, thomas_JAMS} Thomas, continuing on work of Hjorth and Kechris \cite{hjorth_kechris}, also proved interesting results on the isomorphism relation on {\em rigid} torsion-free abelian groups of finite rank, where, in their terminology, an abelian group $G$ is said to be rigid when the only automorphisms of $G$ are the identity and the map $g \mapsto -g$. This notion of rigidity is relevant also to our main theorem, as we will show that the Borel map witnessing our Borel complete construction can be taken to have range in the rigid abelian torsion-free groups.
	
	The fundamental work of Thomas thus resolved completely the case of torsion-free abelian groups of finite rank, leaving open the problem for countable torsion-free abelian groups of arbitrary rank, i.e. the problem referred to as ``among the most important in the subject'' in \cite{friedman_and_stanley}. The problem remained ``dormant'' for various years (at the best of our knowledge), until Downey and Montalb{\'a}n \cite{dow_mon} made some important progress showing that the isomorphism relation on countable torsion-free abelian groups is complete analytic, a necessary but not sufficient condition for Borel completeness, as recalled in Fact~\ref{complete_implies_analytic}. This was of course possible evidence that the isomorphism relation was indeed Borel complete, as conjectured in \cite{friedman_and_stanley}. Despite this advancement, the problem of Borel completeness of countable torsion-free abelian groups resisted for other 12 years, until this day, when we prove:
	 
%	 Other references which we want to mention are a paper of Calderoni and Thomas [...], which shows that the bi-embeddability relation on torsion-free abelian groups is a complete analytic equivalence relation, and a paper of Ulrich and Shelah [...] on some consistency results related to the problem of Borel completeness of $\mrm{TFAB}_\omega$.

\begin{maintheorem}\label{main_theorem} The space $\mathrm{TFAB}_\omega$ is Borel complete, in fact there exists a continous map $\mathbf{B}: \mathrm{Graph}_{\omega} \rightarrow \mathrm{TFAB}_\omega$ s.t. for every $H_1, H_2, \in \mathrm{Graph}_{\omega}$ we have:
$$H_1 \cong H_2 \text{ if and only if } \mathbf{B}(H_1) \cong \mathbf{B}(H_2).$$
%Furthermore, for every $H \in \mathrm{Graph}_{\omega}$, $\mrm{Aut}(\mathbf{B}(H)) \cong ((\prod_{\ell < \omega} N_\ell) \rtimes \mrm{Aut}(H)) \times \mathbb{Z}_2$, for some explicitly defined subgroups, for $\ell < \omega$, $N_\ell$ of $\mrm{Aut}(\mathbf{B}(H))$.
\end{maintheorem}

%\todoo{The change in (ii) above is minor ONLY if the following is true: there is a continous map $\mathbf{F}: \mathrm{Graph}_{\omega} \rightarrow \mathrm{Graph}_{\omega}$ such that for $H_1, H_2 \in \mathrm{Graph}_{\omega}$ we have:
%\newline - $H_1 \cong H_2$ if and only if $\mathbf{F}(H_1) \cong \mathbf{F}(H_2)$;
%\newline - $F(H_1)$ is rigid, i.e. its automorphism group has size $1$.
%\newline Do you know if the above is true? Actually I had though about this myself + Boban Velickovic asked me about this (I gave a talk in their seminar last week).}
%\todob{Just to understand each others. You write "No clear". You mean the question is clear but the answer is not clear, right? Otherwise I can try to make the question clearer.}

%	Clause (ii) of our Main Theorem, shows that the Borel (actually continuous) function witnessing Borel reducibility of the isomorphism relation on $\mathrm{Graph}_{\omega}$ to the isomorphism relation on $\mathrm{TFAB}_\omega$ can be taken to have range in the subset of rigid abelian groups, in the sense of Hjorth, Kechris and Thomas mentioned above. In fact this might be considered of independent interest and it is the infinite rank counterpart of the results of Hjorth, Kechris and Thomas relating the isomorphism relation on rigid and general torsion-free abelian groups of finite rank, e.g. Thomas famously showed that the isomorphism relation on the rigid $\mathrm{TFAB}_\omega$ of rank $n+1$ does not Borel reduces to the isomorphism relation on $\mathrm{TFAB}_\omega$ of rank $n \geq 1$.
	
	The techniques employed in the proof of our Main Theorem lead us to (and at the same time were inspired by) classification questions of ``rigid'' countable abelian groups. One of the most important notions of rigidity in abelian group theory is the notion of endorigidity, where we say that $G \in \mrm{AB}$ is endorigid if the only endomorphisms of $G$ are multiplication by an integer. The analysis of endorigid abelian groups is an old topic in abelian group theory, famous in this respect is the result of the second author \cite{sh44} that for every infinite cardinal $\lambda$ there is an endorigid torsion-free abelian group of cardinality $\lambda$. We prove in Thereom~\ref{co-hopfian_TFAB_th} below that the classification of the countable endorigid $\mrm{TFAB}$ is an highly untractable problem.
	
%	of co-Hopfian torsion-free abelian groups, where we recall that a group $G$ is said to be co-Hopfian if $G$ does not have proper subgroups $H$ isomorphic to $G$, i.e., every injective endomorphism of $G$ is surjective. As well-known (see e.g. \cite[Proposition~2.2, pg. 130]{gobel}), for $G \in \mathrm{TFAB}$, $G$ is co-Hopfian iff $G$ is divisible and of finite rank, i.e. $G$ is a finitely dimensional vector space over $\mathbb{Q}$, and so clearly the co-Hopfian groups form a Borel subset of $\mrm{TFAB}_\omega$. We wonder: what if replace the notion of surjective morphism with a notion of  ``almost-surjective'' morphism which is appropriate for the class $\mathrm{TFAB}$? Does the classification problem becomes intractable? In particular we might consider:

\begin{theorem}\label{co-hopfian_TFAB_th} 
The set of endorigid torsion-free abelian groups is a complete co-analytic subset of the Borel space space $\mathrm{TFAB}_\omega$. In fact, more strongly, there is a Borel function $\mathbf{F}$ from the set of tree with domain $\omega$ into $\mrm{TFAB}_\omega$ such that:
\begin{enumerate}[(i)]
	\item if $T$ is well-founded, then $\mathbf{F}(T)$ is endorigid;
	\item if $T$ is not well-founded, then $\mathbf{F}(T)$ has a $1$-to-$1$ $f \in \mrm{End}(G)$ which is not multiplication by an integer and such that $G/f[G]$ is not torsion.
	\end{enumerate}
\end{theorem}

	In \cite{2005} we extend the ideas behind Theorem~\ref{co-hopfian_TFAB_th} to a systematic investigation of several classification problems for various rigidity conditions on abelian and nilpotent groups from the perspective of descriptive set theory of countable structures. In another direction, in \cite{2035} we study the question of existence of uncountable \mbox{(co-)Hopfian} abelian groups, this work was later continued by the second author et al. in the preprint  \cite{goldshani}, which settles some questions left open in \cite{2035}.
	
\medskip
\noindent 

	We conclude with a few words on the history of this article. At the end of the refereeing process, the referee indicated some points which needed correction in the original version of this paper. Around the same time, Laskowski and Ulrich indicated another point which needed correction in our original submission. The referee also asked to change the presentation of our Main Theorem and to simply its proof, in particular separating the algebra from the combinatorics (division which is reflected by the current division in Sections~\ref{S3} and~\ref{Borel_complete_section}). Here all the points raised there are addressed. We thank the anonymous referee, Laskowski and Ulrich. Meanwhile, Laskowski and Ulrich have found another proof of our Main Theorem, see \cite{las_new, las_new+}.

\section{Notations and Preliminaries}

	For the readers of various backgrounds we try to make the paper self-contained.

\subsection{General notations}

	\begin{definition}\label{def_general_notation}
	\begin{enumerate}[(1)]
	\item Given a set $X$ we write $Y \subseteq_\omega X$ for $\emptyset \neq Y \subseteq X$ and $|Y| < \aleph_0$.
	\item Given a set $X$ and $\bar{x}, \bar{y} \in X^{< \omega}$ we write $\bar{y} \triangleleft \bar{x}$ to mean that $\mrm{lg}(\bar{y}) < \mrm{lg}(\bar{x})$ and $\bar{x} \restriction \mrm{lg}(\bar{y}) = \bar{y}$, where $\bar{x}$ is naturally considered as a function ${\mrm{lg}(\bar{x})} \rightarrow X$.
	\item Given a partial function $f: M \rightarrow M$, we denote by $\mrm{dom}(f)$ and $\mrm{ran}(f)$ the domain and the range of $f$, respectively.
	\item For $\bar{a} \in B^n$ we write $\bar{x} \subseteq B$ to mean that $\mrm{ran}(\bar{x}) \subseteq B$, where, as usual, $\bar{a}$ is considered as a function $\{0, ..., n-1 \} \rightarrow B$.
	\item Given a sequence $\bar{f} = (f_i : i \in I)$ we write $f \in \bar{f}$ to mean that there exists $j \in I$ such that $f = f_j$.
	\end{enumerate}
\end{definition}

\subsection{Groups}

\begin{notation} Let $G$ and $H$ be groups.
\begin{enumerate}[(1)]
	\item $H \leq G$ means that $H$ is a subgroup of $G$.
	\item We let $G^+ = G \setminus \{ e_G \}$, where $e_G$ is the neutral element of $G$.
	\item If $G$ is abelian we might denote the neutral element $e_G$ simply as $0_G = 0$.
\end{enumerate}
\end{notation}

	\begin{definition} Let $H \leq G$ be groups, we say that $H$ is pure in $G$, denoted by $H \leq_* G$, when if $h \in H$, $0 < n < \omega$, $g \in G$ and (in additive notation) $G \models ng = h$, then there is $h' \in H$ s.t. $H \models nh' = h$. Given $S \subseteq G$ we denote by $\langle S \rangle^*_S$ the pure subgroup generated by $S$ (the intersection of all the \mbox{pure subgroups of $G$ containing~$S$).}
\end{definition}

	\begin{observation}\label{obs_pure_TFAB} $H \leq_* G \in \mrm{TFAB}$, $h \in H$, $0 < n < \omega$, $G \models ng = h \Rightarrow g \in H$.
\end{observation}

	\begin{observation}\label{generalG1p_remark} Let $G \in \mathrm{TFAB}$, $p$ a prime and let:
	$$G_{p} = \{ a \in G : a \text{ is divisible by $p^m$, for every $0 < m < \omega$} \},$$
then $G_{p}$ is a pure subgroup of $G$.
\end{observation}

	\begin{proof} This is well-known, see e.g. the discussion in \cite[pg. 386-387]{harrison}.
\end{proof}

\begin{definition}\label{Qp} Let $p$ be a prime. We let:
$$\mathbb{Q}_p = \{\frac{m_1}{m_2} : m_1 \in \mathbb{Z}, m_2 \in \mathbb{Z}^+, p \text{ and $m_2$ are coprime}\}.$$
\end{definition}

%	\begin{definition}\label{Qp} Let $p$ be a prime.
%	\begin{enumerate}[(1)]
%	\item We let $\mathbb{Q}_p = \{\frac{m_1}{m_2} : m_1 \in \mathbb{Z}, m_2 \in \mathbb{Z}^+, p \text{ and $m_2$ are coprime}\}$.
%%	\item\label{Qp+} We let $\mathbb{Q}^{\oplus}_p = \{q \in \mathbb{Q}^+ : pq \in \mathbb{Q}_p, q \notin \mathbb{Q}_p\}$.
%	\item\label{Qp+} We let $\mathbb{Q}^{\odot}_p = \{q \in \mathbb{Q}^+ : q \in \mathbb{Q}_p, \frac{q}{p} \notin \mathbb{Q}_p\}$, so $\mathbb{Q}^{\odot}_p \cap \mathbb{Z} = \{q \in \mathbb{Z}^+ : p \not\vert \; q\}$.
%\end{enumerate}
%\end{definition}

\subsection{Trees}
%
%\begin{definition}\label{def_digraph} By a directed graph we mean a structure in the language $L = \{R\}$, where $R$ is a binary predicate symbol. We say that the directed graph $M$ is irreflexive when $M \models \forall x (\neg R(x, x))$. We say that the directed graph $M$ is asymmetric when $M \models \forall x \forall y (R(x, y) \rightarrow \neg R(y, x))$. We say that the directed graph $M$ has no cycles (or that it is acyclic) when there is no $n < \omega$ and $x_0, ..., x_n \in M$ such that:
%$$M \models x_0 = x_n, \; M \models R(x_n, x_0) \text{ and, for every $i < n$, } M \models R(x_i, x_{i+1}).$$
%\end{definition}
%
%
%	\begin{definition}\label{def_graph} By a graph we mean a structure $M$ in the language $L = \{R\}$, where $R$ is a binary predicate symbol, satisfying the following axioms:
%	\begin{enumerate}[(i)]
%	\item $\forall x (\neg R(x, x))$ (irreflexivity of $R$);
%	\item $\forall x \forall y (R(x, y) \rightarrow R(y, x))$ (symmetry of $R$).
%	\end{enumerate}
%$M$ has no cycles when there is no $2 < n < \omega$ and distinct $x_0, ..., x_n \in M$ such that:
%$$M \models x_0 = x_n, \; M \models R(x_n, x_0) \text{ and, for every $i < n$, } M \models R(x_i, x_{i+1}).$$
%\end{definition}

	\begin{definition} Given an $L$-structure $M$ by a partial automorphism of $M$ we mean a partial function $f: M \rightarrow M$ such that $f: \langle \mrm{dom}(f) \rangle_M \cong \langle \mrm{ran}(f) \rangle_M$.
\end{definition}

	In Section~\ref{sec_cohop} we shall use the following notions.

	\begin{definition}\label{def_trees} Let $(T, <_T)$ be a strict partial order.
	\begin{enumerate}[(1)]
	\item $(T, <_T)$ is a {\em tree} when, for all $t \in T$, $\{s \in T : s <_T t\}$ is well-ordered by the relation $<_T$. Notice that according to our definition a tree $(T, <_T)$ might have more than one root, i.e. more than one $<_T$-minimal element. We say that the tree $(T, <_T)$ is rooted when it has only one $<_T$-minimal element (its root).
	\item A {\em branch} of the tree $(T, <_T)$ is a maximal chain of the partial order $(T, <_T)$.
	\item A tree $(T, <_T)$ is said to be {\em well-founded} if it has only finite branches.
	\item Given a tree $(T, <_T)$ and $t \in T$ we let the level of $t$ in $(T, <_T)$, denoted as $\mrm{lev}(t)$, to be the order type of $\{s \in T : s <_T t\}$ (recall item (1)).
	\end{enumerate}
\end{definition}

\begin{remark} Concerning Def.~\ref{def_trees}(4), we will only consider trees $(T, <_T)$ such that, for every $t \in T$, $\{s \in T : s <_T t\}$ is finite, so for us $\mrm{lev}(t) \in \omega$.
\end{remark}

\section{The Combinatorial Frame}\label{S3}

%	In Section~\ref{Borel_complete_section} (where the main construction of this paper occurs) we only use \ref{seq_not}-\ref{hyp_C1}, so the reader willing to do so could simply digest \ref{seq_not}-\ref{hyp_C1} and move there.

	The isomorphism problem for countable models of the theory of two equivalence relations is known to be at least as complex as the isomorphism problem for any other Borel class of countable structures. In what follows we will reduce this problem to the isomorphism problem for countable $\mrm{TFAB}$'s. Our reduction will consist of an elaborated coding of finite partial automorphisms $g$'s of models with two equivalence relations into partial automorphisms $f_g$'s of $\mrm{TFAB}$'s. For technical reasons we will actually work with finite sequences $\bar{g}$ of finite partial automorphisms, and to avoid the troublesome case $g = g^{-1}$ we will actually work with objects of the form $(\bar{g}, \iota)$ with $\iota \in \{ 0, 1\}$. Finally, also for technical reasons, it will be useful to consider models of the theory of three equivalence \mbox{relation, with one of them being equality.}
	
\smallskip

The definition of $\mathfrak{m} \in \mathrm{K}^{\mrm{bo}}_2(M)$ which we will introduce in Definition~\ref{hyp_C1} is phrased just to construct a certain $G_M \in \mrm{TFAB}$ (see Definition~\ref{def_G012_borel}), for any relevant $M$, and in fact, as we will seee, it will suffice (and it will be very useful to do so) to construct such as $G_M$ just for $M$ the countable homogeneous universal model of the theory of two equivalence relations, cf. \ref{hyp_C0}(\ref{universal_model}). In this respect, below $X$ will serve as set of generators for $G_M$ (in the appropriate sense) and $a \in M$ will be coded as $X'_{\{ a \}} \subseteq X$ (in a certain sense). As the $f_{\bar{g}}$'s are finite partial automorphisms  related to the partial automorphisms in $\bar{g}$, it is natural to ask that if $\bar{g} = (g_0, ..., g_n)$, then $f_{\bar{g}}$ maps elements in $\mrm{dom}(f_{\bar{g}}) \cap X'_{\{ a \}}$ into elements in $X'_{\{g_n(a)\}}$. Almost all clauses in Definition~\ref{hyp_C1} right below are combinatorial, but, not surprisingly, one is more algebraic, namely clause (8). As it will be clear from reading Section~4 below, this clause is {\em crucial} in reconstructing an isomorphism between models of the theory of two equivalence relations from an isomorphism between $\mrm{TFAB}_\omega$'s.

	\begin{notation}\label{seq_not} For $Z$ a set and $0 < n < \omega$, we let $\mathrm{seq}_n(Z) = \{ \bar{x} \in Z^n : \bar{x} \text{ injective} \}$.
\end{notation}

	\begin{hypothesis}\label{hyp_C0} \begin{enumerate}[(1)]
	\item\label{hyp_C0_item1} $\mathbf{K}^{\mrm{eq}}$ is the class of models $M$ in a vocabulary $\{\mathfrak{E}_0, \mathfrak{E}_1, \mathfrak{E}_2\}$ such that each $\mathfrak{E}^M_i$ is an equivalence relation and $\mathfrak{E}^M_2$ is the equality relation. We use the symbol $\mathfrak{E}_i$ to avoid confusions, as the symbol $E_i$ appears in \ref{hyp_C1}.
	\item\label{universal_model} $M$ is the countable homogeneous universal model in $\mathbf{K}^{\mrm{eq}}$.
	\item\label{hyp_C0_item2} $\mathcal{G}$ is essentially the set of finite non-empty partial automorphisms $g$ of $M$ but for technical reasons\footnote{The reason is that we want to force that $g \neq g^{-1}$.} it is the set of objects $g = (\mathbf{h}_g, \iota_g)$ where:
	\begin{enumerate}[(A)]
	\item 
	\begin{enumerate}[(a)]
	\item $\mathbf{h}_g$ is a finite non-empty partial automorphism of $M$;
	\item $\iota_g \in \{0, 1\}$;
	\end{enumerate}
	\item for $g \in \mathcal{G}$ we let:
	\begin{enumerate}[(a)]
	\item $g^{-1} = (\mathbf{h}^{-1}_g, 1-\iota_g)$;
	\item for $a \in M$, $g(a) = \mathbf{h}_g(a)$;
	\item for $\mathcal{U} \subseteq M$, $g[\mathcal{U}] = \{ \mathbf{h}_g(a) : a \in \mathcal{U} \}$;
	\item $g_1 \subseteq g_2$ means $\mathbf{h}_{g_1} \subseteq \mathbf{h}_{g_2}$ and $\iota_{g_1} = \iota_{g_2}$;
	\item $g_1 \subsetneq g_2$ means $g_1 \subseteq g_2$ and $g_1 \neq g_2$;
	\item $\mrm{dom}(g) = \mrm{dom}(\mathbf{h}_{g})$ and $\mrm{ran}(g) = \mrm{ran}(\mathbf{h}_{g})$;
	\item for $\mathcal{U} \subseteq M$, $g \restriction \mathcal{U} = (\mathbf{h}_{g} \restriction \mathcal{U}, \iota_g)$.
\end{enumerate}
\end{enumerate}	 
	\item\label{G_*} For $m < \omega$, $\mathcal{G}_{*}^{m} = \{ (g_0, ..., g_{m-1}) \in \mathcal{G}^m: g_0 \subsetneq \cdots \subsetneq g_{m-1} \}$.
	\item\label{G_*_real} $\mathcal{G}_* = \bigcup \{\mathcal{G}^m_* : m < \omega \}$ (notice that the empty sequence belongs to $\mathcal{G}_*$).
	\end{enumerate}
\end{hypothesis}

%	\begin{remark}\label{disjoint_remark} Notice that \ref{hyp_C0}(\ref{hyp_C0_item2}) implies that if $g \in \mathcal{G}$ then $\emptyset \neq s \subseteq \mrm{dom}(g) \Rightarrow g[s] \neq s$. Notice also that if $(g_0, ..., g_{m-1}) \in \mathcal{G}^m$ and $\ell \leq m-2$, then $g_\ell \neq g^{-1}_{\ell+1}$. Finally, on the other hand it is very well possible that $g \in \mathcal{G}$ and $\mrm{dom}(g) \cap \mrm{ran}(g) \neq \emptyset$.
%\end{remark}

	\begin{notation}\label{dom_barg_notation}
	\begin{enumerate}[(1)]
	\item We use $s, t, ...$ to denote finite non-empty subsets of $M$ and $\mathcal{U}, \mathcal{V}, ...$ to denote arbitrary subsets of $M$. Recall from \ref{def_general_notation} that $\subseteq_\omega$ means finite subset.
	\item For $A$ a set, we let $s \subseteq_1 A$ mean $s \subseteq A$ and $|s| = 1$.
	\item For $\bar{g} = (g_0, ..., g_{\mrm{lg}(\bar{g})-1}) \in \mathcal{G}_{*}^{\mrm{lg}(\bar{g})}$ and $s, t \subseteq_\omega M$, we let:
	\begin{enumerate}
	\item\label{barg_a} for $a, b \in M$, $\bar{g}(a) = b$ mean that $g_{\mathrm{lg}(\bar{g})-1}(a) = b$;
	\item $\bar{g}[s] = t$ mean that $g_{\mathrm{lg}(\bar{g})-1}[s] = t$;
	\item\label{item_domain} $\mathrm{dom}(\bar{g}) = \mathrm{dom}(g_{\mathrm{lg}(\bar{g})-1})$, and $\emptyset$ if $\mathrm{lg}(\bar{g}) = 0$;
	\item\label{item_range} $\mathrm{ran}(\bar{g}) = \mathrm{ran}(g_{\mathrm{lg}(\bar{g})-1})$, and $\emptyset$ if $\mathrm{lg}(\bar{g}) = 0$;
	\item $\bar{g}^{-1} = (g^{-1}_i : i < \mrm{lg}(\bar{g}))$;
	\item $\bar{g}((x_\ell : \ell < n)) = (\bar{g}(x_\ell): \ell < n)$.
\end{enumerate}
\end{enumerate}
\end{notation}

\begin{definition}\label{hyp_C1} In the context of Hyp.~\ref{hyp_C0}, let $\mathrm{K}^{\mrm{bo}}_2(M)$ be the class of objects (called {\em systems}) $\mathfrak{m}(M) = \mathfrak{m} = (X^\mathfrak{m}, \bar{X}^\mathfrak{m}, \bar{f}^\mathfrak{m}, \bar{E}^\mathfrak{m}) = (X, \bar{X}, \bar{f}, \bar{E})$ such that:
	\begin{enumerate}[(1)]
	\item\label{hyp_C1_item1} $X$ is an infinite countable set and $X \subseteq \omega$;
	\item
	\begin{enumerate}
	\item\label{item_partition_a} $(X'_s : s \subseteq_{1} M)$ is a partition of $X$ into infinite sets;
	\item\label{item_partition_b}  for $s \subseteq_{\omega} M$, let $X_s = \bigcup _{t \subseteq_1 s} X'_t$;
	\item $\bar{X} = (X_s : s \subseteq_{\omega} M)$ and so $s \subseteq t \subseteq_{\omega} M$ implies $X_s \subseteq X_t$;
\end{enumerate}	 
	%\item if $s \subseteq_\omega M$, then $X'_s = X_s \setminus \bigcup \{X_t : t \subsetneq s \}$ is infinite;
	\item\label{def_XU} for $\mathcal{U} \subseteq M$ let $X_{\mathcal{U}} = \bigcup\{X_s : s \subseteq_1 \mathcal{U} \}$ and so $X = X_M = \bigcup\{X_s : s \subseteq_1 M \}$;
	\item\label{item7} $\bar{f} = (f_{\bar{g}} : \bar{g} \in \mathcal{G}_*)$ (recall the definition of $\mathcal{G}_*$ from \ref{hyp_C0}(\ref{G_*_real})) and:
	\begin{enumerate}[(a)]
	\item\label{item_8a} $f_{\bar{g}}$ is a finite partial bijection of $X$ and $f_{\bar{g}}$ is the empty function iff $\mrm{lg}(\bar{g}) = 0$;
	%with no cycles, i.e. there are no $0 < n  < \omega$ and $x_0, ..., x_n \in \mrm{dom}(f_{\bar{g}})$ such that $x_0 = x_n$ and, for every $i < n$, $f_{\bar{g}}(x_i) = x_{i+1}$, so in particular $f_{\bar{g}}(x) \neq x$ and $f_{\bar{g}}(x) = y \Rightarrow f_{\bar{g}}(y) \neq x$;	
	%\item $f_{(\bar{g}, -1)} = (f_{(\bar{g}, 1)})^{-1}$;
	\item\label{item_dom_f} $\mathrm{dom}(f_{\bar{g}}) \subseteq X_{\mathrm{dom}(\bar{g})}$ and $\mathrm{ran}(f_{\bar{g}}) \subseteq X_{\mathrm{ran}(\bar{g})}$ (cf. \ref{dom_barg_notation}(\ref{item_domain})(\ref{item_range})),\mbox{ so $\mrm{dom}(f_{()}) = \emptyset$;}% where $\mathrm{dom}(\bar{g}) = \mathrm{dom}(\bigcup_{i < \mathrm{lg}(\bar{g})}g_i)$;
	%\item\label{item_ran_f} $\mathrm{ran}(f_{(\bar{g}, 1)}) \subseteq X_{\mathrm{ran}(\bar{g})}$, where $\mathrm{ran}(\bar{g}) = \mathrm{ran}(\bigcup_{i < \mathrm{lg}(\bar{g})}g_i)$;
	%\item let $f_{\bar{g}} = f_{(\bar{g}, 1)}$;
	\item\label{itemf} for $s, t \subseteq_1 M$ and $\bar{g}[s] = t$ we have:
$$f_{\bar{g}}(x) = y \text{ implies } (x \in X'_s \text{ iff }  y \in X'_t).$$
	\item\label{item7d} for $s, t \subseteq_1 M$, ($f_{\bar{g}}(x) = y, x \in X'_s, y \in X'_t$) implies ($\bar{g}[s] = t$);
	\item\label{item_7e_inverses} $f_{\bar{g}^{-1}} = f^{-1}_{\bar{g}}$ (recall that $\bar{g}^{-1} \neq \bar{g}$, when $\mrm{dom}(\bar{g}) \neq \emptyset$);
\end{enumerate}
	\item\label{extending_thef's} $\bar{g}, \bar{g}' \in \mathcal{G}_*$, $\bar{g} \triangleleft \bar{g}'$ $\Rightarrow$  $f_{\bar{g}} \subsetneq f_{\bar{g}'}$;
	\item\label{the_graph_Rn} we define the graph $(\mathrm{seq}_n(X), R^{\mathfrak{m}}_n)$ as $(\bar{x}, \bar{y}) \in R^{\mathfrak{m}}_n = R_n$ when $\bar{x} \neq \bar{y}$ and:
	$$\text{ for some } \bar{g} \in \mathcal{G}_* \text{ we have } f_{\bar{g}}(\bar{x}) = \bar{y},$$
notice that $f^{-1}_{\bar{g}} = f_{\bar{g}^{-1}} \in \bar{f}$, as $\bar{g} \in \mathcal{G}_*$ implies $\bar{g}^{-1} \in \mathcal{G}_*$;
	\item\label{def_En}  $\bar{E}^{\mathfrak{m}} = \bar{E} = (E_n : 0 < n < \omega) = (E^{\mathfrak{m}}_n : 0 < n < \omega)$, and, for $0 < n < \omega$, 
	$E_n$ is the equivalence relation corresponding to the partition of $\mrm{seq}_n(X)$ given by the connected components of the graph $(\mrm{seq}_n(X), R_n)$;
	%\item\label{no_cycles} for every $n \geq 1$, the graph $(\mathrm{seq}_n(X), R_n)$ has no cycles;
	\item\label{item17_new} if $p$ is a prime, $k \geq 2$, $\bar{x} \in \mrm{seq}_k(X)$, $\mathbf{y} = (\bar{y}^i : i < i_*) \in (\bar{x}/E^{\mathfrak{m}}_k)^{i_*}$, with the $\bar{y}^i$'s pairwise distinct, $\bar{r} \in \mathbb{Q}^{\mathbf{y}}$, $q_\ell \in \mathbb{Q}_p$, for $\ell < k$, and:
	$$a_{(\mathbf{y}, \bar{r})}(y) = a_{(\mathbf{y}, \bar{r}, y)} = \sum \{r_{\bar{y}} q_\ell : \ell < k, \bar{y} = \bar{y}^i, i < i_*, y = y^i_\ell\},$$
	for $y \in \mrm{set}(\mathbf{y}) = \bigcup \{\mrm{ran}(\bar{y}^i) : i < i_* \}$, then we have the following:
	$$|\{y \in \mrm{set}(\mathbf{y}) : a_{(\mathbf{y}, \bar{r})}(y) \notin \mathbb{Q}_p\}| \neq 1,$$
where we recall that $\mathbb{Q}_p$ was defined in Definition~\ref{Qp};
	\item\label{item_for_iso_ltr} if for every $n < \omega$, $g_n \in \mathcal{G}$ and $g_n \subsetneq g_{n+1}$, $\mathcal{U} = \bigcup_{n < \omega} \mrm{dom}(g_n) \subseteq M$ and $\mathcal{V} = \bigcup_{n < \omega} \mrm{ran}(g_n) \subseteq M$, then we have the following:
	$$\bigcup_{n < \omega} \mrm{dom}(f_{(g_\ell \, : \, \ell < n)}) = X_\mathcal{U} \text{ and } \bigcup_{n < \omega} \mrm{ran}(f_{(g_\ell \, : \, \ell < n)}) = X_\mathcal{V}.$$
\end{enumerate}
\end{definition}

	The definition of $\mathfrak{m} \in \mathrm{K}^{\mrm{bo}}_2(M)$ from \ref{hyp_C1} isolates exactly what is needed for the group theoretic construction from Section~4 to take place. The rest of this section has as its sole purpose to show that an object as in Definition~\ref{hyp_C1} exists. To this extent, we introduce an auxiliary class of objects, $ \mathrm{K}^{\mrm{bo}}_1(M)$, cf. Definition~\ref{hyp_C1+}. This definition is devised with a twofold aim in mind: on one hand to put more detailed information on the objects at play in Definition~\ref{hyp_C1}, and on the other hand to be able to construct the desired $\mathfrak{m} \in \mathrm{K}^{\mrm{bo}}_2(M)$ as a limit of a sequence of approximations $\mathfrak{m}_\ell \in \mathrm{K}^{\mrm{bo}}_1(M)$, for $\ell < \omega$, of such an $\mathfrak{m} \in \mathrm{K}^{\mrm{bo}}_2(M)$. In this process the crucial algebraic condition (8) from Definition~\ref{hyp_C1} gets translated in the more technical algebraic condition (11) from Definition~\ref{hyp_C1+}, showing that this condition is preserved in the limit construction will be the most elaborated part of this section.

	\begin{definition}\label{hyp_C1+} In the context of Hyp.~\ref{hyp_C0}, let $\mathrm{K}^{\mrm{bo}}_1(M)$ be the class of objects $\mathfrak{m}(M) = \mathfrak{m} = (X^\mathfrak{m}, \bar{X}^\mathfrak{m}, I^\mathfrak{m}, \bar{I}^\mathfrak{m}, \bar{f}^\mathfrak{m}, \bar{E}^\mathfrak{m}, Y_{\mathfrak{m}}) = (X, \bar{X}, I, \bar{I}, \bar{f}, \bar{E}, Y)$ s.t.:
	\begin{enumerate}[(1)]
	\item\label{hyp_C1_item1+} $X$ is an infinite countable set and $X \subseteq \omega$;
	\item
	\begin{enumerate}
	\item\label{item_partition+} $(X'_s : s \subseteq_{1} M)$ is a partition of $X$ into infinite sets;
	\item for $s \subseteq_{\omega} M$, let $X_s = \bigcup _{t \subseteq_1 s} X'_t$;
	\item $\bar{X} = (X_s : s \subseteq_{\omega} M)$ and so $s \subseteq t \subseteq_{\omega} M$ implies $X_s \subseteq X_t$;
\end{enumerate}	 
	%\item if $s \subseteq_\omega M$, then $X'_s = X_s \setminus \bigcup \{X_t : t \subsetneq s \}$ is infinite;
	\item\label{def_XU+} for $\mathcal{U} \subseteq M$ let $X_{\mathcal{U}} = \bigcup\{X_s : s \subseteq_1 \mathcal{U} \}$ and so $X = X_M = \bigcup\{X_s : s \subseteq_1 M \}$;
	\item\label{def_In+} \begin{enumerate}
	\item\label{item4a+} $\bar{I} = (I_n : n < \omega) = (I^{\mathfrak{m}}_n : n < \omega)$ are pairwise disjoint;
	%\item\label{for_why_iso} $I_n = \bigcup_{m \leq n} \mathcal{G}_{*}^{m}$, where $\mathcal{G}_{*}^{m} = \{ (g_0, ..., g_{m-1}) \in \mathcal{G}^m: g_0 \subseteq \cdots \subseteq g_{m-1} \}$;
	\item $\bar{g} \in I_n$ implies $\bar{g} \in \mathcal{G}_{*}^{m}$ for some $m \leq n$;
	%\item $\bar{g} \in I_n$ implies $\mrm{lg}(\bar{g}) \leq n$;
	\item $I_n$ is finite;
	\end{enumerate} 
	\item\label{3.3(5)+} if $\bar{g}' \triangleleft \bar{g} \in I_n$, then $\bar{g}' \in I_{< n} := \bigcup_{\ell < n} I_\ell$;
	\item $I = I^{\mathfrak{m}} = \bigcup_{n < \omega} I_n$;
	\item\label{item7+} $\bar{f} = (f_{\bar{g}} : \bar{g} \in I)$ and:
	\begin{enumerate}[(a)]
	\item\label{item_8a+} $f_{\bar{g}}$ is a finite partial bijection of $X$ and $f_{\bar{g}}$ is the empty function iff $\mrm{lg}(\bar{g}) = 0$;
% with no cycles, i.e. there are no $0 < n  < \omega$ and $x_0, ..., x_n \in \mrm{dom}(f_{\bar{g}})$ such that $x_0 = x_n$ and, for every $i < n$, $f_{\bar{g}}(x_i) = x_{i+1}$, so in particular $f_{\bar{g}}(x) \neq x$ and $f_{\bar{g}}(x) = y \Rightarrow f_{\bar{g}}(y) \neq x$;	
	\item\label{item_dom_f+} $\mathrm{dom}(f_{\bar{g}}) \subseteq X_{\mathrm{dom}(\bar{g})}$ and $\mathrm{ran}(f_{\bar{g}}) \subseteq X_{\mathrm{ran}(\bar{g})}$ (cf. Notation~\ref{dom_barg_notation}(\ref{item_domain})(\ref{item_range}));% where $\mathrm{dom}(\bar{g}) = \mathrm{dom}(\bigcup_{i < \mathrm{lg}(\bar{g})}g_i)$;
	%\item\label{item_ran_f} $\mathrm{ran}(f_{(\bar{g}, 1)}) \subseteq X_{\mathrm{ran}(\bar{g})}$, where $\mathrm{ran}(\bar{g}) = \mathrm{ran}(\bigcup_{i < \mathrm{lg}(\bar{g})}g_i)$;
	%\item let $f_{\bar{g}} = f_{(\bar{g}, 1)}$;
	\item\label{itemf+} for $s, t \subseteq_1 M$ and $\bar{g}[s] = t$ we have:
$$f_{\bar{g}}(x) = y \text{ implies } (x \in X'_s \text{ iff }  y \in X'_t).$$
	\item\label{item7d+} for $s, t \subseteq_1 M$, ($f_{\bar{g}}(x) = y, x \in X'_s, y \in X'_t$) implies ($\bar{g}[s] = t$);
	\item\label{item_7e_inverses+} if $\bar{g} \in I_n$, then $\bar{g}^{-1} \in I_n$ and $f_{\bar{g}^{-1}} = f^{-1}_{\bar{g}}$;
\end{enumerate}
	\item $\bar{g} \triangleleft \bar{g}'$ $\Rightarrow$  $f_{\bar{g}} \subsetneq f_{\bar{g}'}$;
%	\item for $Z \subseteq X$, we let $\mathrm{seq}(Z) = \bigcup_{0 < n < \omega} \mathrm{seq}_n(Z)$, where, for $0 < n < \omega$, we let:
%	$$\mathrm{seq}_n(Z) = \{ \bar{x} \in Z^n : \bar{x} \text{ injective} \};$$
	%\item\label{item9} if $\bar{x} \in \mathrm{seq}(X)$ is included in $\mathrm{dom}(f_{(\bar{g}, i)})$, then $\bar{x} E_n f_{(\bar{g}, i)}(\bar{x})$;
	\item\label{the_graph_Rn+} we define the graph $(\mathrm{seq}_n(X), R^{\mathfrak{m}}_n)$ as $(\bar{x}, \bar{y}) \in R^{\mathfrak{m}}_n = R_n$ when $\bar{x} \neq \bar{y}$ and:
	$$\text{ for some } \bar{g} \in \mathcal{G}_* \text{ we have } f_{\bar{g}}(\bar{x}) = \bar{y},$$
notice that $f^{-1}_{\bar{g}} = f_{\bar{g}^{-1}} \in \bar{f}$, as $\bar{g} \in I$ implies $\bar{g}^{-1} \in I$;
	%\item\label{no_cycles+} the graph $(\mathrm{seq}_n(X), R_n)$ has no cycles;
	\item\label{item12+}
	\begin{enumerate}
	\item\label{def_En+} $\bar{E}^{\mathfrak{m}} = \bar{E} = (E_n : n < \omega) = (E^{\mathfrak{m}}_n : n < \omega)$, and, for $n < \omega$, 
	%$E_n$ is the minimal equivalence relation on $\mrm{seq}_n(X)$ such that if $\bar{x}, \bar{y} \in \mrm{seq}_n(X)$ and there is $\bar{g} \in I_n$ such that $f_{\bar{g}}(\bar{x}) = \bar{y}$, then $\bar{x} E_n \bar{y}$, equivalently, 
	$E_n$ is the equivalence relation corresponding to the partition of $\mrm{seq}_n(X)$ given by the connected components of the graph $(\mrm{seq}_n(X), R_n)$;
	\item\label{item13c+} $Y = Y_{\mathfrak{m}}$ is a non-empty subset of $X$ which \underline{includes} the following set:
	$$\{x \in X : \text{ for some } \bar{g} \in I, \, x \in \mrm{dom}(f_{\bar{g}})\},$$
	notice that this inclusion may very well be proper;
	\item $\mrm{seq}_k(\mathfrak{m}) = \{\bar{x} \in \mrm{seq}_k(X) : \text{ for some } \bar{g} \in I, \, \bar{x} \subseteq \mrm{dom}(f_{\bar{g}})\}$, notice $\mrm{seq}_k(\mathfrak{m}) \subseteq \mrm{seq}_k(Y_{\mathfrak{m}})$ but the converse need not hold;
	\end{enumerate}
%	\item\label{item_primes} $\bar{p}^\mathfrak{m}$ is a sequence of prime numbers without repetitions such that:
%	$$\bar{p}^\mathfrak{m} = \{ p_{(e, \bar{q})}	: (e, \bar{q}) \in W_{\mathfrak{m}}\}$$
%	$$W_{\mathfrak{m}} \subseteq W_* = \{(e, \bar{q}) : e \in \mathrm{seq}_n(X)/E_n \text{ for some } 0 < n < \omega \text{ and } \bar{q} \in (\mathbb{Z} \setminus \{0 \})^n );$$
%	\item\label{item_old18+} $S^\mathfrak{m} \subseteq I^\mathfrak{m}$ satisfies the following conditions: 
%	\begin{enumerate}
%	\item if $\bar{g} \in I^\mathfrak{m}$ is trivial, i.e., $\mrm{lg}(\bar{g}) = 0$, then $\bar{g} \in S^{\mathfrak{m}}$;
%	\item for non-trivial $\bar{g} \in I^\mathfrak{m}$, $\bar{g} \in S^{\mathfrak{m}}$ iff $\bar{g}^{-1} \notin S^{\mathfrak{m}}$;
%	\item if $\bar{g}^\frown(g) \in I^\mathfrak{m}$ and $\bar{g}$ is non-trivial, then $\bar{g}^\frown(g) \in S^\mathfrak{m} \Leftrightarrow \bar{g} \notin S^{\mathfrak{m}}$;
%	\end{enumerate}
%	\item 
%	\begin{enumerate} 
%	\item for $k \geq 1$ and $\bar{x} \in \mrm{seq}_k(Y_{\mathfrak{m}})$, let:
%	$$\mrm{suc}_k^{\mathfrak{m}}(\bar{x}) = \{f_{\bar{g}}(\bar{x}) : \bar{g} \in S^{\mathfrak{m}}, \bar{x} \in \mrm{seq}_k(\mrm{dom}(f_{\bar{g}})), \bar{h} \triangleleft \bar{g} \Rightarrow \bar{x} \notin \mrm{seq}_k(\mrm{dom}(f_{\bar{h}}))\};$$
%	\item\label{order<_k} the transitive closure of $\{(\bar{x}, \bar{y}) : \bar{x} \in \mrm{seq}_k(Y_{\mathfrak{m}}), \bar{y} \in \mrm{suc}_k^{\mathfrak{m}}(\bar{x}) \}$, denoted as $<^{\mathfrak{m}}_k$, is a strict partial order on the set $\mrm{seq}_k(Y_{\mathfrak{m}})$;
	%\end{enumerate}
	\item\label{item17_new+} if $p$ is a prime, $k \geq 2$, $\bar{x} \in \mrm{seq}_k(X)$, $\bar{q} \in (\mathbb{Q}_p)^k$, $\mathfrak{s} = (p, k, \bar{x}, \bar{q})$ and $\bar{a} \in \mathcal{A}_{\mathfrak{s}}$, then $\mrm{supp}_p(\bar{a})$ is not a singleton, \underline{where} we define $\mathcal{A}_{\mathfrak{s}}$, $\mathcal{A}_{\mathfrak{m}}$ and $\mrm{supp}_p(\bar{a})$ as follows:
	\begin{enumerate}[(a)]
	\item\label{item17_new+_a} $\mathcal{A}_{\mathfrak{s}} \subseteq \mathcal{A}_{\mathfrak{m}} = \{(a_y : y \in Z) : Z \subseteq_\omega X \text{ and } a_y \in \mathbb{Q}\}$;
	\item\label{item17_new+_b} if $\bar{a} \in \mathcal{A}_{\mathfrak{m}}$, then we let:
	$$\mrm{supp}_p(\bar{a}) = \{y \in \mrm{dom}(\bar{a}) : a_y \notin \mathbb{Q}_p\};$$
	% and we let:
	%$$\bar{a} \approx \bar{b}\; \Leftrightarrow \;\bar{a} \restriction \mrm{supp}_p(\bar{a}) = \bar{b} \restriction \mrm{supp}_p(\bar{a});$$
	\item\label{item17_new+_c} if $\mathbf{y} = (\bar{y}^i : i < i_*) \in (\bar{x}/E^{\mathfrak{m}}_k)^{i_*}$ (but abusing notation we may treat $\mathbf{y}$ as a set), with the $\bar{y}^i$'s pairwise distinct and $\bar{r} \in \mathbb{Q}^{\mathbf{y}}$, then $\bar{a} \in \mathcal{A}_{\mathfrak{s}}$, where:
	$$\bar{a} = \bar{a}_{(\mathbf{y}, \bar{r})} = (a_y : y \in \mrm{set}(\mathbf{y})),$$
and where $a_y$ and $\mrm{set}(\mathbf{y})$ are defined as follows:
$$a_y = a_{(\mathbf{y}, \bar{r})}(y) = a_{(\mathbf{y}, \bar{r}, y)} = \sum \{r_{\bar{y}} q_\ell : \ell < k, \bar{y} = \bar{y}^i, i < i_*, y = y^i_\ell\},$$
	$$\mrm{set}(\mathbf{y}) = \bigcup \{\mrm{ran}(\bar{y}^i) : i < i_*\};$$
	\item\label{item17_new+_d} if $\bar{a} \in \mathcal{A}_{\mathfrak{s}}$ and $\mrm{supp}_p(\bar{a}) \subseteq Z \subseteq \mrm{dom}(\bar{a})$, then $\bar{a} \restriction Z \in \mathcal{A}_{\mathfrak{s}}$;
	\item\label{item17_new_e} if $\bar{a}, \bar{b} \in \mathcal{A}_{\mathfrak{s}}$, then $\bar{c} = \bar{a} + \bar{b} \in \mathcal{A}_{\mathfrak{s}}$, where $\mrm{dom}(\bar{c}) = \mrm{dom}(\bar{a}) \cup \mrm{dom}(\bar{b})$ and:
	\begin{enumerate}[(i)]
	\item $c_y = a_y + b_y$, if $y \in \mrm{dom}(\bar{a}) \cap \mrm{dom}(\bar{b})$;
	\item $c_y = a_y$, if $y \in \mrm{dom}(\bar{a}) \setminus \mrm{dom}(\bar{b})$;
	\item $c_y = b_y$, if $y \in \mrm{dom}(\bar{b}) \setminus \mrm{dom}(\bar{a})$;
\end{enumerate}	 
	\item\label{item17_new_extra} if $\bar{g} \in I^{\mathfrak{m}}$, $Z_1 \subseteq_\omega \mrm{dom}(f_{\bar{g}})$, $Z_2 = f_{\bar{g}}[Z_1]$ and $\bar{a} = (a_y : y \in Z_2) \in \mathcal{A}_{\mathfrak{s}} $, \underline{then} 
	$$\bar{a}^{[f_{\bar{g}}]} = (a_{f_{\bar{g}}(y)} : y \in Z_1) \in \mathcal{A}_{\mathfrak{s}};$$
	\item\label{item17_new_f} $\mathcal{A}_{\mathfrak{s}}$ is the minimal subset of $\mathcal{A}_{\mathfrak{m}}$ satisfying clauses (c)-(f).
	\end{enumerate}
\end{enumerate}
\end{definition}

	As mentioned above, members in $\mathfrak{m} \in \mathrm{K}^{\mrm{bo}}_1(M)$ are to be thought of as approximations to objects in $\mathrm{K}^{\mrm{bo}}_2(M)$, but technically an $\mathfrak{m} \in \mathrm{K}^{\mrm{bo}}_1(M)$ and an $\mathfrak{m} \in \mathrm{K}^{\mrm{bo}}_2(M)$ are made of different components, so we give a name to the objects in $\mathfrak{m} \in \mathrm{K}^{\mrm{bo}}_1(M)$ which which are essentially members of $\mathrm{K}^{\mrm{bo}}_2(M)$, we call them full, see \ref{def_full_K_hop_bis}.

\begin{definition}\label{def_full_K_hop_bis} For $\mathfrak{m} \in \mathrm{K}^{\mrm{bo}}_1(M)$, we say that $\mathfrak{m}$ is full when in addition to (\ref{hyp_C1_item1})-(\ref{item17_new+}) condition \ref{hyp_C1}(\ref{item_for_iso_ltr}) is satisfied and \ref{hyp_C1+}(\ref{item7}) is strengthen to \ref{hyp_C1}(\ref{item7}) (that is we ask $I = \mathcal{G}_*$), explicitly to (\ref{hyp_C1_item1})-(\ref{item17_new+}) from \ref{hyp_C1+} we add: 
	\begin{enumerate}[(12)]
	\item\label{item_for_iso_ltr+} if for every $n < \omega$, $g_n \in \mathcal{G}$ and $g_n \subsetneq g_{n+1}$, $\mathcal{U} = \bigcup_{n < \omega} \mrm{dom}(g_n) \subseteq M$ and $\mathcal{V} = \bigcup_{n < \omega} \mrm{ran}(g_n) \subseteq M$, then we have the following:
	$$\bigcup_{n < \omega} \mrm{dom}(f_{(g_\ell \, : \, \ell < n)}) = X_\mathcal{U} \text{ and } \bigcup_{n < \omega} \mrm{ran}(f_{(g_\ell \, : \, \ell < n)}) = X_\mathcal{V};$$
\end{enumerate}
\begin{enumerate}[(13)]
	\item\label{I_equal_G} $I = \bigcup_{n < \omega} I_n = \mathcal{G}_*$. 
\end{enumerate}
%	if $\bar{g} \in I_n$, $i \in \{+1, -1 \}$ and $u \subseteq X$ is finite and such that $\mathrm{dom}(f_{(\bar{g}, i)}) \subseteq u$, then for some $g_n \in \mathcal{G}$ we have that:
%	\begin{enumerate}[(i)]
%	\item $\bar{g}' = (g_0, ..., g_n) \in I_{n+1}$;
%	\item $u \subseteq \mathrm{dom}(f_{(\bar{g}', i)})$.
%	\end{enumerate}
\end{definition}

	We shall concentrate on the $\mathfrak{m} \in \mathrm{K}^{\mrm{bo}}_1(M)$ which are, in some sense, with ``finite information'', i.e., the ones in which both $Y_{\mathfrak{m}}$ and $I^{\mathfrak{m}}$ are finite. Furthermore, we will define a notion of ``$\mathfrak{n}$ is a successor of $\mathfrak{m}$''. These notions are tailor made for our inductive construction of a full $\mathfrak{m} \in \mathrm{K}^{\mrm{bo}}_1(M)$ to take place.
% i.e., (essentially) of an object $\mathfrak{m} \in \mathrm{K}^{\mrm{bo}}_2(M)$.

	\begin{definition}\label{def_K0}\begin{enumerate}[(1)]
	\item\label{def_K0_n_of} $\mathrm{K}^{\mrm{bo}}_0(M)$ is the class of $\mathfrak{m} \in \mathrm{K}^{\mrm{bo}}_1(M)$ such that $Y_{\mathfrak{m}}$ is finite and for some $n < \omega$ we have that for every $m \geq n$, $I_m = \emptyset$ and $I_0 = \{()\}$.
In this case we let $n = n(\mathfrak{m})$ to be the minimal such $n < \omega$ (so $n(\mathfrak{m}) > 0$). 
	\item\label{def_K0(2)} We say that $\mathfrak{n} \in \mathrm{suc}(\mathfrak{m})$ when:
	\begin{enumerate}[(a)]
	\item $\mathfrak{n}, \mathfrak{m} \in \mathrm{K}^{\mrm{bo}}_0(M)$, $X^{\mathfrak{m}} = X^{\mathfrak{n}}$;
	\item for $s \subseteq_1 M$, $(X'_s)^{\mathfrak{m}} = (X'_s)^{\mathfrak{n}}$;
	\item for $t \subseteq_\omega M$, $(X_t)^{\mathfrak{m}} = (X_t)^{\mathfrak{n}}$ (follows);
	\item $n(\mathfrak{n}) = n+1$, where $n(\mathfrak{m}) = n$;
	\item if $\ell < n(\mathfrak{m})$, then $I^{\mathfrak{m}}_\ell = I^{\mathfrak{n}}_\ell$ and $\bigwedge_{\bar{g} \in I^{\mathfrak{m}}_\ell} f^{\mathfrak{m}}_{\bar{g}} = f^{\mathfrak{n}}_{\bar{g}}$;
	\item for some $\bar{g} \in \mathcal{G}_*$, $I^{\mathfrak{n}}_{n} = \{ \bar{g}, \bar{g}^{-1} \}$, $lg({\bar{g}}) \leq n$ and $\ell < \mrm{lg}(\bar{g})$ implies:
	$$\bar{g} \restriction \ell \in \bigcup_{\ell < n} I^{\mathfrak{m}}_{\ell},$$
notice that $\bar{g} \notin \bigcup_{\ell < n} I^{\mathfrak{m}}_\ell$ (by Definition~\ref{hyp_C1+}(\ref{item4a+})) and the symmetric condition $\bar{g}^{-1} \restriction \ell \in \bigcup_{\ell < n} I^{\mathfrak{m}}_\ell$ follows from Definition~\ref{hyp_C1+}(\ref{item_7e_inverses+});
	\item\label{itemf_suc_m}
	\begin{enumerate}[$(\alpha)$]
	\item if $\bar{x} E^{\mathfrak{n}}_k \bar{y}$ and $\neg (\bar{x} E^{\mathfrak{m}}_k \bar{y})$, then $\bar{x} \notin \mrm{seq}_k(\mathfrak{m})$ or $\bar{y} \notin \mrm{seq}_k(\mathfrak{m})$;
\end{enumerate}	 
\begin{enumerate}[$(\beta)$]
	\item $E^\mathfrak{n}_k \restriction \mrm{seq}_k(\mathfrak{m}) = E^\mathfrak{m}_k \restriction \mrm{seq}_k(\mathfrak{m})$.
\end{enumerate}	
\end{enumerate}	
	\item $<_{\mathrm{suc}}$ on $\mathrm{K}^{\mrm{bo}}_0(M)$ is the transitive closure of the relation $\mathfrak{n} \in \mrm{suc}(\mathfrak{m})$.
	\end{enumerate}
\end{definition}

	The heart of this section is the following claim.

	\begin{cclaim}\label{K2bo_non_empty} For $M$ as in Hyp.~\ref{hyp_C0}, there exists $\mathfrak{m} \in \mathrm{K}^{\mrm{bo}}_1(M)$ which is full.
\end{cclaim}

	\begin{proof} Our strategy is to construct a full $\mathfrak{m} \in \mathrm{K}^{\mrm{bo}}_1(M)$ as a limit of members $\mathfrak{m}_\ell \in \mathrm{K}^{\mrm{bo}}_0(M)$, for $\ell < \omega$. Naturally, $\mathfrak{m}_0$ is not hard to choose, see $(*)_1$ below. Concerning the choice of the $\mathfrak{m}_{\ell}$'s, in $(*)_3$ below we list our tasks: for every $\bar{g} \in \mathcal{G}^*$ we have a $\bar{g}$-task which is ensuring that $f_{\bar{g}}$ is well-defined. Thus, we list $\mathcal{G}^*$ as $(\bar{g}_\ell : \ell < \omega)$ appropriately and in choosing $\mathfrak{m}_{\ell+1}$, a successor of $\mathfrak{m}_{\ell}$, we take care of the $\bar{g}_\ell$-task. 
%The choice of $(\bar{g}_\ell : \ell < \omega)$ ensures that $(*_2)$ is applicable to 
This lead us to the main part of the proof, namely $(*)_2$. Here we are given $\mathfrak{m}$ and appropriate $\bar{g}^\frown(g) \in \mathcal{G}^*$ such that $\bar{g} \in I_{\mathfrak{m}}$, i.e., $f_{\bar{g}}$ is already well-defined for $\mathfrak{m}$. Our aim is to define a suitable successor $\mathfrak{n}$ of $\mathfrak{m}$ and, in particular, to define $f_{\bar{g}^\frown(g)}$ for $\mathfrak{n}$. Moreover to take care of the fullness of the limit we want both $\mrm{dom}(f_{\bar{g}^\frown(g)})$ and $Y_{\mathfrak{m}}$ to be large enough, \mbox{this explains the statement of $(*)_2$.}
	
	\begin{enumerate}[$(*)_1$]
	\item $\mathrm{K}^{\mrm{bo}}_0(M) \neq \emptyset$.
\end{enumerate}
\noindent [Why? Let $\mathfrak{m}$ be such that:
\begin{enumerate}[(a)]
	\item $|X| = \aleph_0$, and $X \subseteq \omega$;
	\item $(X'_s : s \subseteq_1 M)$ is a partition of $X$ into infinite sets;
	\item for $s \subseteq_\omega M$, $X_s = \bigcup_{t \subseteq_1 s} X'_t$;
	\item $\bar{X} = (X_s : s \subseteq_\omega M)$;
	\item $I^{\mathfrak{m}}_0 = \{()\}$, $f_{()}$ is the empty function, $\bar{f} = (f_{()})$ and $I_{1+n} = \emptyset$, for every $n < \omega$;
	\item $Y_{\mathfrak{m}}$ is any finite non-empty subset of $X$.
	\end{enumerate}
Notice that $()$ denotes the empty sequence and under this choice of $\mathfrak{m}$,  $n(\mathfrak{m}) = 1$, where we recall that the notation $n(\mathfrak{m})$ was introduced in Definition~\ref{def_K0}(\ref{def_K0_n_of}). Notice also that \ref{hyp_C1+}(\ref{item17_new+}) is easy to verify for $\mathfrak{m}$ as above, as $\bar{x}/E^{\mathfrak{m}}_k$ is always a singleton.]
\begin{enumerate}[$(*)_2$]
	\item If $\mathfrak{m} \in \mathrm{K}^{\mrm{bo}}_0(M)$, $n = n(\mathfrak{m}) > 0$, $\bar{g} = (g_0, ..., g_{m-1}) \in I^{\mathfrak{m}}$ (so $n > m$) and:
	\begin{enumerate}[(i)]
	\item $g \in \mathcal{G}$;
	\item for every $\ell < m$, $g_\ell \subsetneq g$;
	\item $\bar{g}^\frown (g) \notin I^\mathfrak{m}$;
\end{enumerate}	
\underline{then} there is $\mathfrak{n} \in \mathrm{K}^{\mrm{bo}}_0(M)$ such that:
	\begin{enumerate}
	\item $\mathfrak{n} \in \mrm{suc}(\mathfrak{m})$;
	\item $\bar{g}^\frown (g) \in I^{\mathfrak{n}}_{n}$;
	\item\label{for_fulness} if $s \subseteq_1 s^+ = \mrm{dom}(g) \cup \mrm{ran}(g)$, then $Y_{\mathfrak{n}}$ contains $\mrm{min}(X'_s \setminus Y_{\mathfrak{m}})$;
	\item $\mrm{dom}(f^{\mathfrak{n}}_{\bar{g}^\frown(g)}) = Y_{\mathfrak{m}} \cap X_{\mrm{dom}(g)}$;
	\item so $n(\mathfrak{n}) = n(\mathfrak{m}) +1$.
	%\item $\mrm{dom}(f_{(\bar{g}^\frown (g))^{-1}})$ contains $\{ 0, ..., m-1 \} \cap X_{\mrm{dom}(g^{-1})}$.
\end{enumerate}	 
\end{enumerate}
The proof of $(*)_2$ is clearly the heart of the proof. The choice of $\mathfrak{n}$ in  $(*)_{2.3}$ below is natural: we choose $f^{\mathfrak{n}}_{\bar{g}^\frown(g)} = f_*$ ``freely'', i.e., it extends $f^{\mathfrak{m}}_{\bar{g}}$, it has large enough domain and no ``accidental equality'' holds. Lastly, $Y_{\mathfrak{n}}$ has to include $Y_{\mathfrak{m}}$, $\mrm{ran}(f_*)$ and witnesses toward the proof of fullness (cf. $(*)_2$(c)), which will be dealt with in the next successor step, so we are making sure that the induction goes on.
 
\smallskip
\noindent We thus move to the proof of $(*)_2$, where we let $f^{\mathfrak{m}}_{\bar{g}} = f_{\bar{g}}$.
	\begin{enumerate}[$(*)_{2.1}$]
	\item Let $s_* = \mrm{dom}(g) \subseteq_{\omega} M$, hence $\mrm{dom}(\bar{g}) \subsetneq s_*$ and let $u_* = Y_{\mathfrak{m}} \cap X_{s_*}$.
	\end{enumerate}
%	\begin{enumerate}[$(*)_{2.1}$]
%	\item Let $u_2 = Y_{\mathfrak{m}} \cap X_{s_*}$.
%	\end{enumerate}
%	\begin{enumerate}[$(*)_{2.1.1}$]
%	\item For $s \subseteq_{\omega} M$, $X'_s \cap u_1$ is an initial segment of $X'_s$ and $|X'_s \cap u_1|> \mrm{lg}(\bar{g}) + 2$;
%\end{enumerate}
	\begin{enumerate}[$(*)_{2.2}$]
	\item Let $f_*$ be a finite permutation of $X$ satisfying the following:
	\begin{enumerate}
	\item $f_*$ obeys \ref{hyp_C1+}(\ref{item_8a+})-(\ref{item7d+}) for $\bar{g}^\frown(g)$ and $\mrm{dom}(f_*) = u_*$;
	\item $f_*$ extends $f_{\bar{g}}$;
	\item $\mrm{dom}(f_*) \cap \mrm{ran}(f_*) = \mrm{ran}(f_{\bar{g}})$;
	\item if $x \in \mrm{dom}(f_*) \setminus \mrm{dom}(f_{\bar{g}})$ then $f_*(x) \notin Y_{\mathfrak{m}}$ (so $f_{*}(x) \notin \mrm{dom}(f_*)$).
	\end{enumerate}	
\end{enumerate}
%It follows that:
We now define $\mathfrak{n}$, as required in $(*)_2$.
	\begin{enumerate}[$(*)_{2.3}$]
	\item
	\begin{enumerate}[(A)]
	\item 
	\begin{enumerate}[(a)]
	\item $X^{\mathfrak{n}} = X^{\mathfrak{m}}$ and $\bar{X}^{\mathfrak{n}} = \bar{X}^{\mathfrak{m}}$;
	\item $I^{\mathfrak{n}}_n = \{ \bar{g}^\frown(g), (\bar{g}^{-1})^\frown(g^{-1}) \}$;
	\item $I^{\mathfrak{n}} = I^{\mathfrak{m}} \cup I^{\mathfrak{n}}_n$;
	\item $I^{\mathfrak{n}}_\ell = I^{\mathfrak{m}}_\ell$, for $\ell \neq n$;
	\item $f^{\mathfrak{n}}_{\bar{h}} = f^{\mathfrak{m}}_{\bar{h}}$, for $\bar{h} \in I^{\mathfrak{m}}$.
	\end{enumerate}
	\item
	\begin{enumerate}[(a)]
	\item $n(\mathfrak{n}) = n+1$;
	\item $f^{\mathfrak{n}}_{\bar{g}^\frown(g)} = f_{*}$,  $f^{\mathfrak{n}}_{(\bar{g}^{-1})^\frown(g^{-1})} = f^{-1}_{*}$;
	\item $Y_{\mathfrak{n}} = Z \cup Z^+$, where (noticing $f_{*}[Y_\mathfrak{m}] = \mrm{ran}(f_{*})$):
	\begin{enumerate}[($\cdot_1$)]
	\item $Z = Y_{\mathfrak{m}} \cup f_{*}[Y_{\mathfrak{m}}]$;
	\item $Z^+ = \{\mrm{min}(X'_s \setminus Y_{\mathfrak{m}}) : s \subseteq_1 s^+ \} \setminus Z$, recalling $(*)_2$(c).
	\end{enumerate}
	\end{enumerate}
\end{enumerate}
	\end{enumerate}
The reason for $Z^+$ in (B)(c) above it to satisfy condition $(*)_2$(c).
	\begin{enumerate}[$(*)_{2.3.1}$]
	\item $R^{\mathfrak{n}}_k$ and $E^{\mathfrak{n}}_k$ are defined from the information in $(*)_{2.3}$, as in \ref{hyp_C1+}(\ref{the_graph_Rn+}).
\end{enumerate}

\smallskip
\noindent Comparing $(\mrm{seq}_k(X), R^\mathfrak{n}_k)$ and $(\mrm{seq}_k(X), R^\mathfrak{m}_k)$ the set of new edges is:
$$\{(\bar{x}, \bar{y}) : (\bar{x}, \bar{y}) \in Z^k_1 \cup Z^k_{-1} \},$$
\noindent where we let:
\begin{enumerate}[$(*)_{2.4}$]
	\item
	$$Z^k_1 = \{(\bar{x}, \bar{y}) : \bar{x} \in \mrm{seq}_k(\mrm{dom}(f_*)), f_*(\bar{x}) = \bar{y}, \bar{x} \notin \mrm{seq}_k(\mrm{dom}(f_{\bar{g}}))\},$$
	$$Z^k_{-1} = \{(\bar{x}, \bar{y}) : (\bar{y}, \bar{x}) \in Z^k_1 \},$$
\end{enumerate}
Notice that possibly $\bar{x} \subseteq \mrm{dom}(f_*) \wedge \bar{x} \notin \mrm{seq}_k(\mathfrak{m})$, and possibly $\bar{x} \subseteq \mrm{dom}(f_*) \wedge \bar{x} \not\subseteq \mrm{dom}(f^{\mathfrak{m}}_{\bar{g}}) \wedge \bar{x} \in \mrm{seq}_k(\mathfrak{m})$ (as witnessed by some $\bar{g}' \in I^{\mathfrak{m}}_{< n}$), anyhow the union $Z^k_1 \cup Z^k_{-1}$ is disjoint, as $\mrm{dom}(f_*) = u_*$, by $(*)_{2.2}$(a), 
$u_* \subseteq Y_{\mathfrak{m}}$, by $(*)_{2.1}$, and $x \in \mrm{dom}(f_*) \setminus \mrm{dom}(f_{\bar{g}})$ implies $f_*(x) \notin Y_{\mathfrak{m}}$, by $(*)_{2.2}$(d). Notice now that:
\begin{enumerate}[$(*)_{2.4.1}$]
	\item if $\bar{x} \in \mrm{seq}_k(u_*)$ and $\bar{y} = f_*(\bar{x})$, then:
	$$\bar{x} \subseteq \mrm{dom}(f_{\bar{g}}) \Leftrightarrow \bar{y} \subseteq \mrm{ran}(f_{\bar{g}}) \Rightarrow (\bar{x} \in \mrm{seq}_k(\mathfrak{m}) \wedge \bar{y} \in \mrm{seq}_k(\mathfrak{m})).$$
\end{enumerate}
Now, we have:
	\begin{enumerate}[$(*)_{2.5}$]
	\item
	\begin{enumerate}[(a)]
	\item if $(\bar{x}, \bar{y}) \in Z^k_1$, then:
	\begin{enumerate}[$(\alpha)$]
	\item $\bar{x} \in \mrm{seq}_k(u_*)$ and $\bar{x} \not\subseteq \mrm{dom}(f_{\bar{g}})$;
	\end{enumerate}
	\begin{enumerate}[$(\beta)$]
	\item $\bar{y} \subseteq f_*(u_*)$, $\bar{y} \not\subseteq Y_{\mathfrak{m}}$, $\bar{y} \not\subseteq \mrm{ran}(f_{\bar{g}})$ and $\bar{y} \cap Y_{\mathfrak{m}} \subseteq \mrm{ran}(f_{\bar{g}})$;
	\end{enumerate}
	\item the dual of item (a) for $(\bar{x}, \bar{y}) \in Z^k_{-1}$;
	\item if $\bar{z} \in \mrm{seq}_k(\mathfrak{n}) \setminus \mrm{seq}_k(\mathfrak{m})$, then $\bar{z}$ occurs in exactly one edge of $R^{\mathfrak{n}}_k$.
	%\item for $\bar{y}$ as in (c), $\bar{y}$ occurs in one edge of $R^{\mathfrak{n}}_k$ iff $\bar{y} \subseteq \mrm{dom}(f_*)$ or $\bar{y} \in \mrm{ran}(f_*)$. 
	\end{enumerate}
\end{enumerate}
[Why? Item (a)$(\beta)$ is by $(*)_{2.2}(d)$. Item (c) is by $(*)_{2.2}(c)$.]
\newline Notice now that:
	\begin{enumerate}[$(*)_{2.6}$]
	\item in the graph $(\mrm{seq}_k(X), R^{\mathfrak{n}}_k)$ we have (where $\bar{x} \in \mrm{seq}_k(X)$ below):
	\begin{enumerate}[(i)]
	\item all the new edges have at least one node in $\mrm{seq}_k(u_*) \setminus \mrm{seq}_k(\mrm{dom}(f_{\bar{g}}))$ and one in $\mrm{seq}_k(f_*[u_*]) \setminus \mrm{seq}_k(\mrm{ran}(f_{\bar{g}})) = \mrm{seq}_k(f_*[u_*]) \setminus \mrm{seq}_k(Y_{\mathfrak{m}})$;
	\item every node in $\mrm{seq}_k(\mathfrak{n}) \setminus \mrm{seq}_k(Y_\mathfrak{m})$ has valency $1$;
	\item if $\bar{x} \not\subseteq Y_{\mathfrak{m}}$ and $\bar{x} \not\subseteq \mrm{ran}(f_{*})$, then $\bar{x}/E^{\mathfrak{n}}_k = \{ \bar{x} \}$;
	\item if $\bar{x} \subseteq Y_{\mathfrak{m}}$ and $\bar{x} \not\subseteq \mrm{dom}(f_{*})$, then $\bar{x}/E^\mathfrak{n}_k = \bar{x}/E^\mathfrak{m}_k$;
	\item if $\bar{x} \subseteq \mrm{dom}(f_{*})$ (hence $\bar{x} \subseteq Y_{\mathfrak{m}}$), then:
	$$\bar{x}/E^{\mathfrak{n}}_k = \bar{x}/E^{\mathfrak{m}}_k \cup \{f_{*}(\bar{y}) : \bar{y} \in \bar{x}/E^{\mathfrak{m}}_k, \bar{y} \subseteq u_*, \bar{y} \not\subseteq \mrm{dom}(f_{\bar{g}}) \};$$
	\item if $\bar{x} \subseteq \mrm{dom}(f_{\bar{g}})$ and $\bar{x}/E^{\mathfrak{m}}_k \cap \mrm{seq}_k(u_*) \subseteq \mrm{seq}_k(\mrm{dom}(f_{\bar{g}}))$, then:
	$$\bar{x}/E^\mathfrak{n}_k = \bar{x}/E^\mathfrak{m}_k = f_*(\bar{x})/E^\mathfrak{m}_k;$$
	\item if $\bar{x} \not\subseteq Y_{\mathfrak{m}}$ but $\bar{x} \subseteq f_*(u_*)$, then $\bar{x}/E^\mathfrak{n}_k = f_*^{-1}(\bar{x})/E^\mathfrak{n}_k$;
        \item if $\bar{x} \in \mrm{seq}_k(Y_{\mathfrak{m}})$, then:
        $$(\bar{x}/E^\mathfrak{n}_k) \cap \mrm{seq}_k(Y_{\mathfrak{m}}) = (\bar{x}/E^\mathfrak{m}_k) \cap \mrm{seq}_k(Y_{\mathfrak{m}}).$$
	\end{enumerate}
\end{enumerate}
Notice also that:
\begin{enumerate}[$(*)_{2.6.1}$]
	\item
	\begin{enumerate}[(a)]
	\item if $\bar{x}_0, ..., \bar{x}_m$ is a path in $(\mrm{seq}_k(\mathfrak{n}), R^{\mathfrak{n}}_k)$ with no repetitions and $0 < \ell < m$, then $\bar{x}_\ell \in \mrm{seq}_k(\mathfrak{m})$;
	\item $E^\mathfrak{n}_k \restriction \mrm{seq}_k(\mathfrak{m}) = E^\mathfrak{m}_k \restriction \mrm{seq}_k(\mathfrak{m})$ and $E^\mathfrak{n}_k \restriction \mrm{seq}_k(Y_\mathfrak{m}) = E^\mathfrak{m}_k \restriction \mrm{seq}_k(Y_\mathfrak{m})$;
	%\item $E^{\mathfrak{n}}_k \restriction \mrm{seq}_k(f_{*}[Y_{\mathfrak{m}}]) = E^{\mathfrak{m}(2)}_k \restriction \mrm{seq}_k(f_{*}[Y_{\mathfrak{m}}])$.
	\end{enumerate}
\end{enumerate}
%We introduce an auxiliary definition which will we used in the proof of $(*)_{2.7}$ below.
%\begin{enumerate}[$(*)_{2.7.1}$]
%	\item Let $\mathfrak{m}(1)$ be $\mathfrak{m}$ and $\mathfrak{m}(2)$ be the $f_*$-image of $\mathfrak{m}(1)$, that is:
%	\begin{enumerate}[(a)]
%	\item $Y_{\mathfrak{m}(2)} = f_*[Y_{\mathfrak{m}(1)}]$;
%	\todog{Used in $(*_{2.7.4})$ below.}
%	\item $I^{\mathfrak{m}(2)} = \{g \bar{h} g^{-1} : \bar{h} \in I_{\mathfrak{m}(1)}\}$, and similarly $I^{\mathfrak{m}(2)}_n$;
%	\todog{Seems not used.}
%	\item $f^{\mathfrak{m}(2)}_{g \bar{h} g^{-1}} = f_* f^{\mathfrak{m}(1)}_{\bar{h}} f^{-1}_*$.
%	\end{enumerate}
%\end{enumerate}
Now, we claim:
\begin{enumerate}[$(*)_{2.7}$]
	\item $\mathfrak{n} \in \mathrm{K}^{\mrm{bo}}_0(M)$ and $\mathfrak{n} \in \mathrm{suc}(\mathfrak{m})$.
\end{enumerate}
The only non-trivial thing is to verify that $\mathfrak{n}$ satisfies \ref{hyp_C1+}(\ref{item17_new+}). In principle, verifying that this holds should be straightforward. As $\mathfrak{n}$ is explicitly defined in an essentially free manner, we should be able to check the algebraic condition \ref{hyp_C1+}(\ref{item17_new+}). In actuality, though, verifying \ref{hyp_C1+}(\ref{item17_new+}) would require an explicit description of $\mathcal{A}^{\mathfrak{n}}_{\mathfrak{s}}$. We circumvent this by defining explicitly an $\mathcal{A}'$ such that $\mathcal{A}^{\mathfrak{n}}_{\mathfrak{s}} \subseteq \mathcal{A}'$ (cf. $(*)_{2.7.5}$) and such that $\mathcal{A}'$ satisfies the crucial condition that each $\bar{a} \in \mathcal{A}'$ has non singleton $p$-support (cf. $(*)_{2.7.6}$). Notice that in order to show that $\mathcal{A}^{\mathfrak{n}}_{\mathfrak{s}} \subseteq \mathcal{A}'$ it suffices to show that $\mathcal{A}'$ satisfies the minimal set of condition defining $\mathcal{A}^{\mathfrak{n}}_{\mathfrak{s}}$, as defined in \ref{hyp_C1+}(\ref{item17_new+}), and so it is not hard to achieve, although the proof requires careful checking. Also the proof $(*)_{2.7.6}$ is in principle not hard but it \mbox{involves a careful checking of many cases.}

\smallskip
\noindent We thus move to the proof of \ref{hyp_C1+}(\ref{item17_new+}). To this extent:
\begin{enumerate}[$(*)_{2.7.0}$]
	\item let $\mathfrak{s} = (p, k, \bar{x}, \bar{q})$ be as in \ref{hyp_C1+}(\ref{item17_new+}).
\end{enumerate}
Now, if $\bar{x} \notin \mrm{seq}_k(Y_{\mathfrak{m}})$ and  $\bar{x} \notin \mrm{seq}_k(\mrm{ran}(f_{*}))$, then $\bar{x}/E^\mathfrak{n}_k$ is a singleton and so the proof is as in $(*)_1$. Thus, from now on we assume:
\begin{enumerate}[$(*)_{2.7.1}$]
	\item W.l.o.g. $\bar{x} \in \mrm{seq}_k(Y_{\mathfrak{m}})$ or $\bar{x} \in \mrm{seq}_k(\mrm{ran}(f_{*}))$.
\end{enumerate}
\begin{enumerate}[$(*)_{2.7.2}$]
	\item
	\begin{enumerate}[(a)]
	\item W.l.o.g. $\bar{x} \in \mrm{seq}_k(Y_{\mathfrak{m}})$;
	\item let $\mathfrak{s}$ be is as in \ref{hyp_C1+}(\ref{item17_new+}) for $\mathfrak{m}$ and $\bar{x}$;
	\item so $\mathcal{A}^{\mathfrak{m}}_{\mathfrak{s}}$ is well-defined.
\end{enumerate}	
\end{enumerate}
[Why (a)? If $\bar{x} \not\subseteq Y_{\mathfrak{m}}$, then, by $(*)_{2.7.1}$, necessarily $\bar{x} \subseteq \mrm{ran}(f_{*})$, so $f^{-1}_{*}(\bar{x}) \in \bar{x}/E^{\mathfrak{n}}_k$ and $f^{-1}_{*}(\bar{x}) \subseteq Y_{\mathfrak{m}}$ and, by $(*)_{2.6}$(vii), we can replace $\bar{x}$ by \mbox{$f^{-1}_{*}(\bar{x})$. (b), (c) are clear.]}
\begin{enumerate}[$(*)_{2.7.3}$]
	\item
	\begin{enumerate}[(a)]
	\item $\mathcal{A}^{\mathfrak{m}}_{\mathfrak{s}} \subseteq \mathcal{A}^{\mathfrak{n}}_{\mathfrak{s}}$, let $\mathcal{A}^1_{\mathfrak{s}} = \mathcal{A}^{\mathfrak{m}}_{\mathfrak{s}}$, recalling \ref{hyp_C1+}(\ref{item17_new+});
	\item let $\mathcal{A}^2_{\mathfrak{s}} = \{\bar{b}^{[f^{-1}_*]} : \bar{b} \in \mathcal{A}^1_{\mathfrak{s}} \text{ and } \mrm{dom}(\bar{b}) \subseteq \mrm{dom}(f_*) \}$, where for $\bar{b} = (b_y : y \in Z_1)$ with $Z_1 \subseteq \mrm{dom}(f_*)$ and $Z_2 = f_*[Z_1]$, where we let:
$$\bar{b}^{[f^{-1}_*]} = (b_{f^{-1}_*(y)} : y \in Z_2);$$
\item $\mathcal{A}^2_{\mathfrak{s}} \subseteq \{\bar{b} \in \mathcal{A}^2_{\mathfrak{s}} : \mrm{dom}(\bar{b}) \subseteq \mrm{ran}(f_*) \}$
\item recalling \ref{hyp_C1+}(\ref{item17_new+})(f), notice that for any function $h$ such that $\bar{b}^{[h]}$ is well-defined we have that if $\bar{b}^{[h]} = \bar{c}$, then the following happens:
$$\mrm{dom}(\bar{b}) \subseteq \mrm{ran}(h) \text{ and } \mrm{dom}(\bar{c}) \subseteq \mrm{dom}(h);$$
%\todov{For the benefit of the reader we should write (well, if you agree):
%\newline - Notice that $\bar{a} \in \mathcal{A}^2_{\mathfrak{s}}$ implies that the domain of $\bar{a}$ is a subset of $\mrm{ran}(f_*)$.}
\end{enumerate}
\end{enumerate}
\begin{enumerate}[$(*)_{2.7.4}$]
	\item Let $\mathcal{A}'$ be the set of $\bar{a}$ such that for some $\bar{a}_1 \in \mathcal{A}^1_{\mathfrak{s}}$ and $\bar{a}_2 \in \mathcal{A}^2_{\mathfrak{s}}$ and $u$ such that $\mrm{supp}_p(\bar{a}_1 + \bar{a}_2) \subseteq u \subseteq \mrm{dom}(\bar{a}_1) \cup \mrm{dom}(\bar{a}_2)$ we have that $(\bar{a}_1 + \bar{a}_2) \restriction u = \bar{a}$. In this case we call $(\bar{a}_1, \bar{a}_2, u)$ a witness for $\bar{a}$.
\end{enumerate}
Now we crucially claim:
\begin{enumerate}[$(*)_{2.7.5}$]
	\item $\mathcal{A}^{\mathfrak{n}}_{\mathfrak{s}} \subseteq \mathcal{A}'$.
\end{enumerate}
Why $(*)_{2.7.5}$? Obviously $\mathcal{A}'$ satisfies \ref{hyp_C1+}(\ref{item17_new+})(a) and \ref{hyp_C1+}(\ref{item17_new+})(b) is a definition.
By \ref{hyp_C1+}(\ref{item17_new+})(g) it suffices to prove that $\mathcal{A}'$ satisfies (c)-(f) from \ref{hyp_C1+}(\ref{item17_new+}).
\begin{enumerate}[$(*)_{2.7.5.1}$]
	\item $\mathcal{A}'$ satisfies Clause \ref{hyp_C1+}(\ref{item17_new+})(c).
\end{enumerate}
Let $\mathbf{y} = (\bar{y}^i : i < i_*) \in (\bar{x}/E^{\mathfrak{n}}_k)^{i_*}$, $\bar{r} \in \mathbb{Q}^{\mathbf{y}}$ and $\bar{a} = \bar{a}_{(\mathbf{y}, \bar{r})}$ be as there. Recall that abusing notation we treat $\mathbf{y}$ as a set. Let:
$$\mathbf{y}_1 = \{\bar{y}^i : i < i_*, \bar{y}^i \subseteq Y_{\mathfrak{m}}\}$$ 
$$\mathbf{y}_2 = \{\bar{y}^i : i < i_*, \bar{y}^i \not\subseteq Y_{\mathfrak{m}} \text{ (so } \bar{y}^i \subseteq \mrm{ran}(f_{*}))\}.$$
Easily we have that $\mathbf{y}$ is the disjoint union of $\mathbf{y}_1$ and $\mathbf{y}_2$ and we have:
$$\bar{a}_{(\mathbf{y}, \bar{r})} = \bar{a}_{(\mathbf{y}_1, \bar{r} \restriction \mathbf{y}_1)} + \bar{a}_{(\mathbf{y}_2, \bar{r} \restriction \mathbf{y}_2)},$$
provided that we show that $\bar{a}_2 = \bar{a}_{(\mathbf{y}_2, \bar{r} \restriction \mathbf{y}_2)} \in \mathcal{A}^2_{\mathfrak{s}}$ (as $\bar{a}_1 = \bar{a}_{(\mathbf{y}_1, \bar{r} \restriction \mathbf{y}_1)} \in \mathcal{A}^1_{\mathfrak{s}}$ is obvious by $(*)_{2.6}$(viii)). We do this. Let $\mathbf{y}'_2 = \{f^{-1}_*(\bar{y}) : \bar{y} \in \mathbf{y}_2\}$. Now, if $\bar{y} \in \mathbf{y}_2$, then $f^{-1}_*(\bar{y}) = f^{-1}_{\bar{g}^\frown(g)}(\bar{y}) \in \bar{x}/E^{\mathfrak{n}}_k \cap \mrm{seq}_k(Y_{\mathfrak{m}}) \subseteq \bar{x}/E^{\mathfrak{m}}_k$. Why? First of all $f_*^{-1}(\bar{y}) = f^{-1}_{\bar{g}^\frown(g)}(\bar{y})$, by the choice of $\bar{g}^\frown(g)$. Secondly, $f^{-1}_{\bar{g}^\frown(g)} (\bar{y}) \in \bar{x}/E^{\mathfrak{n}}_k$ as $\bar{y} \subseteq f_*[u_*]$, $\bar{y} \notin Y_{\mathfrak{m}}$ and $\bar{y}/E^{\mathfrak{n}}_k = f^{-1}_*(\bar{y})/E^{\mathfrak{n}}_k$, by $(*)_{2.6}$(vii). Thirdly, $f^{-1}_{\bar{g}^\frown(g)} (\bar{y}) \in Y_{\mathfrak{m}}$, by the choice of $f_{\bar{g}^\frown(g)}$. Thus, $f^{-1}_{\bar{g}^\frown(g)} (\bar{y}) \in \bar{x}/E^{\mathfrak{n}}_k \cap Y_{\mathfrak{m}}$, and, by $(*)_{2.6}$(viii) we have that $\bar{x}/E^{\mathfrak{n}}_k \cap Y_{\mathfrak{m}} \subseteq \bar{x}/E^{\mathfrak{m}}_k$. Let now $\bar{r}'_2 = (r'_{(2, \bar{y})} : \bar{y} \in \mathbf{y}'_2)$, where $r'_{(2, \bar{y})} = r_{(2, f_*(\bar{y}))}$. Also, let $\bar{a}'_2 = (a'_{(2, y)} : y \in \mrm{set}(\mathbf{y}'_2))$, where for $y \in \mrm{set}(\mathbf{y}'_2)$, we let:
$$a'_{(2, y)} = \sum \{r'_{(2, \bar{y})} q_\ell : \bar{y} \in \mathbf{y}'_2 \text{ and } y_\ell = y\}.$$ 
As $\mathbf{y}'_2 \subseteq \bar{x}/E^{\mathfrak{m}}_k$ and $\mathfrak{m}$ satisfies \ref{hyp_C1+}(\ref{item17_new+})(c), easily $\bar{a}'_2 \in \mathcal{A}^{\mathfrak{m}}_{\mathfrak{s}} = \mathcal{A}^{1}_{\mathfrak{s}}$. Also, easily $y \in \mrm{set}(\mathbf{y}'_2)$ implies $a'_{(2, y)} = a_{(2, f_*(y))}$ (recall that $r'_{(2, \bar{y})} = r_{(2, f_*(\bar{y}))}$) and so $(\bar{a}'_2)^{[f^{-1}_*]} = \bar{a}_2$. 
Thus, $\bar{a}_2 \in \mathcal{A}^{2}_{\mathfrak{s}}$. Now, $\mathbf{y}'_2, \bar{r}'_2$ witness that $\bar{a}'_2 \in \mathcal{A}^1_{\mathfrak{s}}$ and so by the definition of $\mathcal{A}^2_{\mathfrak{s}}$ we are done. This concludes the proof of $(*)_{2.7.5.1}$.
\begin{enumerate}[$(*)_{2.7.5.2}$]
	\item $\mathcal{A}'$ satisfies Clause \ref{hyp_C1+}(\ref{item17_new+})(d).
\end{enumerate}
This is obvious by the definition of $\mathcal{A}'$.
\begin{enumerate}[$(*)_{2.7.5.3}$]
	\item $\mathcal{A}'$ satisfies Clause \ref{hyp_C1+}(\ref{item17_new+})(e).
\end{enumerate}
Let $\bar{a}, \bar{b} \in \mathcal{A}'$ and let $(\bar{a}_1, \bar{a}_2, u)$ be a witness for $\bar{a}$ and $(\bar{b}_1, \bar{b}_2, v)$ be a witness for $\bar{b}$, now $(\bar{a}_1 + \bar{a}_2, \bar{b}_1 + \bar{b}_2, u \cup v)$ is a witness for $\bar{a} + \bar{b}$. Hence, $\bar{c} = \bar{a} + \bar{b} \in \mathcal{A}'$.
\begin{enumerate}[$(*)_{2.7.5.4}$]
	\item $\mathcal{A}'$ satisfies Clause \ref{hyp_C1+}(\ref{item17_new+})(f).
\end{enumerate}
Let $\bar{h} \in I^{\mathfrak{n}}$, $Z_1 \subseteq \mrm{dom}(f_{\bar{h}})$, $Z_2 = f_{\bar{h}}[Z_1]$ and $\mrm{dom}(\bar{a}) \subseteq Z_2$. We shall prove that $\bar{a}^{[f_{\bar{h}}]} \in \mathcal{A}'$, where $\bar{a} \in \mathcal{A}'$ and $(\bar{a}_1, \bar{a}_2, u)$ is a witness of this.
\newline \underline{Case 1}. $u \not\subseteq Y_{\mathfrak{m}}$ and $u \not\subseteq \mrm{ran}(f_*)$.
\newline In this case there is no such $\bar{h}$.
\newline \underline{Case 2}. $u \not\subseteq Y_{\mathfrak{m}}$ and $u \subseteq \mrm{ran}(f_*)$.
\newline Notice that $u \not \subseteq Y_{\mathfrak{m}}$, so there is $y \in u \setminus Y_{\mathfrak{m}}$. 
Now, $y \in u \subseteq \mrm{dom}(f_{*})$. But we have:
$$\bar{h} \in I^{\mathfrak{m}} \; \Rightarrow \; \mrm{dom}(f_{\bar{h}}) \subseteq Y_{\mathfrak{m}} \; \Rightarrow \; y \notin \mrm{dom}(f_{\bar{h}}),$$
$$\bar{h} = \bar{g}^\frown(g) \; \Rightarrow \; \mrm{dom}(f_{\bar{h}}) = u_* \Rightarrow \; y \notin \mrm{dom}(f_{\bar{h}}),$$
so necessarily $\bar{h} = (\bar{g}^{-1})^\frown(g^{-1})$ and $f_{\bar{h}} = f^{-1}_*$. Now:
\begin{enumerate}[$(\cdot)$]
	\item W.l.o.g. $\mrm{dom}(\bar{a}_1) \subseteq \mrm{ran}(f_{\bar{g}})$.
\end{enumerate}
[Why? If $z \in \mrm{dom}(\bar{a}_1) \setminus \mrm{ran}(f_{\bar{g}})$, then $z \notin u$ so $z \notin \mrm{supp}_p(\bar{a})$ and $z \notin \mrm{dom}(\bar{a}_2)$, hence $a_z = a_{(1, z)} \in \mathbb{Q}_p$. Thus, $\bar{a}^*_1 = \bar{a}_1 \restriction (\mrm{dom}(\bar{a}_1) \cap \mrm{ran}(f_{\bar{g}})) \in \mathcal{A}^1_{\mathfrak{s}}$ and $(\bar{a}_1 + \bar{a}_2) \restriction u = (\bar{a}^*_1 + \bar{a}_2) \restriction u$, so we can replace $\bar{a}_1$  by $\bar{a}^*_1$, as $\mathfrak{m}$ satisfies clause (f).]
\newline Let $\bar{a}'_1 = \bar{a}_1^{[f_{\bar{g}}]} = \bar{a}_1^{[f_{*}]}$, this is well-defined, it belongs to $\mathcal{A}^1_{\mathfrak{s}}$ and has domain $\subseteq \mrm{dom}(f_*)$. Also, $\mrm{dom}(\bar{a}_2) \subseteq \mrm{ran}(f_*)$ and $\bar{a}_2 \in \mathcal{A}^2_{\mathfrak{s}}$, hence $\bar{a}'_2 = a_2^{[f_{*}]} \in \mathcal{A}^1_{\mathfrak{s}}$ and has domain $\subseteq \mrm{dom}(f_*)$. 
By \ref{hyp_C1+}(\ref{item17_new+})(e) and the above we have that $\bar{a}' = \bar{a}'_1 + \bar{a}'_2 \in \mathcal{A}^1_{\mathfrak{s}}$. Also, $\mrm{supp}_p(\bar{a}') \subseteq f^{-1}_*[u] \subseteq \mrm{dom}(\bar{a}'_1 + \bar{a}'_2)$, hence $\bar{a}' \restriction f^{-1}_*[u] \in \mathcal{A}^1_{\mathfrak{s}}$.  Hence:
%$$\bar{a}^{[f_{\bar{h}}]} = \bar{a}^{[f^{-1}_{*}]} = (\bar{a} \restriction u)^{[f^{-1}_{*}]} = \bar{a}^{[f^{-1}_{*}]} \restriction f^{-1}_{*}[u] \in \mathcal{A}_{\mathfrak{s}(1)}.$$
%But clearly $(\bar{a}')^{[f_*]} = (\bar{a}'_1)^{[f_*]} + (\bar{a}'_2)^{[f_*]} = \bar{a}_1 + \bar{a}_2$, and so: 
$$\begin{array}{rcl}
\bar{a}^{[f_{\bar{h}}]} & = & \bar{a}^{[f_*]} \\
 & = & ((\bar{a}_1 + \bar{a}_2) \restriction u)^{[f^{-1}_*]} \\
 & = & (\bar{a}_1 + \bar{a}_2)^{[f_*]} \restriction {f^{-1}_*[u]} \\
 & = & (\bar{a}^{[f_*]}_1 + \bar{a}^{[f_*]}_2)\restriction {f^{-1}_*[u]} \\
 & = & (\bar{a}'_1 + \bar{a}'_2) \restriction {f^{-1}_*[u]} \\
  & = & \bar{a}' \restriction {f^{-1}_*[u]} \in \mathcal{A}^1_{\mathfrak{s}}.\\
\end{array}$$
\newline \underline{Case 3}. $u \subseteq Y_{\mathfrak{m}}$ and $\bar{h} = f^{\mathfrak{n}}_{\bar{g}^\frown(g)} = f_*$.
\newline In this case we have:
\begin{enumerate}[$(\cdot)$]
	\item W.l.o.g. $\mrm{dom}(\bar{a}_2) \subseteq \mrm{ran}(f_{\bar{g}})$.
\end{enumerate}
[Why? If $y \in \mrm{dom}(\bar{a}_2) \setminus \mrm{ran}(f_{\bar{g}})$, then (recalling $\mrm{dom}(\bar{a}_2) \setminus \mrm{ran}(f_{\bar{g}}) \subseteq f_*[u_*] \setminus \mrm{ran}(f_{\bar{g}}) \subseteq f_*[u_*] \setminus u$) we have that $y \notin u$ so $y \notin \mrm{supp}_p(\bar{a})$ and $y \notin \mrm{dom}(\bar{a}_1)$, hence $a_y = a_{(2, y)} \in \mathbb{Q}_p$. Thus, by \ref{hyp_C1+}(\ref{item17_new+})(d), $\bar{a}^*_2 = \bar{a}_2 \restriction (\mrm{dom}(\bar{a}_2) \cap \mrm{ran}(f_{\bar{g}})) \in \mathcal{A}^2_{\mathfrak{s}}$ and $(\bar{a}_1 + \bar{a}_2) \restriction u = (\bar{a}_1 + \bar{a}^*_2) \restriction u$, so we can replace $\bar{a}_2$  by $\bar{a}^*_2$, as $\mathfrak{m}$ satisfies clause (f).]
\newline Let $\bar{a}'_2 = \bar{a}_2^{[f_{\bar{g}}]} = \bar{a}_2^{[f_{*}]}$, this is well-defined, it belongs to $\mathcal{A}^1_{\mathfrak{s}}$ (by the definition of $\mathcal{A}^2_{\mathfrak{s}}$, recalling $\bar{a}_2 \in \mathcal{A}^2_{\mathfrak{s}}$). Also, $\bar{a}'_2$ has domain $\subseteq \mrm{dom}(f_*)$. Now, as $\mathfrak{m} \in \mrm{K}^{\mrm{bo}}_0(M)$, $f_{\bar{g}} \in I^{\mathfrak{m}}$ and $\bar{a}'_2 \in \mathcal{A}^1_{\mathfrak{s}} = \mathcal{A}^{\mathfrak{m}}_{\mathfrak{s}}$, recalling \ref{hyp_C1+}(\ref{item17_new+})(e), we have $\bar{a}_2 = (\bar{a}'_2)^{[f^{-1}_{\bar{g}}]} \in \mathcal{A}^1_{\mathfrak{s}}$, so as $\mathfrak{m} \in \mrm{K}^{\mrm{bo}}_0(M)$ we have $\bar{a}_1 + \bar{a}_2 \in \mathcal{A}^1_{\mathfrak{s}}$. Thus, as $\mathcal{A}^1_{\mathfrak{s}} \subseteq \mathcal{A}'$, we are done.
\newline \underline{Case 4}. $u \subseteq Y_{\mathfrak{m}}$ and $\bar{h} \in I_{\mathfrak{m}}$.
\newline Similar to Case 3.
\newline \underline{Case 5}. $u \subseteq Y_{\mathfrak{m}}$ and $\bar{h} = f^{\mathfrak{n}}_{(\bar{g}^{-1})^\frown(g^{-1})} = f^{-1}_*$.
\newline As $u \subseteq Y_{\mathfrak{m}}$ and $u \subseteq \mrm{dom}(f^{\mathfrak{n}}_{\bar{h}}) = \mrm{dom}(f^{-1}_*) = \mrm{ran}(f_*)$, necessarily we have:
$$u  \subseteq Y_{\mathfrak{m}} \cap \mrm{ran}(f_*) = \mrm{ran}(f_{\bar{g}}) \;\; \text{(cf. $(*)_{2.2}$(c))}.$$
As in earlier cases:
\begin{enumerate}[$(\cdot)$]
	\item W.l.o.g. $\mrm{dom}(\bar{a}_2) \subseteq \mrm{ran}(f_{\bar{g}})$.
\end{enumerate}
So we can finish as in Case 4.

\smallskip
\noindent Thus, indeed $\mathcal{A}^{\mathfrak{n}}_{\mathfrak{s}} \subseteq \mathcal{A}'$, by (g) of the definition of $\mathcal{A}^{\mathfrak{n}}_{\mathfrak{s}}$ in \ref{hyp_C1+}(\ref{item17_new+}) and $(*)_{2.7.5.1}$-$(*)_{2.7.5.4}$. Hence, we finished proving $(*)_{2.7.5}$.
\begin{enumerate}[$(*)_{2.7.6}$]
	\item If $\bar{a} \in \mathcal{A}'$, then $\mrm{supp}_p(\bar{a})$ is not a singleton.
\end{enumerate}
We prove $(*)_{2.7.6}$. Let $\bar{a} \in \mathcal{A}'$ and let $(\bar{a}_1, \bar{a}_2, u)$ be a witness of this. For $\ell = 1, 2$, let $\bar{a}'_\ell = \bar{a}_\ell \restriction \mrm{supp}_p(\bar{a}_\ell)$, then, by \ref{hyp_C1+}(\ref{item17_new+})(d), we have that $\bar{a}'_1 \in \mathcal{A}^1_{\mathfrak{s}} = \mathcal{A}^{\mathfrak{m}}_{\mathfrak{s}}$. 
Also, $\bar{a}^*_2 = \bar{a}^{[f_*]}_2 \in \mathcal{A}^1_{\mathfrak{s}}$,
by the definition of $\mathcal{A}^2_{\mathfrak{s}}$, and so, as $\mathfrak{m} \in \mrm{K}^{\mrm{bo}}_0(M)$, $\bar{a}^*_2 \restriction \{y \in X : f_*(y) \in \mrm{dom}(\bar{a}'_2)\} \in \mathcal{A}^1_{\mathfrak{s}}$. Clearly we have the following:
	$$$$
	$$\begin{array}{rcl}
\mrm{dom}_p(a^*_2) & := & \{y \in \mrm{dom}(\bar{a}^*_2) : \bar{a}^*_{(2, y)} \notin \mathbb{Q}_p \} \\
 & = & \{y \in \mrm{dom}(a^*_2) : a_{(2, f_*(y))} = a_{(2, y)} \notin \mathbb{Q}_p \}\\
 & = & \{f_*(y) : y \in \mrm{dom}_p(\bar{a}_2) = \{y \in \mrm{dom}(\bar{a}^*_2) : \bar{a}_{(2, y)} \notin \mathbb{Q}_p\}\}.
\end{array}$$
Thus, $\bar{a}^*_2 \restriction \mrm{dom}_p(\bar{a}^*_2) = (\bar{a}'_2)^{[f_*]}$. As $\bar{a}^*_2 = \bar{a}^{[f_*]}_2 \in \mathcal{A}^1_{\mathfrak{s}}$, recalling that $\mathcal{A}^1_{\mathfrak{s}} = \mathcal{A}^{\mathfrak{m}}_{\mathfrak{s}}$, by \ref{hyp_C1+}(\ref{item17_new+})(d), we have that $(\bar{a}'_2)^{[f_*]} = \bar{a}^*_2 \restriction \mrm{dom}_p(\bar{a}^*_2) \in \mathcal{A}^{\mathfrak{m}}_{\mathfrak{s}}$. Hence, by the definition of $\mathcal{A}^2_{\mathfrak{s}}$ (as $\bar{a} = \bar{b}^{[f_*]}$ iff $\bar{b} = \bar{a}^{[f^{-1}_*]}$), $\bar{a}'_2 \in \mathcal{A}^2_{\mathfrak{s}}$. So we have:
\begin{enumerate}[(a)]
	\item if $y \in \mrm{dom}(\bar{a}_1) \cap \mrm{dom}(\bar{a}_2)$, then:
	\begin{enumerate}[$(\cdot)$]
	\item $y \notin \mrm{supp}_p(\bar{a}_1)$ implies $y \in \mrm{supp}_p(\bar{a}_1 + \bar{a}_2)$ iff $y \in \mrm{supp}_p(\bar{a}_2)$;
	\item $y \notin \mrm{supp}_p(\bar{a}_2)$ implies $y \in \mrm{supp}_p(\bar{a}_1 + \bar{a}_2)$ iff $y \in \mrm{supp}_p(\bar{a}_1)$;	\end{enumerate}
	\item if $y \in \mrm{dom}(\bar{a}_1) \setminus \mrm{dom}(\bar{a}_2)$, then $y \in \mrm{supp}_p(\bar{a}_1)$ iff $y \in \mrm{supp}_p(\bar{a}_1 + \bar{a}_2)$;
	\item if $y \in \mrm{dom}(\bar{a}_2) \setminus \mrm{dom}(\bar{a}_1)$, then $y \in \mrm{supp}_p(\bar{a}_2)$ iff $y \in \mrm{supp}_p(\bar{a}_1 + \bar{a}_2)$.
\end{enumerate}
Hence:
\begin{enumerate}[$(*)_{2.7.6.1}$]
	\item W.l.o.g. $\bar{a} = \bar{a}_1 + \bar{a}_2$ and $\bar{a}_\ell = \bar{a}_\ell \restriction \mrm{supp}_p(\bar{a}_\ell)$, for $\ell = 1, 2$.
\end{enumerate}
[Why? Letting $u' = \mrm{dom}(\bar{a}'_1) \cup \mrm{dom}(\bar{a}'_2)$, we have:
\begin{enumerate}[(a)]
	\item $u' \subseteq u$;
	\item $\mrm{dom}(\bar{a}'_1)$, $\mrm{dom}(\bar{a}'_2) \subseteq u$;
	\item $\bar{a}'_1 + \bar{a}'_2 \restriction \mrm{supp}_p(\bar{a}'_1 + \bar{a}'_2) = \bar{a}_1 + \bar{a}_2 \restriction \mrm{supp}_p(\bar{a}_1 + \bar{a}_2)$.
\end{enumerate}
So $(*)_{2.7.6.1}$ holds indeed.]
\newline With $(*)_{2.7.6.1}$ in mind, we now get back to the proof of $(*)_{2.7.6}$.
\newline \underline{Case A}. $\mrm{supp}_p(\bar{a}_1) \not\subseteq \mrm{ran}(f_{\bar{g}})$ and $\mrm{supp}_p(\bar{a}_2) \not\subseteq \mrm{ran}(f_{\bar{g}})$.\newline As $\mrm{supp}_p(\bar{a}_1) \not\subseteq \mrm{ran}(f_{\bar{g}})$ we can choose $y_1 \in \mrm{supp}_p(\bar{a}_1) \setminus \mrm{ran}(f_{\bar{g}})$, and similarly we can choose $y_2 \in \mrm{supp}_p(\bar{a}_2) \setminus \mrm{ran}(f_{\bar{g}})$. Now $\mrm{dom}(\bar{a}_1) \subseteq Y_{\mathfrak{m}}$ and $\mrm{dom}(\bar{a}_2) \subseteq f_{*}[Y_{\mathfrak{m}}]$, hence $\mrm{dom}(\bar{a}_1) \cap \mrm{dom}(\bar{a}_2) \subseteq Y_{\mathfrak{m}} \cap f_{*}[Y_{\mathfrak{m}}] = \mrm{ran}(f_{\bar{g}})$ (recall $(*)_{2.2}$(c)), so necessarily $y_1 \notin \mrm{dom}(\bar{a}_2)$ and $y_2 \notin \mrm{dom}(\bar{a}_1)$ (by the choice of $y_1$ and $y_2$). Hence, letting $\bar{a} = (a_y : y \in u)$ and recalling the definition of $\bar{a} = \bar{a}_1 + \bar{a}_2$ from \ref{hyp_C1+}(\ref{item17_new_e}) we have:
\begin{enumerate}[$(\cdot)$]
	\item $y_1 \in \mrm{dom}(\bar{a}_1) \setminus \mrm{dom}(\bar{a}_2)$, so $a_{y_1} = a_{(1, y_1)}$;
	\item $y_2 \in \mrm{dom}(\bar{a}_2) \setminus \mrm{dom}(\bar{a}_1)$, so $a_{y_2} = a_{(2, y_2)}$.
\end{enumerate}
But $a_{(1, y_1)}, a_{(2, y_2)} \notin \mathbb{Q}_p$ (as $y_\ell \in \mrm{supp}_p(\bar{a}_\ell)$, for $\ell = 1, 2$) and so $a_{y_1}, a_{y_2} \notin \mathbb{Q}_p$, and, as obviously $y_1 \neq y_2$, we are done. This concludes the proof of Case A.
\newline \underline{Case B}. $\mrm{supp}_p(\bar{a}_2) \subseteq \mrm{ran}(f_{\bar{g}})$, equivalently, by $(*)_{2.7.6.1}$, $\mrm{dom}(\bar{a}_2) \subseteq \mrm{ran}(f_{\bar{g}})$.
\newline Define $\mathbf{y}'_2 = \{f_{*}^{-1}(\bar{y}) : \bar{y} \in \mathbf{y}_2 \}$, where, recalling $\bar{x}$ is from $\mathfrak{s}$ (cf. $(*)_{2.7.0}$), we let:
$$\mathbf{y}_2 = \{\bar{y} \in \bar{x}/E^{\mathfrak{n}}_k : \bar{y} \not\subseteq Y_{\mathfrak{m}} \text{ (so } \bar{y} \subseteq \mrm{ran}(f_{*}))\}.$$ 
Let now:
$$\bar{a}'_2 = (a'_{(2, y)} : y \in \mrm{set}(\mathbf{y}'_2)),$$
where:
	$$y \in \mrm{set}(\mathbf{y}'_2) \; \Rightarrow \; a'_{(2, y)} = a_{(2, f_{*}(y))}.$$
%\begin{enumerate}[$(\cdot_a)$]
%	\item if $y \in \mrm{set}(\mathbf{y}'_2)$, then $a'_{(2, y)} = a_{(2, f_{*}(y))}$.
%\end{enumerate}
Now, we have:
\begin{enumerate}[$(\cdot_1)$]
	\item $\mathbf{y}'_2 \subseteq \bar{x}/E^{\mathfrak{m}}_k$;
	\item $\bar{a}'_2 \in \mathcal{A}^1_{\mathfrak{s}}$;
	\item $(\bar{a}'_2)^{[f_{*}]} = \bar{a}_2$;
	\item $\mrm{dom}(\bar{a}'_2) \subseteq \mrm{dom}(f_{\bar{g}})$;
	\item $\bar{a}'_2 \in \mathcal{A}^1_{\mathfrak{s}}$;
	\item $\bar{a}_2 \in \mathcal{A}^1_{\mathfrak{s}}$.
\end{enumerate}
[Why? Concerning $(\cdot_1)$, if $\bar{y}' \in \mathbf{y}'_2$, then by the choice of $\mathbf{y}'_2$, there is $\bar{y} \in \mathbf{y}_2$ such that $f^{-1}_*(\bar{y}) = \bar{y}'$. Furthermore, by the choice of $\mathbf{y}_2$ we have $\bar{y} \in \bar{x}/E^{\mathfrak{n}}_k$ and $\bar{y} \not\subseteq Y_{\mathfrak{m}} \text{ (so } \bar{y} \subseteq \mrm{ran}(f_{*}))$. By the definition of $E^{\mathfrak{n}}_k$ we have $\bar{y}' \in \bar{x}/E^{\mathfrak{n}}_k$. Thus, by $(*)_{2.6}$(viii), we have  
$\bar{y}' \in \bar{x}/E^{\mathfrak{m}}_k$, so $(\cdot_1)$ holds indeed. Also, $(\cdot_2)$ is by $(\cdot_1)$ and $(\cdot_3)$ is because we defined $\bar{a}'_2 = (a'_{(2, y)} : y \in \mrm{set}(\mathbf{y}'_2))$.
Moving to the remaining clauses, we have that $(\cdot_4)$ holds as $\mrm{supp}(\bar{a}_2) \subseteq \mrm{ran}(f_{\bar{g}})$. $(\cdot_5)$ holds by $(\cdot_3)$ and the fact that $\mathfrak{m} \in \mathrm{K}^{\mrm{bo}}_1(M)$. Finally, concerning $(\cdot_6)$, recalling that $f_{\bar{g}} \subseteq f_*$, by $(\cdot_3)$+$(\cdot_4)$ we have that $(\bar{a}'_2)^{[f_{\bar{g}}]} = \bar{a}_2$, and as $\bar{a}'_2 \in \mathcal{A}^1_{\mathfrak{s}} = \mathcal{A}^{\mathfrak{m}}_{\mathfrak{s}}$, by \ref{hyp_C1+}(\ref{item17_new+})(e) we have $\bar{a}_2 \in \mathcal{A}^1_{\mathfrak{s}} = \mathcal{A}^{\mathfrak{m}}_{\mathfrak{s}}$.]

\smallskip
\noindent
Let now $\bar{a}_* = \bar{a}_1 + \bar{a}_2 $, as each summand is in $\mathcal{A}^1_{\mathfrak{s}}$ (notice that the second summand is in $\mathcal{A}^1_{\mathfrak{s}}$ by $(\cdot_6)$), then also $\bar{a}_* \in \mathcal{A}^1_{\mathfrak{s}}$, recalling that $\mathcal{A}^1_{\mathfrak{s}} = \mathcal{A}^\mathfrak{m}_{\mathfrak{s}}$ and $\mathfrak{m}$ satisfies condition \ref{hyp_C1+}(\ref{item17_new})(e). Also, clearly $\bar{a} = \bar{a}_*$, but the latter belonging to $\mathcal{A}^1_{\mathfrak{s}}$ we have that $\mrm{supp}_p(\bar{a})$ is not a singleton, recalling that $\mathcal{A}^1_{\mathfrak{s}} = \mathcal{A}^\mathfrak{m}_{\mathfrak{s}}$ \mbox{and $\mathfrak{m}$ satisfies \ref{hyp_C1+}(\ref{item17_new+}).}
\newline \underline{Case C}. $\mrm{supp}_p(\bar{a}_1) \subseteq \mrm{ran}(f_{\bar{g}})$, equivalently, by $(*)_{2.7.6.1}$, $\mrm{dom}(\bar{a}_1) \subseteq \mrm{ran}(f_{\bar{g}})$.
\newline This case is similar to Case B.
Recalling $\bar{x}$ is from $\mathfrak{s}$ (cf. $(*)_{2.7.0}$), let:
$$\mathbf{y}_2 = \{\bar{y} \in \bar{x}/E^{\mathfrak{n}}_k : \bar{y} \not\subseteq Y_{\mathfrak{m}} \text{ (so } \bar{y} \subseteq \mrm{ran}(f_{*}))\}.$$ 
$$\mathbf{y}'_2 = \{f_{*}^{-1}(\bar{y}) : \bar{y} \in \mathbf{y}_2 \}.$$
$$\bar{a}'_2 = (a'_{(2, y)} : y \in \mrm{set}(\mathbf{y}'_2)),$$
where:
	$$y \in \mrm{set}(\mathbf{y}'_2) \; \Rightarrow \; a'_{(2, y)} = a_{(2, f_{*}(y))}.$$
Let now $Y_1 = \mrm{supp}_p(\bar{a}_1) \subseteq \mrm{ran}(f_{\bar{g}})$ and $Y'_1 = f^{-1}_{\bar{g}}(Y_1) \subseteq \mrm{dom}(f_{\bar{g}})$, then we let:
\begin{enumerate}[$(\cdot_a)$]
	\item $\bar{a}'_1 = (a'_{(1, y)} : y \in Y'_1)$, where:
	\item $a'_{(1, y)} = a_{(1, f_{\bar{g}}(y))}$;
	\item $Y_2 = \mrm{dom}(\bar{a}_2)$,  $Y'_2 = f_*^{-1}[Y_2] = f^{-1}_{\bar{g}}[Y_2]$;
	\item $\bar{a}'_2 = (a'_{(2, y)} : y \in Y'_2)$, where:
	\item $a'_{(2, y)} = a'_{(2, f_*(y))}$.
\end{enumerate}
Then:
\begin{enumerate}[$(\cdot_1)$]
	\item $\mathbf{y}'_2 \subseteq \bar{x}/E^{\mathfrak{m}}_k$ (recall that $\mathfrak{s} = (p, k, \bar{x}, \bar{q})$ and \ref{hyp_C1+}(\ref{item17_new+})(e));
	\item $\bar{a}'_2 \in \mathcal{A}^1_{\mathfrak{s}}$;
	\item $(\bar{a}'_1)^{[f^{-1}_{*}]} = \bar{a}_1$;
	\item $\mrm{dom}(\bar{a}'_1) \subseteq \mrm{dom}(f_{\bar{g}})$;
	\item $\bar{a}'_1 \in \mathcal{A}^1_{\mathfrak{s}}$;
	\item $(\bar{a}'_2)^{[f^{-}_{*}]} = \bar{a}_2$.
\end{enumerate}
Let now $\bar{a}_* = \bar{a}_1 + \bar{a}_2$ and let $\bar{a}'_* = \bar{a}'_1 + \bar{a}'_2 $, so that $(\bar{a}'_*)^{[f^{-1}_{*}]} = \bar{a}_*$. 
As $\bar{a}'_1 \in \mathcal{A}^1_{\mathfrak{s}}$ by $(\cdot_5)$ and $\bar{a}'_2 \in \mathcal{A}^1_{\mathfrak{s}}$ by $(\cdot_2)$, then by \ref{hyp_C1+}(\ref{item17_new})(e), also $\bar{a}'_* = \bar{a}'_1 + \bar{a}'_2 \in \mathcal{A}^1_{\mathfrak{s}}$, hence $\mrm{supp}_p(\bar{a}'_*)$ is not a singleton (as $\mathfrak{m} \in \mathrm{K}^{\mrm{bo}}_1(M)$) and so also $\mrm{supp}_p(\bar{a}_*)$ is not a singleton. 
%But $f_*$ maps $\bar{a}'_1$ to $\bar{a}_1 \restriction Y_1 = \bar{a}_1 \restriction \mrm{supp}_p(\bar{a}_1)$ and $\bar{a}'_2$ to $\bar{a}_2$, hence it maps $\bar{a}'_2$ to $\bar{a}_2$, and so it maps $\bar{a}'_*$ to $\bar{a}_*$.

\smallskip
\noindent So we finished proving  $(*_{2.7.6})$, i.e., $\bar{a} \in \mathcal{A}'$ implies that $\mrm{supp}_p(\bar{a})$ is not a singleton. Thus, we also finished proving $(*)_2$, as by $(*_{2.7.5})$ we have $\bar{a} \in \mathcal{A}^{\mathfrak{n}}_{\mathfrak{s}} \Rightarrow \bar{a} \in \mathcal{A}'$, and so by $(*_{2.7.6})$ we are done, i.e., we have verified that $\mathfrak{n}$ satisfies \ref{hyp_C1+}(\ref{item17_new+}).
\begin{enumerate}[$(*)_3$]
	\item We can choose an $<_\mrm{suc}$-increasing sequence $(\mathfrak{m}_\ell : \ell < \omega)$ in $\mathrm{K}^{\mrm{bo}}_0(M)$ whose limit $\mathfrak{m}$ is as wanted, i.e. $\mathfrak{m} \in \mathrm{K}^{\mrm{bo}}_2(M)$.
\end{enumerate}
We show this. We can find a list $(\bar{g}^\ell : \ell < \omega)$ of $\bigcup_{m < \omega} \mathcal{G}^m_*$ such that:
\begin{enumerate}[$(*)_{3.1}$]
	\item
	\begin{enumerate}[(i)]
	\item $\mrm{lg}(\bar{g}^\ell) \leq \ell$;
	\item if $\bar{g}^\ell \triangleleft \bar{g}^k$, then $\ell < k$;
	\item $\mrm{lg}(\bar{g}^\ell) = 0$ iff $\ell = 0$;
	\item note that for $\ell < \mrm{lg}(\bar{g})$, $g^k_\ell \neq (g^k_\ell)^{-1}$
	\item $\bar{g}^{2\ell+2} = (\bar{g}^{2\ell+1})^{-1}$;
	\item if $\mrm{lg}(\bar{g}^{2\ell+1}) > 1$, then there is a unique $i < \ell$ such that:
	\begin{enumerate}[$(\cdot_1)$]
	\item $\bar{g}^{2i+1} \triangleleft \bar{g}^{2\ell+2}$;
	\item $\bar{g}^{2i+2} \triangleleft \bar{g}^{2\ell+1}$;
	\item $\mrm{lg}(\bar{g}^{2\ell+1}) = \mrm{lg}(\bar{g}^{2\ell+2}) = \mrm{lg}(\bar{g}^{2i+1})+1 = \mrm{lg}(\bar{g}^{2i+2})+1$.
\end{enumerate}
\end{enumerate}
\end{enumerate}
Why we ask what we ask in $(*)_{3.1}$? Clause (i) is just for clarity. Clause (ii) is needed because defining $\mathfrak{m}_{k+1}$ we would like to ensure $\bar{g}_k \in I^{\mathfrak{m}_{k+1}}$, in the interesting case $\bar{g}_k \notin I^{\mathfrak{m}_{k}}$ but $\bar{g}' \triangleleft \bar{g}_k$ implies $f_{\bar{g}'} \subseteq f_{\bar{g}_k}$, so it makes sense to take care of $\bar{g}_k$ only after all the $\bar{g}' \triangleleft \bar{g}_k$ have been taken care of, but this means $\bar{g}' \triangleleft \bar{g}_k$ implies $\bar{g}' \in I^{\mathfrak{m}_{k}}$. Concerning clause (vi), the point is that in $(*)_2$ we only took care of having $\mrm{dom}(\bar{f}_{\bar{g}^\frown(g)})$ to be large enough, but not of $\mrm{ran}(\bar{f}_{\bar{g}^\frown(g)})$. But, by our bookkeeping, if $() \triangleleft \bar{g}_\ell \triangleleft \bar{g}_k$ and $\mrm{lg}(\bar{g}_\ell)+1 = \mrm{lg}(\bar{g}_k)$, then $k$ is odd if and only if $\ell$ is even. Hence if $k$ is odd, then choosing $\bar{g}_{k}$ will increase the domain of $f^{\mathfrak{m}_{k+1}}_{\bar{g}_k}$ to include $Y_{\mathfrak{m}_k} \cap X_{\mrm{dom}(\bar{g}_k)}$ (cf. $(*_2)$(d)), and if $\ell$ is odd, then choosing $\bar{g}_{\ell}$ will increase the domain of $f^{\mathfrak{m}_{\ell+1}}_{\bar{g}_\ell}$ to include $Y_{\mathfrak{m}_\ell} \cap X_{\mrm{dom}(\bar{g}_\ell)}$, but $Y_{\mathfrak{m}_\ell} \subseteq Y_{\mathfrak{m}_k}$, so always $Y_{\mathfrak{m}_\ell} \cap X_{\mrm{dom}(\bar{g}_\ell)} \subseteq \mrm{dom}(f^{\mathfrak{m}_{k+1}}_{\bar{g}_k})$. Mutatis mutandi we have that $Y_{\mathfrak{m}_\ell} \cap X_{\mrm{ran}(\bar{g}_\ell)} \subseteq \mrm{ran}(f^{\mathfrak{m}_{k+1}}_{\bar{g}_k})$. Clearly this suffices.

\smallskip
\noindent
Now, by induction on $\ell < \omega$, we choose $\mathfrak{m}_\ell \in \mrm{K}^{\mrm{bo}}_0$ such that $n(\mathfrak{m}_\ell) \leq \ell+1$ and $\mathfrak{m}_{\ell+1} \in \mrm{suc}(\mathfrak{m}_\ell)$ or $\mathfrak{m}_{\ell+1} = \mathfrak{m}_{\ell}$. We proceed as follows:
\begin{enumerate}[$(*)_{3.2}$]
	\item 
\begin{enumerate}[($\ell = 0$)]
	\item use $(*)_1$;
\end{enumerate}
\begin{enumerate}[($\ell = k+1$)]
	\item
	\begin{enumerate}[($\cdot_1$)]
	\item if $\bar{g}^{k+1} \in I^{\mathfrak{m}_k}$, then $\mathfrak{m}_\ell = \mathfrak{m}_k$ (if this occurs, then $k$ is odd);
	\item if $\bar{g}^{k+1} \notin I^{\mathfrak{m}_k}$, let $m_k = \mrm{lg}(\bar{g}^{k+1}) - 1$, so $\bar{g}^{k+1} \restriction m_k \in I^{\mathfrak{m}_k}$, and use $(*)_2$ with the pair $n(\mathfrak{m}_k)$, $\bar{g}^{2k+1}$ here standing for $n, \bar{g}^\frown (g)$ there.
	\end{enumerate}
\end{enumerate}
\end{enumerate}
Clearly $\mathfrak{m} = \mathrm{lim}_{\ell < \omega}(\mathfrak{m}_\ell) \in \mathrm{K}^{\mrm{bo}}_1(M)$. 
Notice that by $(*)_{3.1}$ we have:
\begin{enumerate}[$(*)_{3.3}$]
	\item if $\bar{g}^k \triangleleft \bar{g}^\ell \triangleleft \bar{g}^m$, then:
	\begin{enumerate}[(i)]
	\item $f_{\bar{g}^k} \subseteq f_{\bar{g}^\ell} \subseteq f_{\bar{g}^m}$;
	\item $Y_{\mathfrak{m}_k} \cap X_{\mrm{dom}(f_{\bar{g}^k})} \subseteq \mrm{dom}(f_{\bar{g}^m})$;
	\item $Y_{\mathfrak{m}_k} \cap X_{\mrm{ran}(f_{\bar{g}^k})} \subseteq \mrm{ran}(f_{\bar{g}^m})$;
	\item if $s \subseteq_1 \mrm{dom}(f_{\bar{g}^k})$, then $\mrm{min}(X'_s \setminus Y_{\mathfrak{m}_k}) \in \mrm{dom}(f_{\bar{g}^m})$ (see $(*)_{2.3}$(B)(c)$(\cdot_1)$);
	\item if $s \subseteq_1 \mrm{ran}(f_{\bar{g}^k})$, then $\mrm{min}(X'_s \setminus Y_{\mathfrak{m}_k}) \in \mrm{ran}(f_{\bar{g}^m})$ (see $(*)_{2.3}$(B)(c)$(\cdot_2)$).
\end{enumerate}
\end{enumerate}
Thus we are only left to show that $\mathfrak{m} \in \mathrm{K}^{\mrm{bo}}_1(M)$ is full, that this, that $\mathfrak{m}$ satisfies conditions (12), (13) from \ref{def_full_K_hop_bis}. For this notice:
\begin{enumerate}[(i)]
	\item Def.~\ref{def_full_K_hop_bis}(12) holds by the definition of $\mathfrak{m}_{k+1} \in \mrm{suc}_{\mathfrak{m}_k}$, recalling $(*)_{3.3}$(iv)(v);
	\item Def.~\ref{def_full_K_hop_bis}(13) holds as the $\bar{g}^\ell$'s list $\mathcal{G}_*$.
\end{enumerate}
\end{proof}

	\begin{corollary}\label{K2bo_non_empty++} $\mathrm{K}^{\mrm{bo}}_2(M) \neq \emptyset$.
\end{corollary}

	\begin{proof} This is obvious by \ref{K2bo_non_empty} simply comparing Def.~\ref{hyp_C1} and Def.~\ref{hyp_C1+}+\ref{def_full_K_hop_bis}.
\end{proof}

\section{Borel Completeness of Torsion-Free Abelian Groups}\label{Borel_complete_section}

\subsection{The Definition of the Groups $G_{(1, \mathcal{U})}$}

	\begin{definition}\label{item_primes} Let $\mathrm{K}^{\mrm{bo}}_3(M)$ be the class of $\mathfrak{m} \in \mathrm{K}^{\mrm{bo}}_2(M)$ expanded with a sequence $\bar{p} = \bar{p}^\mathfrak{m}$ of prime numbers without repetitions such that we have the following:
	\begin{enumerate}[(1)]
		\item $\bar{p} = (p_{(e, \bar{q})} : e \in \mathrm{seq}_n(X)/E^\mathfrak{m}_n \text{ for some } 0 < n < \omega \text{ and } \bar{q} \in (\mathbb{Z}^+)^n )$;
	\item\label{item_primes(b)} for every $\ell < n$, $p \not \vert \; q_\ell$.
%	\item if $e \in \mrm{seq}_n(X)/E^\mathfrak{m}_n$, $\bar{q} = (q_0, ..., q_{n-1}) \in (\mathbb{Z}^+)^n$, $p_{(e, \bar{q})} = p$, and $\ell < n$, then $p \not\vert \; q_\ell$, so, recalling \ref{Qp}(\ref{Qp+}), we have $q_\ell \in \mathbb{Q}^\odot_p$ (this is relevant for \ref{hyp_C1}(\ref{item17_new})).
	\end{enumerate}
\end{definition}

	\begin{fact}\label{K2bo_non_empty+} Clearly every element of $\mathfrak{m} \in \mathrm{K}^{\mrm{bo}}_2(M)$ can be expanded to an element of $\mathfrak{m} \in \mathrm{K}^{\mrm{bo}}_3(M)$, and, as we showed in \ref{K2bo_non_empty++}, that $\mathrm{K}^{\mrm{bo}}_2(M) \neq \emptyset$ we have $\mathrm{K}^{\mrm{bo}}_3(M) \neq \emptyset$.
\end{fact}

	We try to give some intuition on the group $G_1 = G_1[\mathfrak{m}]$ which we are about to introduce in \ref{def_G012_borel}. This group will be some sort of universal domain for our construction, and in fact all the $\mrm{TFAB}_\omega$'s which will be in the range of our Borel reduction from $\mathbf{K}^{\mrm{eq}}_2$ (cf. \ref{hyp_C0}) to $\mrm{TFAB}_\omega$ will be pure subgroups of this group $G_1$. The group $G_1$ naturally interpolates between $G_0 = \bigoplus \{ \mathbb{Z}x : x \in X\}$ and $G_2 = \bigoplus \{ \mathbb{Q}x : x \in X\}$, which have respectively the minimal and the maximal amount of divisibility possible.  Clearly, the groups $G_0$ and $G_2$ do not code anything of the universal countable model $M \in \mathbf{K}^{\mrm{eq}}_2$ (cf. \ref{hyp_C0}). Thus, we want to find a subgroup $G_0 \leq G_1 \leq G_2$ which does encode $M$. We do this adding divisibility conditions to $G_0$ which depend on the relation $E^{\mathfrak{m}}_n$ from \ref{hyp_C1}. So the first step is that for every $a \in G^+_0$ we choose a prime $p_a$ and require the following condition:
	$$G_0 \models a = \sum_{\ell < k} q_\ell x_\ell \neq 0 \; \Rightarrow \; G_1 \models p^{\infty}_a \vert \, a.$$
However, we want the partial permutations $f_{\bar{g}}$ of $X$ from \ref{hyp_C1} to induce partial automorphisms $\hat{f}^1_{\bar{g}}$ of our desired group $G_1$, and so we naturally demand:
$$\iota \in \{1, 2 \}, \; a_\iota = \sum_{\ell < k} q_\ell x^\iota_\ell, \; \bigwedge_{\ell < k} f_{\bar{g}}(x^1_\ell) = x^2_\ell \; \Rightarrow \; p_{a_1} = p_{a_2}.$$
Formally, this translates into a choice of $p_{(e, \bar{q})}$ as in \ref{item_primes}, where condition \ref{item_primes}(\ref{item_primes(b)}) is simply a useful technical requirement. We finally define our ``universal'' group $G_1$.

\begin{definition}\label{def_G012_borel} Let $\mathfrak{m} \in \mathrm{K}^{\mrm{bo}}_3(M)$.
	\begin{enumerate}[(1)]
	\item Let $G_2 = G_2[\mathfrak{m}]$ be $\bigoplus \{ \mathbb{Q}x : x \in X\}$.
	\item Let $G_0 = G_0[\mathfrak{m}]$ be the subgroup of $G_2$ generated by $X$, i.e. $\bigoplus \{ \mathbb{Z}x : x \in X\}$.
	\item\label{deff_G1} Let $G_1 = G_1[\mathfrak{m}]$ be the subgroup of $G_2$ generated by:
	\begin{enumerate}[(a)]
	\item $G_0$;
	\item\label{def_G012_borel_itemb} $p^{-m}(\sum_{\ell < n}q_\ell x_\ell)$, where:
	\begin{enumerate}[(i)]
	\item $0 < m < \omega$;
	\item $\bar{x} = (x_\ell : \ell < n) \in \mathrm{seq}_n(X)$, $e = \bar{x}/E^{\mathfrak{m}}_n$, $n > 0$;
	\item $\bar{q}$ is as in \ref{item_primes};
	\item $p = p_{(e, \bar{q})}$ (so a prime, recalling Definition~\ref{item_primes});
	\end{enumerate}
	\item\label{4.3(c)} [follows] for every $a \in G_1$ there are $i_* < \omega$ and, for $i < i_*$, $k_i$, $\bar{x}_i \in \mrm{seq}_{k_i}(X)$, $\bar{q}_i \in (\mathbb{Z}^+)^{k(i)}$, $e_i = \bar{x}_i/E^{\mathfrak{m}}_{k_i}$, $p_i = p_{(e_i, \bar{q}_i)}$ (hence $\bar{q}_i$ is as in \ref{item_primes}), $m(i) \geq 0$ and $r^i \in \mathbb{Z}^+$ such that the following condition holds:
	$$a = \sum \{p^{-m(i)}_i r^i q_{(i, \ell)} x_{(i, \ell)} : i < i_*, \ell < k_i \}.$$
	\end{enumerate}
	\item\label{deff_G1p} For a prime $p$, let $G_{(1, p)} = \{ a \in G_1 : a \text{ is divisible by $p^m$, for every $0 < m < \omega$}\}$ (notice that, by Observation~\ref{generalG1p_remark}, $G_{(1, p)}$ is always a pure subgroup of $G_1$).
	\item\label{G_U_pure} For $\mathcal{U} \subseteq M$, we let:
	$$G_{(1, \mathcal{U})}[\mathfrak{m}] = G_{(1, \mathcal{U})}[\mathfrak{m}(M)] = G_{(1, \mathcal{U})} = \langle y: y \in X_u, u \subseteq_{1} \mathcal{U} \rangle^*_{G_1} = \langle X_\mathcal{U} \rangle^*_{G_1}.$$
The notation $\mathfrak{m}(M)$ is from the second line of Def.~\ref{hyp_C1} and $X_\mathcal{U}$ is from \ref{hyp_C1}(\ref{def_XU}).
	\item\label{hatf_partial_auto} For $f_{\bar{g}} \in \bar{f}^{\mathfrak{m}}$ (cf. Definition~\ref{hyp_C1}(\ref{item7})), let $\hat{f}^2_{\bar{g}}$ be the unique partial automorphism of $G_2$ which is induced by $f_{\bar{g}}$ (see \ref{lemma1_Borel}(\ref{lemma1_Borel_2})), explicitly: if $k < \omega$ and for every $\ell < k$ we have that $y^1_\ell \in \mathrm{dom}(f_{\bar{g}})$, $y^2_\ell = f_{\bar{g}}(y^1_\ell)$, $q_\ell \in \mathbb{Q}^+$, then:
	$$a = \sum_{\ell < k} q_\ell y^1_\ell \in G_2 \; \Rightarrow \; \hat{f}^2_{\bar{g}}(a) = \sum_{\ell < k} q_\ell y^2_\ell.$$
%Notice that if $\sum_{\ell < k} q_\ell y^1_\ell \in G_1$, then also $\sum_{\ell < k} q_\ell y^2_\ell \in G_1$, by Definition~\ref{def_G012_borel}(\ref{deff_G1}) recalling Definition~\ref{hyp_C1}(\ref{item_8a}) and (\ref{def_En}), this is relevant for Lemma~\ref{lemma1_Borel}(\ref{item_hat1}).
	\item\label{def_G012_borel_item7} For $\ell \in \{0, 1\}$ we let $\hat{f}^2_{\bar{g}} \restriction G_\ell = \hat{f}^\ell_{\bar{g}}$ and $\hat{f}_{\bar{g}} = \hat{f}^1_{\bar{g}}$ (see \ref{lemma1_Borel}(\ref{lemma1_Borel_2})).
	\item\label{def_supp} For $i \in \{0, 1, 2\}$, $a = \sum_{\ell < m}q_\ell x_\ell \in G_i$, with $(x_\ell : \ell < k) \in \mrm{seq}_k(X)$ and $q_\ell \in \mathbb{Q}^+$, let $\mathrm{supp}(a) = \{ x_\ell : \ell < m \}$, i.e., when $a \in G^+_i$, $\mrm{supp}(a) \subseteq_\omega X$ is the smallest subset of $X$ such that $a \in \langle \mrm{supp}(a) \rangle^*_{G_i}$.
	\item For $p$ a prime and $a \in G^+_2$ we define the $p$-support of $a$, denoted as $\mrm{supp}_p(a)$, as: if $a = \sum \{q_\ell x_\ell : \ell < k\}$ with $(x_\ell : \ell < k) \in \mrm{seq}_k(X)$ and $q_\ell \in \mathbb{Q}^+$, then:
	$$\mrm{supp}_p(a) = \{x_\ell : \ell < k \text{ and } q_\ell \notin \mathbb{Q}_p\},$$
where we recall that $\mathbb{Q}_p$ was defined in \ref{Qp}.
	\end{enumerate}	 
\end{definition}

\begin{lemma}\label{lemma1_Borel} Let $\mathfrak{m} \in \mathrm{K}^{\mrm{bo}}_3$  and $\ell \in \{0, 1, 2\}$.
	\begin{enumerate}[(1)]
	\item $G_\ell[\mathfrak{m}] \in \mathrm{TFAB}$ and $|G_\ell[\mathfrak{m}]| = \aleph_0$.
	\item\label{lemma1_Borel_2}
	\begin{enumerate}
	\item $\hat{f}^2_{\bar{g}}$ is a partial automorphisms of $G_2[\mathfrak{m}]$ mapping $G_0[\mathfrak{m}]$ into itself;
	\item\label{item_hat1} $\hat{f}_{\bar{g}} = \hat{f}^1_{\bar{g}} = \hat{f}^2_{\bar{g}} \restriction G_{(1, \mathrm{dom}(\bar{g}))}$ (cf. Def.~\ref{def_G012_borel}(\ref{G_U_pure})(\ref{def_G012_borel_item7})), the map $\hat{f}_{\bar{g}}$ is a well-defined partial automorphism of $G_1$, and $\mathrm{dom}(\hat{f}_{\bar{g}})$ is a pure subgroup of $G_1[\mathfrak{m}]$, in fact $\mathrm{dom}(\hat{f}_{\bar{g}})$ is the pure closure in $G_1$ of $\mathrm{dom}(\hat{f}^0_{\bar{g}})$;
	\item $\hat{f}_{\bar{g}^{-1}} = \hat{f}^{-1}_{\bar{g}}$;
	\item $\bar{g}_1 \subseteq \bar{g}_2 \Rightarrow \hat{f}_{\bar{g}_1} \subseteq \hat{f}_{\bar{g}_2}$;
	\item $f_{\bar{g}} \subseteq \hat{f}^0_{\bar{g}} \subseteq \hat{f}^1_{\bar{g}} \subseteq \hat{f}^2_{\bar{g}}$.
	\end{enumerate}
	\item\label{G1p_pure} If $p = p_{(e, \bar{q})}$, $e \in \mathrm{seq}_n(X)/E^{\mathfrak{m}}_n$, $\bar{q} = (q_\ell : \ell < n)$ is as in \ref{item_primes}, and $n \geq 1$, then:
	\begin{enumerate}[(a)]
	\item\label{G1p_pure_1} $\langle \sum_{\ell < n} p^{-m} q_\ell y_\ell : m < \omega, \bar{y} \in e\rangle^*_{G_1} \leq G_{(1, p)}$;
	\item $G_1 \leq \langle \{ p^{-m} \sum_{\ell < n} q_\ell y_\ell : \bar{y} \in e \} \cup \mathbb{Q}_p G_0 \rangle_{G_2}$;
	\item if $a \in G_1$, then there are $k <\omega$, and, for $i < k$, $\bar{y}^i \in e$, $s_i \in \mathbb{Q}^+$ s.t.:
	\begin{enumerate}
	\item $a = \sum_{i < k} s_i(\sum_{\ell < n} q_\ell y^i_\ell) \;\;\mrm{mod}(\mathbb{Q}_pG_0 \cap G_1)$;
	\item  for all $i < k$, $s_i \sum_{\ell < n} q_\ell y^i_\ell \notin \mathbb{Q}_pG_0$, and $\ell < n$ implies $s_i q_\ell y^i_\ell \notin \mathbb{Q}_p G_0$;
	\item $s_i \sum \{q^i_\ell y^i_\ell : \ell < n\} \in G_1$.
	\end{enumerate}
\end{enumerate}
\item In \ref{lemma1_Borel}(\ref{G1p_pure}) we may add: $(\bar{y}^i : i < i_*)$ is with no repetitions.
\end{enumerate}
\end{lemma}

	\begin{proof} Item (1) is clear. Concerning item (2), clause (a) holds as $f_{\bar{g}}$ is a partial $1$-to-$1$ function from $X$ to $X$; while for clause (b) it suffices to prove that given $\sum_{\ell < k} q_\ell y^1_\ell$ and $\sum_{\ell < k} q_\ell y^2_\ell$ as in Definition~\ref{def_G012_borel}(\ref{hatf_partial_auto}) we have that:
	$$\sum_{\ell < k} q_\ell y^1_\ell \in G_1 \Rightarrow \sum_{\ell < k} q_\ell y^2_\ell \in G_1.$$
In order to verify this it suffices to consider the case in which $a := \sum_{\ell < k} q_\ell y^1_\ell$ is one of the generators of $G_1$ from \ref{def_G012_borel}(\ref{deff_G1}). Thus, to conclude it suffices to notice that $f_{\bar{g}}$ maps $\bar{y}^1 = (y^1_\ell : \ell < k)$ to $\bar{y}^2 = (y^2_\ell : \ell < k)$, hence $\bar{y}^2 \in \bar{y}^1/E^{\mathfrak{m}}_k$ and recall \ref{def_G012_borel}(\ref{def_G012_borel_itemb}). This shows (2)(b). Finally, items (2)(c)-(e) are easy and so we omit details.

\smallskip
\noindent
Concerning item (3), if $\bar{y} \in e$ and $0 < m < \omega$, then $p^{-m} \sum_{\ell < k} q_\ell y_\ell$ is one of the generators of $G_1$, as this holds for every $0 < m < \omega$, it follows that $\sum_{\ell < k} p^{-m} q_\ell y_\ell \in G_{(1, p)}$, by the definition of $G_{(1, p)}$. As $G_{(1, p)}$ is a subgroup of $G_1$, for every $\bar{y} \in e$ we have that $\sum_{\ell < n} q_\ell y_\ell \in G_{(1, p)} \leq G_1$. Let $Z_{(e, \bar{q})} = \{ \sum_{\ell < n} q_\ell y_\ell : \bar{y} \in e\} \subseteq G_{(1, p)}$, then $\langle Z_{(e, \bar{q})} \rangle^*_{G_1} \leq G_{(1, p)}$, because by Definition~\ref{def_G012_borel}(\ref{deff_G1p}) we have that $G_{(1, p)}$ is a pure subgroup of $G_1$ (cf. Obs. \ref{generalG1p_remark}).
This proves (3)(a). 

\smallskip
\noindent Concerning (3)(b)(c), assume:
	\begin{enumerate}[$(*_1)$]
	\item $a \in G^+_1$.
	\end{enumerate}
	By \ref{def_G012_borel}(3)(c), we have:
	\begin{enumerate}[$(*_2)$]
	\item As $a \in G_{1}$ we can find:
	\begin{enumerate}[(a)]
	\item $i_* < \omega$;
	\item for $i < i_*$, $e_i = \bar{x}_i/E_{k_i}$, $\bar{x}_i \in \mrm{seq}_{k_i}(X)$, $\bar{q}^i = (q^i_\ell : \ell < k_i) \in (\mathbb{Z}^+)^{k_i}$;
	\item $r^i \in \mathbb{Z}^+$, $\bar{y}^i \in e_i$, $b_i = \sum_{\ell < k_i} q^i_\ell y^i_\ell \in G_0$;
	\item $p_i = p_{(e_i, \bar{q}^i)}$;
	\item $a = \sum_{i < i_*} p^{-m(i)}_i r^i b_i$, where $m(i) < \omega$;
	\item $(b_i : i < i_*)$ is with no repetitions;
	\item $p^{-m(i)}_i r^i b_i \in G_1$.
	\end{enumerate}
\end{enumerate}
Let now:
\begin{enumerate}[$(*_3)$]
	\item $V = \{i < i_* : p_i = p = p_{(e, \bar{q})} \text{ and } p^{-m(i)}_i r^ib_i \notin \mathbb{Q}_p G_0 \}$,
	\end{enumerate}
where we recall that the object $p_{(e, \bar{q})}$ is from the statement of lemma and in particular it is fixed. Notice also that if $i \in V$, then $(e, \bar{q}) = (e_i, \bar{q}^i)$. Hence, we have:
\begin{enumerate}[$(*_4)$]
	\item 
	\begin{enumerate}[(a)]
	\item if $i \in i_* \setminus V$, then $p^{-m(i)}_i r^i b_i \in \mathbb{Q}_pG_0$;
	\item $i \in V$ implies $\bar{y}^i \in e$ and $\bar{q}^i = \bar{q}$;
	\item if $i \in V$ and $\ell(1), \ell(2) < k$, then:
	$$p^{-m(i)}_i r^i q^i_{\ell(1)}\in \mathbb{Q}_p \; \Leftrightarrow \; p^{-m(i)}_i r^i q^i_{\ell(2)}\in \mathbb{Q}_p \; \Leftrightarrow \; p^{-m(i)}_i r^i b_i \in \mathbb{Q}_p G_0;$$
	\item if $i \in V$, then $p^{-m(i)}_i r^i b_i \notin \mathbb{Q}_p G_0$;
	\item if $i \in V$ and $\ell < k$, then $p^{-m(i)}_i r^i q^i_\ell y^i_\ell \notin \mathbb{Q}_p G_0$.
	\end{enumerate}
	\end{enumerate}
[Notice that in the first equivalence of $(*_4)$(c) we use: $\ell < k \Rightarrow q_\ell \in \mathbb{Z}^+, p \not\vert \; q_\ell$.]

\smallskip
\noindent
By $(*_4)$ we have:
\begin{enumerate}[$(*_5)$]
	\item 
	\begin{enumerate}[(a)]
	\item $a = \sum \{p^{-m(i)}_i r^i b_i : i \in V\} \; \mrm{mod}(\mathbb{Q}_pG_0 \cap G_1)$; 
	\item $i \in V$ implies $p^{-m(i)}_i r^i b_i \notin \mathbb{Q}_p G_0$.
	\end{enumerate}
	\end{enumerate}
So, defining $s_i$ as $p^{-m(i)} r^i$, we are done proving (3)(b)(c). Finally, (4) is easy.
\end{proof}

%	\begin{observation}\label{nice_and_special_obs} Let $K$ and $G$ be graphs such that $|K| = \aleph_0 = |G|$.
%	\begin{enumerate}[(1)]
%	\item If $K$ is a nice graph, then $K \times \omega$ is a specially nice graph.
%	\item If $G$ is a specially nice graph, then there is a nice graph $H$ such that $G \cong H \times \omega$, $|H| = \aleph_0$ and such an $H$ is unique up to isomorphism.
%	\item\label{nice_and_special_obs_3} If $G$ is specially nice, $A \subseteq G$, $|A| < \aleph_0$ and $b \in G \setminus A$, then the set $\{\pi(b) : \pi \in \mrm{Aut}(G) \text{ and } \pi \restriction A = \mrm{id}_A \}$ has size $\aleph_0$. 
%	\end{enumerate}
%\end{observation}

\begin{fact}\label{the_clear_fact} Assume that $\mathfrak{m} \in \mathrm{K}^{\mrm{bo}}_3(M)$, $\mathcal{U}, \mathcal{V} \subseteq M$ and $|\mathcal{U}| = |\mathcal{V}| = \aleph_0$. Suppose further that there is $h: M \restriction \mathcal{U} \cong M \restriction \mathcal{V}$. Then there is $\bar{g} = (g_k : k < \omega)$ such that:
	\begin{enumerate}[(a)]
	\item\label{the_clear_fact_a} for every $k < \omega$, $g_k \in \mathcal{G}$ (cf. \ref{hyp_C0}(\ref{hyp_C0_item2})); 
	\item\label{the_clear_fact_b} for every $k < \omega$, $g_k \subsetneq g_{k+1}$;
	\item\label{the_clear_fact_c} $\bigcup_{k < \omega} g_k: M \restriction \mathcal{U} \cong M \restriction \mathcal{V}$.
	\end{enumerate}
\end{fact}

	\begin{proof} Let $h: M \restriction \mathcal{U} \cong M \restriction \mathcal{V}$. We can choose an increasing sequence $(n_k : k < \omega)$ such that $g_k = h \cap (n_k \times n_k)$ (pedantically $g = (h \cap (n_k \times n_k), 1)$ recalling \ref{hyp_C0}(\ref{hyp_C0_item2})) is strictly increasing and $\bigcup_{k < \omega} g_k = h$.
\end{proof}

	As mentioned, $G_1$ will be some sort of universal domain for our construction. This is reflected by the fact that instead of varying $M \in \mathbf{K}^{\mrm{eq}}$ in Definition~\ref{hyp_C1}, we fix $M$ to be the countable universal homogeneous model of $\mathbf{K}^{\mrm{eq}}$, and, for $\mathcal{U} \subseteq M$, we consider the substructure $M \restriction \mathcal{U}$ and the group $G_{(1, \mathcal{U})}$. We intend to show:
	$$M \restriction \mathcal{U} \cong M \restriction \mathcal{V} \; \Leftrightarrow \; G_{(1, \mathcal{U})}[\mathfrak{m}] \cong G_{(1, \mathcal{V})}[\mathfrak{m}].$$
The easy direction is course the left-to-right one, which we now establish:

	\begin{cclaim}\label{why_iso} Assume that $\mathfrak{m} \in \mathrm{K}^{\mrm{bo}}_3(M)$, $\mathcal{U},  \mathcal{V} \subseteq M$ and $|\mathcal{U}| = |\mathcal{V}| = \aleph_0$. Then:
	$$M \restriction \mathcal{U} \cong M \restriction \mathcal{V} \; \Rightarrow \; G_{(1, \mathcal{U})}[\mathfrak{m}] \cong G_{(1, \mathcal{V})}[\mathfrak{m}].$$
\end{cclaim}

	\begin{proof} Let $(g_k : k < \omega)$ be as in Fact~\ref{the_clear_fact}, $s_k = \mrm{dom}(g_k)$ and $t_k = \mrm{ran}(g_k)$. Then:
\begin{enumerate}[(i)]
\item for $k < \omega$, $\bar{g}_k = (g_\ell : \ell \leq k)$, so $\bar{g}_k \in \mathcal{G}^{k+1}_*$ (cf.~\ref{hyp_C0}(\ref{G_*}) and \ref{the_clear_fact}(\ref{the_clear_fact_a})(\ref{the_clear_fact_b}));
\item $\bigcup_{k < \omega} g_k: M \restriction \mathcal{U} \cong M \restriction \mathcal{V}$ (cf. \ref{the_clear_fact}(\ref{the_clear_fact_c}));
\item for every $k < \omega$ we have that $\bar{g}_k \in \mathcal{G}_*$ and so, by Definition~\ref{hyp_C1}(\ref{item7}),  $f_{\bar{g}_k} \in \bar{f}^{\mathfrak{m}}$.
\end{enumerate} 

\noindent Notice also that by \ref{hyp_C1}(\ref{item_for_iso_ltr}) we have:
\begin{enumerate}[$(\star_1)$]
\item 
\begin{itemize}
\item[(d)] $\bigcup_{k < \omega} \mrm{dom}(f_{\bar{g}_k}) = \bigcup_{k < \omega} X_{s_{k}} = X_{\mathcal{U}}$;
\item[(e)] $\bigcup_{k < \omega} \mrm{ran}(f_{\bar{g}_k}) = \bigcup_{k < \omega} X_{h[s_{k}]} = X_{\mathcal{V}}$.
\end{itemize} 
\end{enumerate} 

\noindent Hence, we have:
\begin{enumerate}[$(\star_2)$]
\item $\bigcup_{k < \omega} \hat{f}_{\bar{g}_k}$ is an isomorphism from $G_{(1, \mathcal{U})}$ onto $G_{(1, \mathcal{V})}$ (cf. Def.~\ref{def_G012_borel}(\ref{G_U_pure})(\ref{def_G012_borel_item7})).
\end{enumerate}
[Why? By \ref{def_G012_borel}(\ref{G_U_pure})(\ref{hatf_partial_auto})(\ref{def_G012_borel_item7}), \ref{lemma1_Borel}(\ref{item_hat1}) and \ref{hyp_C1}(\ref{item_for_iso_ltr}).]
\end{proof}

\subsection{Analyzing Isomorphism}\label{sub_ana_iso}

	Our aim in this subsection is to prove the converse of Claim~\ref{why_iso}.

	\begin{hypothesis}\label{hyp_non_iso} Throughout this subsection the following hypothesis holds:
	\begin{enumerate}[(1)]
	\item $\mathfrak{m} \in \mathrm{K}^{\mrm{bo}}_3(M)$;
	\item $\mathcal{U}, \mathcal{V} \subseteq M$;
	\item\label{U_infinite} $|\mathcal{U}| = \aleph_0 = |\mathcal{V}|$;
	\item\label{hyp_non_iso_item3} $\pi$ is an isomorphism from $G_{(1, \mathcal{U})}[\mathfrak{m}]$ onto $G_{(1, \mathcal{V})}[\mathfrak{m}]$.
	\end{enumerate}
	%This hypothesis is toward proving Conclusion~\ref{crucial_Borel_conclusion}.
\end{hypothesis}	

	Our aim in \ref{crucial_borel_lemma} and \ref{concl_sequence_rationals_Borel} below is to show that $\pi$ essentially comes from a bijection from $X_{\mathcal{U}}$ onto $X_{\mathcal{V}}$, which are respectively the bases of $G_{(1, \mathcal{U})}[\mathfrak{m}]$ and $G_{(1, \mathcal{V})}[\mathfrak{m}]$ (in the appropriate sense). At the bottom of this is the crucial algebraic condition \ref{hyp_C1}(\ref{item17_new}), which puts restrictions on the possible $p$-supports of certain members of $G_1$.

\begin{lemma}\label{crucial_borel_lemma} Let $a \in G_{(1, \mathcal{U})}[\mathfrak{m}] $ and let $b = \pi(a)$.
	\begin{enumerate}[(1)]
	\item\label{crucial_borel_lemma_item1} For a prime $p$, $a \in G_{(1, p)} \Leftrightarrow b \in G_{(1, p)}$;
	\item if $a = qx$, for some $q \in \mathbb{Q}^+$ and $x \in X_{\mathcal{U}}$, then for some $y \in X_{\mathcal{V}}$:
	\begin{enumerate}[(a)]
	\item\label{crucial_borel_lemma_a} $(x) E^{\mathfrak{m}}_1 (y)$;
	\item\label{crucial_borel_lemma_b} $b \in \mathbb{Q}y$, i.e. there exist $m_1, m_2 \in \mathbb{Z}^+$ such that $m_1b = m_2 y$.
\end{enumerate}	 
	\end{enumerate}
\end{lemma}

	\begin{proof} Item (1) is obvious by Hypothesis~\ref{hyp_non_iso}(\ref{hyp_non_iso_item3}). Notice now that: 
	\begin{enumerate}[$(*_{0})$]
	\item It suffices to prove (2)(b).
\end{enumerate}
Why $(*_{0})$? Suppose that $b = \frac{m_2}{m_1}y$ and let $e' = (x)/E^{\mathfrak{m}}_1$ and $p' = p_{(e', (1))}$, then $x \in G_{(1, p')}$, but $a = qx$ and $a \in G_1$, hence $a \in G_{(1, p')}$. Now, applying (1) with $(a, b, p')$ here standing for $(a, b, p)$ there, we get that $b \in G_{(1, p')}$. As $b = \frac{m_2}{m_1}y \in G_1$, we have that $y \in G_{(1, p')}$ and thus:
\begin{enumerate}[$(\cdot)$]
	\item $G_1 \models (p')^\infty \vert \, x$ and $G_1 \models (p')^\infty \vert \, y$.
\end{enumerate}
Now, letting $H_{(p', 0)} = \langle x/E^{\mathfrak{m}}_1 \rangle_{G_0}$ and $H_{(p', 1)} = \langle x/E^{\mathfrak{m}}_1 \rangle^*_{G_1}$ we have that:
\begin{enumerate}[$(*_{0.1})$]
	\item 
	\begin{enumerate}[(i)]
	\item $G_0/H_{(p', 0)}$ is canonically $\cong$ to the direct sum of $\langle \mathbb{Z}y : y \in X \setminus x/E^{\mathfrak{m}}_1 \rangle$;
	\item $H_{(p', 1)} \cap G_0 = H_{(p', 0)}$;
	\item $G_1/H_{(p', 1)}$ naturally extends $G_0/H_{(p', 0)}$;
	\item no non-zero element of $G_1/H_{(p', 1)}$ is divisible by $(p')^\infty$.
\end{enumerate}
\end{enumerate}
Why $(*_{0.1})$? Straightforward or see a detailed proof of a more complicated case in \ref{lemma_pre_main_th}(\ref{G1p_pure_co-hop}). This concludes the proof of $(*_{0})$.

\smallskip
\noindent Coming back to the proof:
	\begin{enumerate}[$(*_{1})$]
	\item Let $n < \omega$, $\bar{y} \in \mathrm{seq}_n(X_{\mathcal{V}})$ and $\bar{q} \in (\mathbb{Q}^+)^n$ be such that $b = \sum \{q_\ell y_\ell : \ell < n \}$.
\end{enumerate}
Trivially, $n > 0$, we shall show that $n = 1$, i.e., that (2)(b) holds. To this extent: 
\begin{enumerate}[$(*_{1.1})$]
	\item Let $q_* \in \omega \setminus \{0 \}$ be such that:
\begin{enumerate}[$(\cdot_1)$]
	\item $b_1 := q_*b \in G_0[\mathfrak{m}]$;
	\item $q_* q_\ell \in \mathbb{Z}$;
	\item for every prime $p'$ we have $p' \, \vert \; (q_*q)$ implies $p' \, \vert \; (q_*q_\ell)$, for all $\ell < n$.
\end{enumerate} 
\end{enumerate}
Let $e = \bar{y}/E_n$, $q'_\ell = q_*q_\ell$ and $\bar{q}' = (q'_\ell : \ell < n)$, so that $q_* q_\ell y_\ell = q'_\ell y_\ell$ and $q'_\ell \in \mathbb{Z}^+$. Let $p = p_{(e, \bar{q}')}$ and let $b_1 = q_*b = \sum \{q'_\ell y_\ell : \ell < n \}$. Notice that we have:
	\begin{enumerate}[$(*_{2})$]
\item $\bigwedge_{\ell < k} p \not\vert \; q'_\ell$ and, for every $\ell < k$, $q'_\ell \in \mathbb{Z}^+ \subseteq \mathbb{Q}_p$.
\end{enumerate}
[Why? Because $p = p_{(e, \bar{q}')}$ has been chosen in \ref{item_primes} exactly in this manner.]
\newline Then we have:
	\begin{enumerate}[$(*_{3})$]
\item
\begin{enumerate}[(i)]
	\item $b \in G_{(1, p)}$;
	\item $a \in G_{(1, p)}$;
	\item if $m < \omega$, then $p^{-m} a \in G_{(1, p)} \leq G_1$.
\end{enumerate} 
\end{enumerate}
[Why (i)? By the choice of $p$ we have that $b_1 \in G_{(1, p)}$ (cf. Def.~\ref{def_G012_borel}(\ref{deff_G1})(\ref{deff_G1p})) and so, as $G_{(1, p)}$ is pure in $G_1$ (cf. Observation~\ref{generalG1p_remark}), 
%$G_{(1, \mathcal{V})}$ is pure in $G_1$ (cf. Definition~\ref{def_G012_borel}(\ref{G_U_pure})), 
$b_1 = q_*b$ and $q_* \in \mathbb{Z}$, we have $b \in G_{(1, p)}$ (cf. Observation~\ref{obs_pure_TFAB}).
Why (ii)? By (1) and (i), recalling Hyp.~\ref{hyp_non_iso}(\ref{hyp_non_iso_item3}). Lastly, (iii) is immediate: by the definition of $G_1$ and of $G_{(1, p)}$ (Definition~\ref{def_G012_borel}(\ref{deff_G1})(\ref{deff_G1p})).]
	\begin{enumerate}[$(*_4)$]
	\item W.l.o.g. $a = qx \notin \mathbb{Q}_p G_0$ and $pa \in G_0$.
	\end{enumerate}
We prove $(*_4)$. Let $a' = p^{-1} q_* a$, $b' = p^{-1} q_* b$ and $q' = p^{-1} q_*$. So by $(*_3)$ we have that $a', b' \in G_1$ and of course $\pi(a') = b'$. Now, by the choice of $b'$ and $q_*$ (cf. in particular $(*_{1.1})(\cdot_3)$) we have that $pb' \in G_{(0, \mathcal{V})}$, hence $pa' = \pi^{-1}(pb') \in G_{(0, \mathcal{U})}$. Notice that $a' \not\in G_{(0, \mathcal{U})}$, as $a' \notin \mathbb{Q}_pG_0$ because $b' \notin \mathbb{Q}_pG_0$, since from $(*_2)$ above, $\bigwedge_{\ell < k} p \not\vert \; q'_\ell$. Noticing that $(a', b', q'_*, b_1, p, \bar{q}')$ satisfies all the demands of $(a, b, q_*, b_1, p, \bar{q}')$ (including $(*_3)$), it follows that:
	\begin{enumerate}[$(*_{4.1})$]
	\item 
	\begin{enumerate}[(a)]
	\item replacing $(a, q, b)$ with $(a', q', b')$ we can assume that $a = qx \notin \mathbb{Q}_pG_0$;
	\item if $b'$ belongs to $\mathbb{Q}y$ for some $y \in X_{\mathcal{V}}$, the the conclusion of (2) is satisfied.
\end{enumerate}
\end{enumerate}
This concludes the proof of $(*_4)$.

%First of all:
%\begin{enumerate}[$(*_{4.1})$]
%	\item There is $q' \in \{p^m r: m \in \mathbb{Z}, r \in \mathbb{Z}^+$, $(r, p) = 1 \}$ such that:
%	\begin{enumerate}[(a)]
%	\item $q'qx \notin G_0$;
%	\item $pq'qx \in G_0$.
%	\end{enumerate}
%	\end{enumerate}
%Secondly, as $a = qx \in G_{(1, p)}$ and $q' \in \{p^m r : m \in \mathbb{Z}, r \in \mathbb{Z}^+, (r, p) = 1\}$, clearly $q'qx \in G_{(1, p)} \leq G_1$. Now, for the sake of contradiction, suppose that $q'qx \in \mathbb{Q}_p G_0$, then for some $m \in \mathbb{Z}^+$ we have:
%\begin{enumerate}[$(*_{4.2})$]
%	\item $(m, p) = 1$ and $m(q'qx) \in G_0$.
%\end{enumerate}
%As $(m, p) = 1$ and $m, p \in \mathbb{Z}^+$ there are $r_1, r_2 \in \mathbb{Z}^+$ such that $r_1 m + r_2 p = 1$. Hence:
%$$\begin{array}{rcl}
%q'qx & = & 1q'qx \\
%	 & = & (r_1 m + r_2 p) q'qx \\
%	 & = & r_1(mq'qx) + r_2(pq'qx) \in G_0,
%\end{array}$$
%where at the end we used that $mq'qx \in G_0$ (by $(*_{4.2})$) and that $pq'qx \in G_0$ by $(*_{4.1})$(b)). Hence, we have that $q'qx \in G_0$, which contradicts $(*_{4.1})$(a). Thus, $q'qx \notin \mathbb{Q}_p G_0$. Also as $\pi(q'qx) = q'b$, necessarily $q'b \in G_1$. Thus, $(q'a, q'b, q'q, \frac{q_*}{q'}, b, \bar{q}, \bar{q}')$ satisfies all the demands on $(a, b, q, q_*, b, \bar{q}, \bar{q}')$ and so replacing the latter with the former we loose nothing and gain $(*_4)$. This concludes the proof of $(*_4)$.

\smallskip
\noindent
Now, by \ref{lemma1_Borel}(3), there are $k <\omega$, and, for $i < k$, $\bar{y}^i \in \bar{y}/E_n$ and $r_i \in \mathbb{Q}^+$ such that:
\begin{enumerate}[$(*_5)$]
	\item %$a = \sum_{i < k} r_i(\sum_{\ell < n} q'_\ell y^i_\ell) = \sum_{i < k} (\sum_{\ell < n} r_i q'_\ell y^i_\ell) \;\; \mrm{mod}(\mathbb{Q}_pG_0 \cap G_1)$;
	\begin{enumerate}
	\item $q x = a = \sum_{i < k} r_i(\sum_{\ell < n} q'_\ell y^i_\ell) = \sum_{i < k} (\sum_{\ell < n} r_i q'_\ell y^i_\ell) \;\; \mrm{mod}(\mathbb{Q}_pG_0 \cap G_1)$;
	\item $r_i \sum_{\ell < n} q'_\ell y^i_\ell \in G_1$ and $r_i  q'_\ell \notin \mathbb{Q}_p$.
\end{enumerate}
\end{enumerate}
By $(*_4)$,  $a = qx \notin \mathbb{Q}_pG_0$, and so clearly $k > 0$. It suffices to prove that $k = 1$, which by $(*_5)$ implies that $n = 1$, i.e., there is $y \in X_{\mathcal{V}}$ such that $b \in \mathbb{Q}y$. Why does it follow that $n = 1$? As otherwise the LHS of $(*_5)$(a) has $p$-support a singleton but the RHS of $(*_5)$(a) has $p$-support of size at least two, a contradiction.

\smallskip
\noindent So toward contradiction assume that $k \geq 2$. Recalling $(*_4)$ notice that:
\begin{enumerate}[$(*_6)$]
	\item $qx = a = \sum_{\ell < n} (\sum_{i < k} r_i q'_\ell y^i_\ell) \;\; \mrm{mod}(\mathbb{Q}_pG_0 \cap G_1)$;
\end{enumerate}
Now, let $Z = \{y^i_\ell : i < i_*, \ell < k\}$ and, for $y \in Z$, let:
$$a_y = \sum \{r_i q'_\ell : i < i_*, \ell < k, y^i_\ell = y\}.$$
So, by $(*_6)$ we have:
\begin{enumerate}[$(*_7)$]
	\item $qx = \sum \{a_y y : y \in Z\} \;\; \mrm{mod}(\mathbb{Q}_pG_0 \cap G_1)$.
\end{enumerate}
Now, as for the sake of contradiction we are assuming that $k \geq 2$, recalling that by $(*_{2})$ we have that $q'_\ell \in \mathbb{Z}^+ \subseteq \mathbb{Q}_p$, by \ref{hyp_C1}(\ref{item17_new}), we have the following:
\begin{enumerate}[$(*_8)$]
	\item $\mrm{supp}_p(\sum_{y \in Y} a_y y) = \{y \in Y : a_y \notin \mathbb{Q}_p \}$ is not a singleton.
\end{enumerate}
Now recall that, by $(*_4)$, $qx = a \notin \mathbb{Q}_pG_0$, hence $\mrm{supp}_p(qx) = \{ x \}$, so it is a singleton. By $(*_8)$, the RHS of $(*_7)$ has a non-singleton $p$-support whereas the LHS of $(*_7)$ has $p$-support a singleton, a contradiction. Hence, we are done proving~(2).
\end{proof}

\begin{conclusion}\label{concl_sequence_rationals_Borel} \begin{enumerate}[(1)]
	\item\label{concl_sequence_rationals_Borel_item1} There is a sequence $(q^1_x : x \in X_{\mathcal{U}})$ of non-zero rationals and a function $\pi_1 : X_{\mathcal{U}} \rightarrow X_{\mathcal{V}}$ such that for every $x \in X_\mathcal{U}$ we have that:
	$$\pi(x) = q^1_x(\pi_1(x)) \;\; \text{ and } \;\; \pi_1(x) \in x/E^{\mathfrak{m}}_1.$$
	\item There is a sequence $(q^2_x : x \in X_{\mathcal{V}})$ of non-zero rationals and a function $\pi_2 : X_\mathcal{V} \rightarrow X_\mathcal{U}$ such that:
	$$\pi^{-1}(x) = q^2_x(\pi_2(x)).$$
	\item \begin{enumerate}[(i)]
	\item $\pi_2 \circ \pi_1 : X_\mathcal{U} \rightarrow X_\mathcal{U} = id_{\mathcal{U}}$;
	\item $\pi_1 \circ \pi_2 : X_\mathcal{V} \rightarrow X_\mathcal{V} = id_{\mathcal{V}}$;
	\item $\pi_1 : X_\mathcal{U} \rightarrow X_\mathcal{V}$ is a bijection.
	\end{enumerate}
	\end{enumerate}
\end{conclusion}

	\begin{proof} (1) is by \ref{crucial_borel_lemma}, we elaborate. To this extent, let $R = \{(x, y) : x, y \in X \text{ and } \pi(x) \in \mathbb{Q}^+y \}$. Now, we have:
\begin{enumerate}[$(*_1)$]
\item For all $x \in X_\mathcal{U}$ there is $y \in X_\mathcal{V}$ such that $R(x, y)$.
\end{enumerate} 
[Why? By \ref{crucial_borel_lemma}(\ref{crucial_borel_lemma_b}) there is $y \in X_\mathcal{V}$ such that $\pi(x) \in \mathbb{Q}y$, as $\pi$ is an automorphism, necessarily $\pi(x) \neq 0$ and so $\pi(x) \in \mathbb{Q}^+y$.]
\begin{enumerate}[$(*_2)$]
\item If $x \in X_\mathcal{U}$ and $(x, y_1), (x, y_2) \in R$, then $y_1 = y_2$.
\end{enumerate} 
[Why? By the definition of $R$, there are $q_1, q_2 \in \mathbb{Q}^+$ such that $q_1 y_1 = \pi(x) = q_2 y_2$. As $q_1, q_2 \neq 0$, necessarily $q_1 = q_2$ and $y_1 = y_2$.]
\newline Together, $R$ is the graph of a function which we call $\pi_1$. Lastly, $\pi_1(x) \in x/E^{\mathfrak{m}}_1$ by \ref{crucial_borel_lemma}(\ref{crucial_borel_lemma_a}). Thus we proved (1).

\noindent (2) is by part (1) applied to $\pi^{-1}$ (and $\mathcal{V}$, $\mathcal{U}$).

\noindent (3) is by (1) and~(2). Why? E.g., for (i) we have that:
$$\pi^{-1} \circ \pi(x) = \pi^{-1}(q^1_x(\pi_1(x))) = q^2_{\pi_1(x)} q^1_x(\pi_2 \circ \pi_1(x)) = x,$$ 
which implies that $\pi_2 \circ \pi_1(x) = x$. (ii) is similar and (iii) follows from (i)+(ii).
\end{proof}

	Our aim in the subsequent claims is to lift the $1$-to-$1$ mapping from $X_{\mathcal{U}}$ onto $X_{\mathcal{U}}$ defined in  \ref{concl_sequence_rationals_Borel} to an isomorphism from $M \restriction \mathcal{U}$ onto $M \restriction \mathcal{V}$. We recall that the equivalence relations $\mathfrak{E}^M_i$ (for $i \in \{0, 1, 2\}$) were defined in \ref{hyp_C0}. We intend to show that our mapping $\pi_1$ and $\pi^{-1}_1 = \pi_2$ preserve them (and so also their negations). This is done introducing some auxiliary equivalence relations $\mathcal{E}_i$ (for $i \in \{0, 1, 2\}$) on $X$ which reflect (to some extent) the equivalence relations $\mathfrak{E}^M_i$ on $M$.

	\begin{definition} For $i < 3$, let:
	$$\mathcal{E}_i = \{ (x, y) : \text{ for some } (a, b) \in \mathfrak{E}^M_i, x \in X'_{\{a\}} \text{ and } y \in X'_{\{b\}}\},$$
where we recall that $\mathfrak{E}^M_i$ was introduced in \ref{hyp_C0}.
\end{definition}

	\begin{cclaim}\label{the_hammer_claim}
	\begin{enumerate}[(1)]
	\item If $(y_0, y_1) \in (x_0, x_1)/E^\mathfrak{m}_2$, $x_0, x_1, y_0, y_1 \in X$ and $i < 3$, then:
	$$x_0 \mathcal{E}_i x_1 \Leftrightarrow y_0 \mathcal{E}_i y_1.$$
	\item\label{the_hammer_claim_item2} The mapping $\pi_1$ from \ref{concl_sequence_rationals_Borel} preserves $\mathcal{E}_i$ and its negation, for all $i < 3$.
	\end{enumerate}
\end{cclaim}

	\begin{proof} (1) Suppose that $(y_0, y_1) \in (x_0, x_1)/E^\mathfrak{m}_2$. Then it is enough to prove:
	\begin{enumerate}[$(\star_1)$]
	\item If $\bar{g} \in \mathcal{G}_*$, $f_{\bar{g}}(x_\ell) = y_\ell$, for $\ell = 0, 1$, then $x_0 \mathcal{E}_i x_1 \Leftrightarrow y_0 \mathcal{E}_i y_1$. 
	\end{enumerate}
For $\ell = 0, 1$, let $x_\ell \in X'_{s_\ell}$, for $s_\ell \subseteq_1 M$ and $y_\ell \in X'_{t_\ell}$, for $t_\ell \subseteq_1 M$. Now, as $f_{\bar{g}}(x_\ell) = y_\ell$, by \ref{hyp_C1}(4)(d) we have that $\bar{g}[s_\ell] = t_\ell$. So $\bar{g}(s_0, s_1) = (t_0, t_1)$, and so, as $\bar{g} \in \mathcal{G}_*$ we have that $s_0 \mathfrak{E}_i s_1 \Leftrightarrow t_0 \mathfrak{E}^M_i t_1$. This implies $x_0 \mathcal{E}_i x_1 \Leftrightarrow y_0 \mathcal{E}_i y_1$. 

\smallskip
\noindent
Concerning (2). Using also $\pi_2, \mathcal{V}, \mathcal{U}$ it suffices to prove that for $x, y \in X_{\mathcal{U}}$ we have:
$$x \mathcal{E}_i y \Rightarrow \pi_1(x) \mathcal{E}_i \pi_1(y).$$ 
To this extent, suppose that $x \mathcal{E}_i y$ and let $s \subseteq_1 \mathcal{U}$ be such that $x, y \in X_{s/\mathfrak{E}^M_i}$ (as $s/\mathfrak{E}^M_i \subseteq M$, we are using \ref{hyp_C1}(\ref{def_XU}) to give meaning to the expression $X_{s/\mathfrak{E}^M_i}$). If $x = y$, then the conclusion is trivial, so we assume that $x \neq y$. 
\begin{enumerate}[$(\star_{1.1})$]
	\item Let $e = (x, y)/E^{\mathfrak{m}}_2$, $\bar{q} = (1, 1)$ and $p = p_{(e, \bar{q})}$. 
\end{enumerate}
Now, we claim:
	\begin{enumerate}[$(\star_{1.2})$]
	\item There is $0 < m < \omega$ be such that $p^{-m} q_x \pi_1(x) + p^{-m} q_y \pi_1(y) \notin \mathbb{Q}_p G_0 \cap G_1$.
\end{enumerate}
[Why? First of all, as $q_x, q_y \in \mathbb{Q}^+$ and $x \neq y \Rightarrow \pi_1(x) \neq \pi_1(y)$ and $\pi(x + y) \in G_1$, we have that $0 \neq a = q_x \pi_1(x) + q_y \pi_1(y) \in G_1$ and so we are done, recalling that by the definition of $\mathbb{Q}_p$ we have that for every $b \in G^+_1$ there is $m < \omega$ s.t. $p^{-m} b \notin \mathbb{Q}_p G_0$.]
\newline So fix an $m < \omega$ as in $(\star_{1.2})$. Now, by the choice of $p$, we have that $p^{-m}(x+y) \in G_{(1, p)} \leq G_1$ and so we have that the following is satisfied:
$$p^{-m} q_x \pi_1(x) + p^{-m} q_y \pi_1(y) = p^{-m} \pi(x) + p^{-m} \pi(y) = \pi(p^{-m} (x+y)) \in G_{(1, p)}.$$
So, by \ref{lemma1_Borel}(\ref{G1p_pure}) applied with $((x, y)/E^{\mathfrak{m}}_2, (1, 1), p, p^{-m} q_x \pi_1(x) + p^{-m} q_y \pi_1(y))$ standing for $(e, \bar{q}, p_{(e, \bar{q})}, a)$ there, there are $(x_j, y_j) \in (x, y)/E^{\mathfrak{m}}_2$ and $r_j \in \mathbb{Q}^+$, for $j < j_*$, s.t.:
	\begin{enumerate}[$(\star_2)$]
	\item \begin{enumerate}[(a)]
	\item $((x_j, y_j) : j < j_*)$ is with no repetitions;
	\item $p^{-m} q_x \pi_1(x) + p^{-m} q_y \pi_1(y) = \sum_{j < j_*} r_j (x_j + y_j) \;\; \mrm{mod}(\mathbb{Q}_p G_0 \cap G_1)$.
	%\item $r_j (x_j + y_j) \in G_1$.
	\end{enumerate}
	\end{enumerate}
Now, by (1), recalling $x \mathcal{E}_i y$, for $j < j_*$, there are $s_j \subseteq_1 M$ such that $x_j, y_j \in X_{s_j/\mathfrak{E}^M_i}$.
Next, by $(\star_{1.2})$, the LHS of $(\star_2)$(b) is not in $\mathbb{Q}_p G_0 \cap G_1$, so the same happens for the RHS of $(\star_2)$(b), hence, necessarily, $\{s_j/ \mathfrak{E}^M_i : j < j_*\} \neq \emptyset$ (i.e., $j_* \geq 1$), let  $(t_\ell/ \mathfrak{E}^M_i : \ell < \ell_*)$ list it without repetitions, with $t_\ell \in \{s_j : j < j_*\}$ for each $\ell < \ell_*$. Let then:
$$u_\ell = \{j < j_* : s_j/ \mathfrak{E}^M_i = t_\ell/ \mathfrak{E}^M_i\}.$$
So we have:
	\begin{enumerate}[$(\star_3)$]
	\item $p^{-m} q_x \pi_1(x) + p^{-m} q_y \pi_1(y) = \sum_{\ell < \ell_*} \sum_{j \in u_\ell} r_j (x_j+y_j) \;\; \mrm{mod}(\mathbb{Q}_p G_0 \cap G_1)$.
\end{enumerate}
Now, for $\ell < \ell_*$, let $c_\ell = \sum_{j \in u_\ell} r_j (x_j+y_j)$, then:
\begin{enumerate}[$(\star_4)$]
	\item $p^{-m} q_x \pi_1(x) + p^{-m} q_y \pi_1(y) = \sum_{\ell < \ell_*} c_\ell \;\; \mrm{mod}(\mathbb{Q}_p G_0 \cap G_1)$.
\end{enumerate}
\begin{enumerate}[$(\star_5)$]
	\item $(\mrm{supp}_p(c_\ell) : \ell < \ell_*)$ is a sequence of pairwise disjoint sets.
\end{enumerate}
[Why? As $\mrm{supp}(c_\ell) \subseteq X_{t_\ell/\mathfrak{E}^M_i}$, recalling the $t_\ell/\mathfrak{E}^M_i$'s are with no repetitions.]
\begin{enumerate}[$(\star_6)$]
	\item If $c_\ell \notin \mathbb{Q}_p G_0$, then $|\mrm{supp}_p(c_\ell)| \geq 2$.
\end{enumerate}
[Why? Recall that $c_\ell = \sum_{j \in u_\ell} r_j (x_j+y_j)$ and let $Y_\ell = \bigcup \{ \{x_j, y_j\} : j \in u_\ell\}$ and, for $z \in Y_\ell$, let $a_z = \sum \{r_j : j \in u_\ell, \,  x_j = z \} + \sum \{r_j : j \in u_\ell, \, y_j = z \}$. Now we can apply \ref{hyp_C1}(\ref{item17_new}) with $(p, 2, (x, y), ((x_j, y_j) : j \in u_\ell), (1, 1), (r_j : j \in u_\ell), (a_z : z \in Y_\ell))$ here standing for $(p, k, \bar{x}, \mathbf{y}, \bar{r}, \bar{a}_{(\mathbf{y}, r)})$ there and get $|\{z \in Y_\ell : a_z \notin \mathbb{Q}_p\}| \neq 1$. But this means that $|\mrm{supp}_p(c_\ell)| \neq 1$, but $|\mrm{supp}_p(c_\ell)| \neq 0$ as $c_\ell \notin \mathbb{Q}_p G_0$, hence $|\mrm{supp}_p(c_\ell)| \geq 2$, as promised. This concludes the proof of $(\star_6)$.]
\begin{enumerate}[$(\star_7)$]
	\item $V = \{\ell < \ell_* : c_\ell \notin \mathbb{Q}_p G_0\}$ has exactly one member.
\end{enumerate}
[Why? If $V = \emptyset$, then the RHS of $(\star_4)$ is in $\mathbb{Q}_p G_0$ but not the LHS, recalling $(\star_5)$ and the choice of $m < \omega$ in $(\star_{1.2})$, a contradiction. On the other hand, if $|V| \geq 2$, then the RHS of $(\star_3)$ has $p$-support of size $\sum_{\ell < \ell_*} |\mrm{supp}_p(c_\ell)| \geq 2|V| > 2$, but the $p$-support of the LHS of $(\star_3)$ has cardinality $2$, a contradiction.]
\newline Let $k$ be the unique member of $V$, then we have the following:
%$$\{\pi_1(x), \pi_1(y)\} = \mrm{supp}_p(p^{-m}q_x \pi_1(x) + p^{-m}q_y \pi_1(y)) = \mrm{supp}_p(\sum_{\ell < \ell_*} c_\ell) = \mrm{supp}_p(c_k) \subseteq X_{t_k/E^M_i}.$$ 
$$\begin{array}{rcl}
\{\pi_1(x), \pi_1(y)\}    & = & \mrm{supp}_p(p^{-m}q_x \pi_1(x) + p^{-m}q_y \pi_1(y)) \\
		& = & \mrm{supp}_p(\sum_{\ell < \ell_*} c_\ell) \\
		& = & \mrm{supp}_p(c_k) \subseteq X_{t_k/\mathfrak{E}^M_i}.
\end{array}$$
So $\pi_1(x), \pi_1(y) \in X_{t_k/\mathfrak{E}^M_i}$ and as $X_{t_k/\mathfrak{E}^M_i}$ is an $\mathcal{E}_i$-equivalence class (by the definition of $\mathcal{E}_i$), then $\pi_1(x) \mathcal{E}_i \pi_1(y)$, as desired. This concludes the proof of the claim.
\end{proof}

	\begin{cclaim} There is a bijection $h: \mathcal{U} \rightarrow \mathcal{V}$ preserving $\mathfrak{E}^M_i$ and $\neg \mathfrak{E}^M_i$, \mbox{for all $i < 3$.}
\end{cclaim}

	\begin{proof} By \ref{the_hammer_claim}(\ref{the_hammer_claim_item2}), we have:
	\begin{enumerate}[$(*_1)$]
	\item If $x, y \in X_{\mathcal{U}}$ and $i < 3$, then $x \mathcal{E}_i y \Leftrightarrow \pi_1(x) \mathcal{E}_i \pi_1(y)$.
	\end{enumerate}
Now apply $(*_1)$ for $i = 2$ and recall that by \ref{hyp_C0}(\ref{hyp_C0_item1}) $\mathfrak{E}^M_2$ is equality on $M$. Then:
\begin{enumerate}[$(*_2)$]
	\item $\exists s \subseteq_1 \mathcal{U} (x, y \in X'_s) \Leftrightarrow \exists t \subseteq_1 \mathcal{V} (\pi_1(x), \pi_1(y) \in X'_t)$.
	\end{enumerate}	
Now, as $X_{\mathcal{U}} = \bigcup_{s \subseteq_1 \mathcal{U}} X'_s$ and $X_{\mathcal{V}} = \bigcup_{s \subseteq_1 \mathcal{V}} X'_s$, there is a function $\mathbf{h}_1$ from $\mathcal{U}$ into  $\mathcal{V}$ such that (not distinguishing $a \in \mathcal{U}$ with $\{a\} \subseteq_1 \mathcal{U}$):
\begin{enumerate}[$(*_3)$]
	\item If $x \in X'_s$, $s \subseteq_1 \mathcal{U}$, then $\pi_1(x) \in X'_{\mathbf{h}_1(s)}$.
	\end{enumerate}	
As $\pi_2 = \pi_1^{-1}$ and $\pi_2$ is a function from $X_{\mathcal{V}}$ onto $X_{\mathcal{U}}$ (cf. \ref{concl_sequence_rationals_Borel}) we have that:
\begin{enumerate}[$(*_4)$]
	\item $\mathbf{h}_1: \mathcal{U} \rightarrow \mathcal{V}$ is $1$-to-$1$ and onto.
	\end{enumerate}
Finally, applying \ref{the_hammer_claim}(\ref{the_hammer_claim_item2}) to $i$ and recalling the definition of $\mathcal{E}_i$ we get:
	\begin{enumerate}[$(*_5)$]
	\item For $i = 0, 1$, $a \neq b \in \mathcal{U} \text{ implies } a \mathfrak{E}^M_i b \Leftrightarrow \pi_1(a) \mathfrak{E}^M_i \pi_1(b)$.
	\end{enumerate}
\end{proof}

	\begin{conclusion}\label{the_iso_conclusion} $M \restriction \mathcal{U}$ and $M \restriction \mathcal{V}$ are isomorphic members of $\mathbf{K}^{\mrm{eq}}$.
\end{conclusion}

%	\begin{proof} Items (1)-(3) are OK. Item (4) follows from item (3) and Hypothesis~\ref{hyp_C1}(\ref{item12}).
%\end{proof}

%\begin{cclaim}\label{claim_q*_Borel} In the context of Conclusion~\ref{concl_sequence_rationals_Borel}(1). For some $q_* \in \mathbb{Q} \setminus \{ 0 \}$ we have that, for every $x \in \mathcal{U}$, $qx = q_*$.
%\end{cclaim}
%
%
%	\begin{cclaim} In the context of Claim~\ref{claim_q*_Borel}, $q_* \in \{+1, -1 \}$.
%\end{cclaim}

    In a work in preparation (among other things) we intend to give a characterization of the automorphism groups of the groups $G_{(1, \mathcal{U})}$ that we construct above.

\subsection{The Proof of the Main Theorem}

	Notice that in this subsection Hypothesis~\ref{hyp_non_iso} is no longer assumed.

	\begin{conclusion}\label{crucial_Borel_conclusion} Let $\mathfrak{m}[M] \in \mathrm{K}^{\mrm{bo}}_3$, $\mathcal{U}, \mathcal{V} \subseteq M$ and $|\mathcal{U}| = |\mathcal{V}| = \aleph_0$. Then:
\begin{equation}\tag{$\star$} M \restriction \mathcal{U} \cong M \restriction \mathcal{V} \Leftrightarrow G_{(1, \mathcal{U})}[\mathfrak{m}] \cong G_{(1, \mathcal{V})}[\mathfrak{m}].
\end{equation} %Further, if $\mathcal{U}$ is rigid, then $G_{(1, \mathcal{U})}[\mathfrak{m}] $ is rigid (i.e. $\mrm{Aut}(G_{(1, \mathcal{U})}) = \{\mathrm{id}, -\mrm{id}\})$.
\end{conclusion}

	\begin{proof} If the LHS of $(\star)$ holds, then by \ref{why_iso} also the RHS of $(\star)$ holds. If the RHS of $(\star)$ holds, then the assumptions in \ref{hyp_non_iso} are fulfilled and thus \ref{the_iso_conclusion} holds, and so the LHS of $(\star)$ holds.
\end{proof}

\begin{convention}\label{remark_sugar} In Fact~\ref{borel_structure} and Notation~\ref{graph_not}(\ref{graph_not_5}) instead of considering structures with domain $\omega$ we could have considered structures with domain an infinite subset of $\omega$. We take the liberty of not distinguishing between these two variants. This happens most notably in the Proof of Main Theorem right below.
\end{convention}

Recall:

	\begin{fact}\label{nice_special_fact} The class $\mathbf{K}^{\mrm{eq}}_\omega$ is Borel complete. In fact, there is a continuous map from $\mathrm{Graph}_{\omega}$ into $\mathbf{K}^{\mrm{eq}}_\omega$ which preserves isomorphism and its negation.
%\todov{If we want to keep $\mathrm{Graph}_{\omega}$ (which is nice for various reasons, e.g. the Quanta article) into the statement of the Main Theorem we have to strengthen this to:
%\newline - furthermore, this map $F$ preserves automorphism groups, i.e., $\mrm{Aut}(F(H)) \cong \mrm{Aut}(H)$.}
%\todob{To me it seems that this is not hard to achieve, but I doubt that it would be appropriate to write the details, here. Do you agree that this is achievable?}
\end{fact}

	\begin{proof} See e.g. \cite[pg.~295]{monk}.
\end{proof}

	\begin{proof}[Proof of Main Theorem] Let $M$ be as in Hypothesis~\ref{hyp_C0}. Fix $\mathfrak{m} \in \mathrm{K}^{\mrm{bo}}_3(M)$ (cf. Fact~\ref{K2bo_non_empty+}) and assume without loss of generality that $G_1[\mathfrak{m}]$ has set of elements $\omega$. For every $H \in \mathbf{K}^{\mrm{eq}}_\omega$ we define $F[H] : H \rightarrow M$ by defining $F[H](n)$ by induction on $n < \omega$ as the minimal $k < \omega$ such that $\{(\ell, F[H](\ell)) : \ell < n \} \cup \{(n, k)\}$ is an isomorphism from $H \restriction (n+1)$ onto $M \restriction (\{F[H](\ell) : \ell < n \} \cup \{k\})$. The map $H \mapsto M \restriction \{F[H](n) \, : \, n < \omega \}$ is clearly continuous. We will show that the map $F': M	\restriction \mathcal{U} \mapsto G_{(1, \mathcal{U})}[\mathfrak{m}]$, for $\mathcal{U} \subseteq M$ infinite, is also continuous (cf.~\ref{remark_sugar}), thus concluding that the map $\mathbf{B} := F' \circ F : H \mapsto G_{(1, \{F[H](n) \, : \, n < \omega \})}[\mathfrak{m}]$
is a continuous map from $\mathbf{K}^{\mrm{eq}}_\omega$ into $\mrm{TFAB}_\omega$ (cf.~\ref{remark_sugar}), so, by \ref{crucial_Borel_conclusion} and \ref{nice_special_fact}, we are done.

\smallskip
\noindent In order to show that $F'$ is continuous, first recall that $\mathfrak{m} \in \mathrm{K}^{\mrm{bo}}_3$ is fixed (cf. \ref{item_primes}), and so in particular $\bar{p}$ is fixed. Now, given $a \in G_1[\mathfrak{m}]$, we have to compute from $\mathcal{U}$ whether $a \in G_{(1, \mathcal{U})}[\mathfrak{m}]$ or not. To this extent, let $a = \sum \{q^a_\ell x^a_\ell : \ell < n\}$ with the $x_\ell$'s pairwise distinct and $q_\ell \in \mathbb{Q}^+$. Now, as by \ref{hyp_C1}(\ref{def_XU}), $X_{\mathcal{U}} = \bigcup\{X_s : s \subseteq_\omega \mathcal{U} \} = \bigcup\{X'_s : s \subseteq_1 \mathcal{U} \}$ and the latter is a partition of $X$, for every $\ell < n$, there is a unique finite $s^a_\ell \subseteq M$ such that the following conditions holds:
$$a \in G_{(1, \mathcal{U})}[\mathfrak{m}] \Leftrightarrow \bigwedge_{\ell < n} s^a_\ell \subseteq \mathcal{U}.$$
This suffices to show continuity of $F'$, thus concluding the proof of the theorem.
\end{proof}

	\begin{remark} We observe that in the context of the Proof of Main Theorem we can choose both $M$ and $\mathfrak{m}$ to be computable stuctures, in the sense of computable model theory, i.e., all the relations and functions of the structure are computable.
\end{remark}

\section{Completeness of Endorigid Torsion-Free Abelian Groups}\label{sec_cohop}

	The aim of this section is to show that deciding whether a group $G \in \mrm{TFAB}_\omega$ is endorigid is a complete co-analytic problem. We do this by reducing a well-known problem to the endorigidity problem, namely the problem of deciding whether a tree with domain $\omega$ has an infinite branch, which is well-known to be complete co-analytic. So the idea here is to code a tree $T$ into a $\mrm{TFAB}_\omega$ $G[T]$. The way we code the tree $T$ is reminiscent of the coding used for the proof of the Main Theorem. Also in this case $X$ will be a basis of the group $G[T]$, and we will code an element $t \in T$ via a partial permutation $f_t$ of $X$. As in Section~\ref{Borel_complete_section}, the group $G[T]$ that we wish to construct will interpolate between between $G_0 = \bigoplus \{ \mathbb{Z}x : x \in X\}$ and $G_2 = \bigoplus \{ \mathbb{Q}x : x \in X\}$, via a set of tailored divisibility conditions which code the behaviour of the partial maps $f_t$'s, which in turn code the elements of the tree $T$.
	
\smallskip
\noindent
	
	In \ref{hyp_A1}-\ref{observation_Xtree} we deal with the combinatorial part, then we will define the groups.

%	\begin{fact}[{\cite[Proposition~2.2, pg. 130]{gobel}}]\label{fact_cohop} For $G \in \mathrm{TFAB}$, $G$ is co-Hopfian iff $G$ is divisible and of finite rank, i.e., $G$ is a finitely dimensional vector space over $\mathbb{Q}$.
%\end{fact}
%
%	\begin{conclusion} The co-Hopfian groups in $\mrm{TFAB}_\omega$ form a Borel subset of $\mrm{TFAB}_\omega$.
%\end{conclusion}
%
%	On the other hand, we will show below that there are variations on the notion of co-hopfianity (cf. Definition~\ref{def_cohop}) which give a completely different answer.

	\begin{hypothesis}\label{hyp_A1} Throughout this section the following hypothesis stands:
	\begin{enumerate}[(1)]
	\item\label{omega_levels} $T = (T, <_T)$ is a rooted tree with $\omega$ levels and we denote by $\mathrm{lev}(t)$ the level of~$t$;
	\item $T = \bigcup_{n < \omega} T_n$, $T_n \subseteq T_{n+1}$, and $t \in T_n$ implies that $\mathrm{lev}(t) \leq n$;
	\item\label{hyp_A1_item3} $T_0 = \emptyset$, $T_n$ is finite, and we let $T_{<n} = \bigcup_{\ell < n} T_\ell$ (so $T_{<(n+1)} = T_n$);
	\item if $s <_T t \in T_{n+1}$, then $s \in T_n$;
	\item $T$ is countable. 
\end{enumerate}
\end{hypothesis}

	\begin{definition}\label{hyp_A2} Let $\mathrm{K}^{\mrm{ri}}_1(T)$ be the class of objects:
	$$\mathfrak{m}(T) = \mathfrak{m} = (X, X^T_n, \bar{f}^T: n < \omega) = (X, X_n, \bar{f} : n < \omega)$$
satisfying the following requirements:
	\begin{enumerate}[(a)]
	\item\label{X0_not_empty} $X_0 \neq \emptyset$ and, for $n < \omega$, $X_n$ is finite and $X_n \subsetneq X_{n+1}$, and $X_{< n} = \bigcup_{\ell < n} X_\ell$;
	\item $\bar{f} = (f_t : t \in T)$;
	\item\label{def_of_f_t} if $n >0$ and $t \in T_n \setminus T_{< n}$, then $f_t$ is a one-to-one function from $X_{n-1}$ into $X_n$;
	\item\label{X0_cap_ft} for every $t \in T$, $X_0 \cap \mathrm{ran}(f_t) = \emptyset$;
	\item\label{item_fs_sub_ft} if $s \leq_T t \in T_n$, then $f_s \subseteq f_t$; 
	\item\label{item_no_fix_point} if $t \in T_{n+1} \setminus T_{n}$, $f_t(x) = y$ and $y \in X_{n}$, then for some $s <_T t$, $x \in \mrm{dom}(f_s)$;
	\item\label{item_meet_tree} if $s, t \in T_n$ and $y \in \mrm{ran}(f_s) \cap \mrm{ran}(f_t)$, then for some $r \in T_n$ such that $r \leq_T s, t$ we have that $y \in \mrm{ran}(f_r)$, equivalently, $\mrm{ran}(f_s) \cap \mrm{ran}(f_t) = \mrm{ran}(f_r)$, for $r = s \wedge t$, where $\wedge$ is the natural semi-lattice operation taken in the tree $(T, <_T)$;
	\item\label{itemh_co} $X_{n+1} \supsetneq \bigcup \{\mrm{ran}(f_t) : t \in T_{n+1} \setminus T_n\} \cup X_n$;
	\item we let $X = X^{\mathfrak{m}} = \bigcup_{n < \omega} X_n$;
	\item\label{hyp_A2_last} if $f_s(x) = f_t(x)$ and $x \in X_n \setminus X_{<n}$, then we have the following:
	$$f_s \restriction X_n = f_t \restriction X_n \text{ and } X_n \subseteq \mrm{dom}(f_s) \cap \mrm{dom}(f_t).$$
\end{enumerate}
\end{definition}

	\begin{notation}\label{n(x)} For $x \in X$, we let $\mathbf{n}(x)$ be the unique $n < \omega$ such that $x \in X_n \setminus X_{< n}$ (so e.g. $x \in X_0$ implies $\mathbf{n}(x) = 0$).
\end{notation}

\begin{convention}\label{the_m_convention} $\mathfrak{m} = (X, X_n, \bar{f}: n < \omega) \in \mathrm{K}^{\mrm{ri}}_1(T)$ (cf. \ref{hyp_A2} and \ref{K_co_non-empty} below).
\end{convention}

	\begin{observation}\label{observation_new_elements_are_moved} %Notice that from items (d) and (f) of Definition~\ref{hyp_A2} it follows that f
	In the context of Definition~\ref{hyp_A2}, we have: 
	\begin{enumerate}[(1)]
	\item If $m < n < \omega$, $t \in T_n \setminus T_{<n}$ and for every $s <_T t$ we have $s \in T_{m}$, then:
$$(X_{n-1} \setminus X_{m}) \cap \mathrm{ran}(f_t) = \emptyset.$$ 
	\item\label{obs_item2} If $t \in T$, then for every $x \in \mrm{dom}(f_t)$ we have that $x \neq f_t(x)$, moreover there is a unique $0 < n < \omega$ such that $x \in X_{n-1}$ and $f_t(x) \in X_n \setminus X_{n-1}$, and for some $s \in T_n \setminus T_{< n}$ we have $s \leq_T t$ and $f_s(x) = f_t(x)$.
\end{enumerate}
\end{observation}

\begin{proof} We prove (1), by Definition~\ref{hyp_A2}(\ref{def_of_f_t}) we know that $f_t$ is one-to-one from $X_{n-1}$ into $X_n$. If $n = 1$, then $m = 0$ and so $X_{n-1 } = X_0 = X_m$, thus the conclusion is trivial. 
Suppose then that $n > 1$ and let $y \in (X_{n-1} \setminus X_{m}) \cap \mathrm{ran}(f_t)$, and let $x \in \mrm{dom}(f_t)$ be such that $f_t(x) = y$. Then, by Definition~\ref{hyp_A2}(\ref{item_no_fix_point}) there exists $s <_T t$ such that $x \in \mrm{dom}(f_s)$. But then, using the assumption in (1), we have that $s \in T_{m}$ (so $m = 0$ is impossible by Definition~\ref{hyp_A1}(\ref{hyp_A1_item3})).  Hence, by Definition~\ref{hyp_A2}(\ref{def_of_f_t}), $\mrm{ran}(f_s) \subseteq X_m$, so $y = f(x) \in X_m$, contradicting the fact that $y \in (X_{n-1} \setminus X_{m})$.

\smallskip
\noindent We prove (2). Assume that $x$, $t$, and thus also $f_t$, are fixed and $x \in \mrm{dom}(f_t)$. Let $s \leq_T t$ be $\leq_T$-minimal such that $f_s(x)$ is well-defined, and let $n < \omega$ be such that $s \in T_n \setminus T_{<n}$ (notice that $n \geq 1$ since $T_0 = \emptyset$). Clearly, there is unique $m < \omega$ such that $x \in X_m \setminus X_{< m}$. As $x \in \mrm{dom}(f_s)$ and $s \in T_n \setminus T_{< n}$ necessarily $m < n$, so $x \in X_{< n}$. But by the choice of $s$ we have that $r <_T s \text{ implies } x \notin \mrm{dom}(f_r)$.
By the last two sentences and Def.~\ref{hyp_A2}(\ref{item_no_fix_point}) we have $f_s(x) \in X_n \setminus X_{< n}$, but $f_t(x) = f_s(x)$.
\end{proof}

	\begin{cclaim}\label{K_co_non-empty} For $T$ as in Hypothesis~\ref{hyp_A1}, $\mathrm{K}^{\mrm{ri}}_1(T) \neq \emptyset$ (cf. Definition~\ref{hyp_A2}).
\end{cclaim}

	\begin{proof} Straightforward.
\end{proof}

	\begin{definition}\label{def_reasonable} On $X$ (cf. Convention~\ref{the_m_convention}) we define:
	\begin{enumerate}[(1)]
	\item for $x \in X$, $\mrm{suc}(x) = \{ f_t(x) : t \in T, x \in \mrm{dom}(f_t) \}$;
	\item\label{def_reasonable_item<} for $x, y \in X$, we let $x <_X y$ if and only if for some $0 < n < \omega$ and $x_0, ..., x_{n} \in X$ we have that $\bigwedge_{\ell < n} x_{\ell + 1} \in \mrm{suc}(x_\ell)$, $x = x_0$ and $y = x_n$;
	\item $\mrm{seq}_k(X) = \{\bar{x} \in \mrm{seq}_k(X) : \bar{x} \text{ is injective} \}$;
	\item\label{def_reasonable_item} we say that $\bar{x} \in \mrm{seq}_k(X)$ is reasonable when the following happens:
	$$n(1) < n(2), \, x_{i(1)} \in X_{n(1)} \setminus  X_{< n(1)}, \, x_{i(2)} \in X_{n(2)} \setminus  X_{< n(2)} \Rightarrow i(1) < i(2);$$
	\item\label{<^j_*}  $<^k_X$ is the partial order on $\mrm{seq}_k(X)$ defined as $\bar{x}^1 <^k_X \bar{x}^2$ if and only if $\bar{x}^1, \bar{x}^2 \in \mrm{seq}_k(X)$ and there are $0 < n < \omega$, $\bar{y}^0, ..., \bar{y}^n \in \mrm{seq}_k(X)$ and $t_0, ..., t_{n-1} \in T$ such that for every $\ell < n$ we have that $f_{t_\ell}(\bar{y}^\ell) = \bar{y}^{\ell+1}$, and $(\bar{x}^1, \bar{x}^2) = (\bar{y}^0, \bar{y}^n)$;
	\item notice that for $k = 1$ we have that $<^k_X = <_X$, where $<_X$ is as in (\ref{def_reasonable_item<}) (ignoring the difference between $x$ and $(x)$, for $x \in X$);
	\item\label{def_E_k_endo} for $k \geq 1$, let $E_k$ be the closure of $\{(\bar{x}, \bar{y}) : \bar{x} <^k_X \bar{y} \}$ to an equivalence relation.
	\end{enumerate}
\end{definition}

	\begin{observation}\label{unique_bart} If $\bar{x}^1 \leq^k_X \bar{x}^2$ (cf. \ref{def_reasonable}(\ref{<^j_*})), then there is a unique $\bar{t}$ such that:
	\begin{enumerate}[(a)]
	\item $\bar{t} \in T^{n}$ for some $n < \omega$;
	\item $f_{\bar{t}}(\bar{x}^1) = \bar{x}^2$, where $f_{\bar{t}} = (f_{t_{n-1}} \circ \cdots \circ f_{t_0})$ and $t_0, ..., t_{n-1}$ are as as in \ref{def_reasonable}(\ref{<^j_*});
	\item for every $\ell < n$, there is no $s <_T t_\ell$ such that $f_s(\bar{y}_\ell)$ is well-defined, where $(\bar{y}_0, ..., \bar{y}_{n-1})$ are as in \ref{def_reasonable}(\ref{<^j_*});
	\item if $\bar{t}' \in T^n$ is as in clauses (a)-(b) above and $\ell < n$, then $t_\ell \leq_T t'_\ell$.
	\end{enumerate}
\end{observation}

	\begin{observation}\label{observation_Xtree}
	\begin{enumerate}[(1)]
	\item $(X, <_X)$ is a tree with $\omega$ levels;
	\item\label{roots_item} every $z \in X_0$ is a root of the tree $(X, <_X)$, further, for every $n < \omega$, some $z \in X_{n+1} \setminus X_n$ is a root of the tree $(X, <_X)$, and so $z/E_1 \cap X_n = \emptyset$;
	\item if $y \in X_{n+1} \setminus X_n$, then for at most one $x \in X_n$ we have $y \in \mrm{suc}(x)$; 
	\item\label{cone} if $y \in \mrm{suc}(x)$, then $\{ t \in T : f_t(x) = y \}$ is a cone of $T$;
	\item\label{wlog_reasonable} if $\bar{x} \in \mrm{seq}_k(X)$, then some permutation of $\bar{x}$ is reasonable (cf. Definition~\ref{def_reasonable}(\ref{def_reasonable_item}));
	\item\label{reasonable_implies_reasonable} if $f_t(\bar{x}) = \bar{y}$ and $\bar{x}$ is reasonable, then so is $\bar{y}$;
	\item\label{ktree} for every $1 \leq k < \omega$, $(\mrm{seq}_k(X), <^k_X)$ is a tree with $\omega$ levels;
	\item if $\bar{x} <^k_X \bar{y}$ and $\bar{x}$ is reasonable, then $\bar{y}$ is also reasonable;
	\item\label{observation_Xtree_item10} if $\bar{x} \in \mrm{seq}_k(X)$ is reasonable, $\bar{x} \leq^k_X \bar{y}^1 = (y^1_0, ..., y^1_{k-1})$, $\bar{x} \leq^k_X \bar{y}^2 = (y^2_0, ..., y^2_{k-1})$ and $y^1_{k-1} = y^2_{k-1}$, then $ \bar{y}^1 = \bar{y}^2$;
	\item\label{item_10} for every $t \in T$, $\mrm{dom}(f_t)$ is $X_n$, for some $n < \omega$, and we have:
	$$x, y \in \mrm{dom}(f_t) \wedge \mathbf{n}(x) = \mathbf{n}(y) \; \Rightarrow \; \mathbf{n}(f_t(x)) = \mathbf{n}(f_t(y));$$
	\item like (10) with $\bar{t} \in T^n$, $n \geq 1$, where we let:
	$$f_{\bar{t}} = (f_{t_{n-1}} \circ \cdots \circ f_{t_0});$$
	\item\label{item_12} recalling the notation from (11), if $f_{\bar{t}_1}(x) = f_{\bar{t}_2}(x)$, then $\mrm{lg}(\bar{t}_1) = \mrm{lg}(\bar{t}_2) = k$; moreover, letting $\bar{t}_1 = (t_{(1, \ell)} : \ell < k)$ and $\bar{t}_2 = (t_{(2, \ell)} : \ell < k)$,  if $\bar{t}_1$ is as in \ref{unique_bart}, then we have that $\ell < \mrm{lg}(\bar{t}_1)$ implies that $t_{(1, \ell)} \leq_T t_{(2, \ell)}$.
	\end{enumerate}
\end{observation}

	\begin{proof} Items (1)-(2) are clear, where  (2) is by \ref{hyp_A2}(\ref{itemh_co}). Item (3) is by Definition~\ref{hyp_A2}(\ref{item_no_fix_point})-(\ref{item_meet_tree}). Items (4) and (5) are also easy (and (4) is not used (except in \ref{G0_is_tree}(\ref{G0_is_tree_cone})) but we retain it to give the picture). Item (6) can be proved for $t \in T_n \setminus T_{< n}$ by induction on $n < \omega$.
Finally, (7) and (8) are easy, and (9) is easy to see using \ref{unique_bart} and \ref{hyp_A2}(\ref{hyp_A2_last}). Also clauses (10), (11) are easy, and clause (12) holds by \ref{unique_bart}.
\end{proof}

%	\begin{remark} Notice that for the purpose of proving Theorem~\ref{main_th_Sec1} item (\ref{cone}) of Observation~\ref{observation_Xtree} is not necessary. 
%%	Notice also that the analogous claim in the context of Section~\ref{sec_hopfian_TFAB} (specifically in the context of Claim~\ref{observation_Xtree_ho}) is false.
%\end{remark}
	
	\begin{definition}\label{def_G02_cohopfian} Let $\mathfrak{m} \in \mathrm{K}^{\mrm{ri}}_1(T)$ (i.e. as in Convention~\ref{the_m_convention}).
	\begin{enumerate}[(1)]
	\item Let $G_2 = G_2[\mathfrak{m}]$ be $\bigoplus \{ \mathbb{Q}x : x \in X\}$.
	\item Let $G_0 = G_0[\mathfrak{m}]$ be the subgroup of $G_2$ generated by $X$, i.e., $\bigoplus \{ \mathbb{Z}x : x \in X\}$.
	\item\label{hatf_on_G2} For $t \in T$, let:
	\begin{enumerate}[(a)]
	\item $H_{(2, t)} = \bigoplus \{ \mathbb{Q}x : x \in \mrm{dom}(f_t)\}$;
	\item $I_{(2, t)} = \bigoplus \{ \mathbb{Q}x : x \in \mrm{ran}(f_t)\}$;
	\item\label{hat_f} $\hat{f}^2_t$ is the (unique) isomorphism from $H_{(2, t)}$ onto $I_{(2, t)}$ such that $x \in \mrm{dom}(f_t)$ implies that $\hat{f}^2_t(x) = f_t(x)$ (cf. Definition~\ref{hyp_A2}(\ref{def_of_f_t})).
	\end{enumerate}
	\item For $t \in T$, we define $H_{(0, t)} := H_{(2, t)} \cap G_0$ and $I_{(0, t)} := I_{(2, t)} \cap G_0$;
	\item For $\hat{f}^2_t$ as above, we have that $\hat{f}^2_t [H_{(0, t)}] = I_{(0, t)}$. We define $\hat{f}^0_t$ as $\hat{f}^2_t \restriction H_{(0, t)}$.
	\item\label{def_<*} We define the partial order $<_*$ on $G^+_0 $ by letting $a <_* b$ if and only if $a \neq b \in G^+_0$ and, for some $0 < n < \omega$,  $a_0, ..., a_n \in G_0, a_0 = a, a_n = b$ and:
	$$\ell < n \Rightarrow \exists t \in T(\hat{f}^0_t(a_\ell) = a_{\ell + 1}).$$
	\item For $a = \sum_{\ell < m}q_\ell x_\ell$, with $x_\ell \in X$ and $q_\ell \in \mathbb{Q}^+$, let $\mathrm{supp}(a) = \{ x_\ell : \ell < m \}$.
	\item\label{n[a]} For $a \in G^+_2$, let $\mathbf{n}(a)$ be the minimal $n < \omega$ such that $a \in \langle X_{n}\rangle^*_{G_2}$.
\end{enumerate}
\end{definition}

	While the aim of Definition~\ref{def_G02_cohopfian} should be clear from the explanations given at the beginning of this section, the reader may wonder what is the aim of Lemma~\ref{G0_is_tree} and Claim~\ref{claim9A}. In the crucial proof of this section we will show that given an endomorphism $\pi$ of $G_1$ and $x \in X$ we have that $\pi(x)$ has the form $qy$ for some $y \in Y$ and $q \in \mathbb{Q}$, this requires a detailed analysis of supports, hence \ref{G0_is_tree} and \ref{claim9A}.

	\begin{lemma}\label{G0_is_tree}\begin{enumerate}[(1)]
	\item\label{G0_is_tree_cone} If $\{ t \in T : \hat{f}^2_t(a) = b\} \neq \emptyset$, then it is a cone of $T$.
	\item\label{G0_is_tree_restrictiontoX}  $<_* \restriction X = <_X$ (where $<_X$ is as in Definition~\ref{def_reasonable}(\ref{def_reasonable_item<})).
	\item\label{G0_tree_item} $(G^+_0, <_*)$ is a countable tree with $\omega$ levels (recall Hypothesis~\ref{hyp_A1}(\ref{omega_levels})).
	\item\label{G0_is_tre_item4} If $s \leq_T t$, then $\hat{f}^\ell_s \subseteq \hat{f}^\ell_t$, for $\ell \in \{0, 2 \}$.
	\item\label{n[a]1} If $t \in T$, $\hat{f}^2_t(a) = b$ and $a \in G^+_0$, then $\mathbf{n}(a) < \mathbf{n}(b)$ (cf. Definition~\ref{def_G02_cohopfian}(\ref{n[a]})).
	\item If $a <_* b$ (so $a, b \in G^+_0$), then the sequence $(a_\ell : \ell \leq n)$ from \ref{def_G02_cohopfian}(\ref{def_<*}) is unique.
	\item\label{after_replacing} If $a_1 <_* a_2$, and, for $\ell \in \{ 1, 2 \}$, $a_\ell = \sum_{i < k} q^\ell_ix^\ell_i$, $q^\ell_i \in \mathbb{Q}^+$, $\bar{x}^\ell = (x^\ell_i : i < k) \in \mrm{seq}_k(X)$, then maybe after replacing $\bar{x}^1$ with a permutation of it we have:
$$\bar{x}^0 \leq^k_X \bar{x}^1 \text{ and } q^1_i = q^2_i \;\,(\text{for $i < k$}).$$
%	\begin{enumerate}[(a)]
%	\item $\bar{y}^0 \leq^k_* \bar{y}^1$;
%	\item\label{item_6b} if $y^1_{j(1)} = y^0_{j(0)}$, then $j(0) = j(1)$.
%	\end{enumerate}
\end{enumerate}
\end{lemma}

	\begin{proof} Unraveling definitions, we elaborate only on item (5). As $a \neq 0$, let $a = \sum_{i \leq n} q_i x_i$, $x_i \in X$ with no repetitions, $q_i \in \mathbb{Q}^+$. Let $x_i \in X_{k(i)} \setminus X_{<k(i)}$ and w.l.o.g. $k(i) \leq k(i+1)$, for $i < n$ (cf. Observation~\ref{observation_Xtree}(\ref{wlog_reasonable})). Clearly $a \in \langle X_{k(n)} \rangle^*_{G_2}$ but $a \notin \langle X_{<k(n)} \rangle^*_{G_2}$. As $\hat{f}^2_t(a)$ is well-defined, clearly $\{x_i : i \leq n \} \subseteq \mrm{dom}(f_t)$ and $b = \hat{f}^2_t(a) = \sum_{i \leq n} q_i f_t(x_i)$ and, as $f_t$ is $1$-to-$1$, the sequence $(f_t(x_i) : i \leq n)$ is with no repetitions. By Observation~\ref{observation_new_elements_are_moved}(\ref{obs_item2}) applied with $n$ there as $k(n)$ here, $f_t(x_n) \notin \langle X_{k(n)} \rangle^*_{G_2} $, hence \mbox{we have that  $\mathbf{n}(b) \geq n(f_t(x_n)) > k(n) = \mathbf{n}(a)$.}
%
%\smallskip
%\noindent
%As for (7), for $i < k$, let $s \leq_T t$ be $\leq_T$-minimal such that $x_i \in \mrm{dom}(f_s)$, hence $f_t(x_i) = f_s(x_i)$, and letting $s \in T_m \setminus T_{<m}$ we have that $f_s(x_i) \in X_m \setminus X_{< m}$ and $k(i) < m$. If $k(i) = k(i+1)$, then clearly $f(x_{i+1}) \in X_m \setminus X_{<m}$ and if $k(i) < k(i+1)$, then clearly $f(x_{i+1}) \notin X_m$. In both cases the demand of being reasonable holds.
\end{proof}

	\begin{cclaim}\label{claim9A} If (A), then (B), where:
	\begin{enumerate}[(A)]
	\item
	\begin{enumerate}[(a)]
	\item $a, b_\ell \in G^+_2$, for $\ell < \ell_*$;
	\item $a \leq_* b_\ell$ and the $b_\ell$'s are with no repetitions;
	\item $a = \sum \{ q_i x_i : i < j \}$, $q_i \in \mathbb{Q}^+$;
	\item $\bar{x} = (x_i : i  < j) \in \mrm{seq}_j(X)$ and it is reasonable;
%	\item $n_* = n_*(\bar{x})$ is minimal such that $\{x_i : i < j\} \subseteq X_{n_*}$ and: 
%	$$u_{=}(\bar{x}) = \{i < j : x_i \in X_{< n_*} \}$$
%	$$u_{\neq}(\bar{x}) = \{i < j : x_i \in X_{n*} \setminus X_{< n_*} \};$$
		\end{enumerate}
	\item there are, for $\ell < \ell_*$, $\bar{y}^\ell = (y_{(\ell, i)} : i < j)$ such that:
	\begin{enumerate}
	\item $y_{(\ell, i)} = : y^\ell_i \in X$ and $\bar{x} \leq^j_X \bar{y}^\ell$ (cf. Definition~\ref{def_reasonable}(\ref{<^j_*}));
	\item $b_\ell = \sum \{ q_i y_{(\ell, i)} : i < j \}$, (so the $\bar{y}^\ell$ are pairwise distinct, as the $b_\ell$'s are);
	\item $(y_{(\ell, i)} : i < j) \in \mrm{seq}_j(X)$ and it is reasonable;
	%\item $y_{(1, i_1)} = y_{(2, i_2)}$ implies $i_1 = i_2$;
%	\item\label{item(b)(e)9A} if $j > 1$ and $\ell_* > 1$, then there are at least two $y \in X$ such that:
%	$$|\{ (\ell, i) : \ell < \ell_*, \, i < j \text{ and } y_{(\ell, i)} = y \}| = 1;$$
	\item if $j > 1$ and $\ell_* > 1$, then there are $\ell_1 \neq \ell_2 < \ell_*$ and $i_1, i_2 < j$ such that:
	\begin{enumerate}[(i)]
	\item if $\ell < \ell_*$, $i < j$ and $y_{(\ell, i)} = y_{(\ell_1, i_1)}$, then $(\ell, i) = (\ell_1, i_1)$;
	\item if $\ell < \ell_*$, $i < j$ and $y_{(\ell, i)} = y_{(\ell_2, i_2)}$, then $(\ell, i) = (\ell_2, i_2)$.
	\end{enumerate}
%	\item\label{claim9A_itemBf} $(y_{(\ell, j-1)} : \ell < \ell_*)$ is without repetitions and none of $\{y_{(\ell, i)} : \ell < \ell_*, i < j-1\}$ appears in it (recall that $\bar{x}, \bar{y}^0, ..., \bar{y}^{\ell_*-1}$ are reasonable).
	\end{enumerate}
	\end{enumerate}
\end{cclaim}

\begin{proof} By the definition of $\leq_*$ there are $(y_{(\ell, i)} : i < j, \, \ell < \ell_*)$ and by by \ref{observation_Xtree}(\ref{reasonable_implies_reasonable}) and \ref{G0_is_tree}(\ref{after_replacing}) they satisfying clauses $(a)$-$(c)$ of $(B)$.
Recall that $(\{\bar{y} : \bar{x} \leq^j_X \bar{y}\}, \leq^j_X)$ is a tree (as $(\mrm{seq}_j(X), \leq^j_X)$ is a tree). We now show $(B)$(d). There are two cases.
\newline \underline{Case 1}. $\{\bar{y}^\ell : \ell < \ell_* \}$ is not linearly ordered by $\leq^j_X$.
\newline  Then there are $\ell(1) \neq \ell(2) < \ell_*$ such that $\bar{y}^{\ell(1)}, \bar{y}^{\ell(2)}$ are locally $\leq^j_X$-maximal. So we can choose $i_1, i_2 < j$ such that we have the following:
$$x^{\ell_1}_{i_1} \in X_{\mathbf{n}(b_{\ell_1})} \setminus X_{< \mathbf{n}(b_{\ell_1})} \text{ and } x^{\ell_2}_{i_2} \in X_{\mathbf{n}(b_{\ell_2})} \setminus X_{< \mathbf{n}(b_{\ell_2})},$$
notice that using the assumption that the sequences are reasonable we can choose $i_1 = j-1 = i_2$, see \ref{G0_is_tree}(\ref{n[a]1}) and \ref{observation_Xtree}(\ref{observation_Xtree_item10}). Hence, $\ell_1, \ell_2, i_1, i_2$ are as required for (d). 

\smallskip
\noindent \underline{Case 2}. Not Case 1.
\newline So w.l.o.g. we have that, for every $\ell < \ell_*-1$, $\bar{y}^\ell <^j_X\bar{y}^{\ell+1}$. Now, for $\ell < \ell_*$ and $i < j$, let $n(\ell, i) = \mathbf{n}(y^\ell_i)$. Let then: 
\begin{enumerate}[$(\cdot_1)$]
	\item $i(1) < j$ be such that $i < j$ implies $n(0, i) \geq n(0, i(1))$;
	\item $i(2) < j$ be such that $i < j$ implies $n(\ell_*-1, i) \leq n(\ell_*-1, i(2))$.
\end{enumerate}
Then $(0, i(1))$, $(\ell_*-1, i(2))$ are as required. As, for $\ell < \ell_*$, $\bar{y}^\ell$ is reasonable we can actually choose $i(1), i(2)$ such that  $i(1) = 0$ and $i(2) = j_*-1$. 
\end{proof}

	Now we turn to the groups which we shall actually use, i.e., the groups $G_1 = G_1[T]$ defined in \ref{def_G1_cohopfian}(\ref{def_G1}) below. Our aim is to include among the partial automorphisms of $G_1$ all the maps of the form $\hat{f}_t$, i.e., the maps induced by the $f_t$'s, but we want in addition that $G_1$ is minimal modulo this. So to each $a \in G^+_0$ we assign a prime number $p_a$ and add $p^{-m}a$ to $G_1$, for all $m < \omega$. But in order to respect the $\hat{f}_t$'s, when $a \in \mrm{dom}(\hat{f}_t)$ we have to add also $p^{-m}\hat{f}_t(a)$ to $G_1$, for all $m < \omega$. Of course all the $\hat{f}_s$'s have to respect this, so we add also $p^{-m}\hat{f}_{\bar{t}}(a)$ to $G_1$, for all $m < \omega$, where $\bar{t} = (t_0, ..., t_n)$ and $\hat{f}_{\bar{t}} = \hat{f}_{t_n} \circ \cdots \circ \hat{f}_{t_0}$ (and $\hat{f}_{\bar{t}}(a)$ is well-defined). This is done in \ref{def_G1_cohopfian}. In \ref{lemma_Pb_infinite}-\ref{lemma_pre_main_th} we analyze the groups $G_{(1, p)} = \{a \in G_1 : G_1 \models p^\infty \vert \; a \}$.

	\begin{definition}\label{def_G1_cohopfian} Let $(p_a : a \in G^+_0)$ be a sequence of pairwise distinct primes s.t.:
$$a = \sum_{\ell < k} q_\ell x_\ell, \, q_\ell \in \mathbb{Z}^+, \, (x_\ell : \ell < k) \in \mrm{seq}_k(X) \Rightarrow p_a \not \vert \; q_\ell. 
$$	\begin{enumerate}[(1)] 
	\item\label{P_a} For $a \in G^+_0$, let:
	$$\mathbb{P}^{\leq_*}_a = \{ p_b : b \in G^+_0, b \leq_* a \}.$$ %($|\mathbb{P}_a| < \aleph_0$, by Observation~\ref{G0_is_tree}(2)).
	\item\label{def_G1} Let $G_1 = G_1[\mathfrak{m}] = G_1[\mathfrak{m}(T)] = G_1[T]$ be the subgroup of $G_2$ generated by:
$$\{ m^{-1} a : a \in G^+_0, \; m \in \omega \setminus \{ 0 \} \text{ a power of a prime from $\mathbb{P}^{\leq_*}_a$}\}.$$
	\item\label{def_G1p} For a prime $p$, let $G_{(1, p)} = \{ a \in G_1 : a \text{ is divisible by $p^m$, for every $0 < m < \omega$}\}$ (notice that, by Observation~\ref{generalG1p_remark}, $G_{(1, p)}$ is always a pure subgroup of $G_1$).
	\item For $b \in G^+_1$, let $\mathbb{P}_b = \{p_a : a \in G^+_0, G_1 \models \bigwedge_{m < \omega} p^m_a \vert \, b\}$.
	\item For $t \in T$ and $\ell \in \{0, 1, 2\}$, let:
	$$H_{(\ell, t)} = \langle x : x \in \mrm{dom}(f_t)\rangle^*_{G_\ell} \;\; \text{ and } \;\; I_{(\ell, t)} = \langle x : x \in \mrm{ran}(f_t)\rangle^*_{G_\ell}$$
\end{enumerate}
\end{definition}

	\begin{remark}\label{lemma_Pb_infinite}
	\begin{enumerate}[(1)]
	\item If $a, b \in G^{+}_1$ and $\mathbb{Q}a = \mathbb{Q}b \subseteq G_2$, then $\mathbb{P}_{a} = \mathbb{P}_{b}$. 
	\item\label{lemma_Pb_infinite_3} If $a \leq_* b$, then $\mathbb{P}_a \subseteq \mathbb{P}_b$.
	\end{enumerate}
	\end{remark}

	\begin{proof} Essentially as in Observation~\ref{generalG1p_remark}.
\end{proof}

%	\begin{cclaim}\label{supp_claim_lin_ind} Let $a \in G^+_0$ and $p = p_a$.
%	\begin{enumerate}[(1)]
%	\item $W_p := \{b \in G^+_0 : a \leq_* b\}$ is a linearly indep. subset of $G_2$ as a $\mathbb{Q}$-vector space.
%	\item\label{supp_claim_lin_ind_2} We can choose a sequence $(b_{(a, n)} : n < \omega)$ listing $W_p$ such that for all $n < \omega$:
%	$$\mrm{supp}(b_{(a, n)}) \not\subseteq \bigcup \{\mrm{supp}(b_{(a, k)}) : k < n\}.$$
%	\end{enumerate}
%\end{cclaim}
%
%	\begin{proof} (1) follows from (2) and (2) is clear from the definition of $\mathfrak{m}$ from \ref{hyp_A2}, that is, by \ref{G0_is_tree}(\ref{G0_tree_item}) we can find a list $(b_{(a, n)} : n < \omega)$ such that $b_{(a, k)} <_* b_{(a, n)}$ implies $k < n$. So for such an $n$, $b_{(a, n)}$ is $<_*$-maximal in $\{b_{(a, k)} : k \leq n\}$, and, as in the proof of \ref{claim9A}, we have $\mrm{supp}(b_{(a, n)}) \not\subseteq \bigcup \{\mrm{supp}(b_{(a, k)}) : k < n\}$.
%\end{proof}
%
%	\begin{choice}\label{the_choice_ban} For every $a \in G^+_0$ choose $(b_{(a, n)} : n < \omega)$ as in \ref{supp_claim_lin_ind}(\ref{supp_claim_lin_ind_2}). 
%\end{choice}

	Here we look more deeply at $G_1$. The crucial point is that any endomorphism of $G_1$ maps $G_{(1, p)} = \{a \in G_1 : \text{ for all } m < \omega, \; p^m \vert \, a \}$ into itself, and so the following characterization of $G_{(1, p)}$ will allow us to reconstruct information on the action of the $f_t$'s on $X$, and thus eventually to reconstruct the tree $T$, to some extent.

	\begin{lemma}\label{lemma_pre_main_th}\begin{enumerate}[(1)]
	\item\label{lemma_pre_main_th_item0} For $b \in G^+_0$ we have that $\mathbb{P}^{\leq_*}_b = \mathbb{P}_b$.
	\item\label{G1p_pure_co-hop} If $p = p_a$, $a \in G^+_0$, then:
	$$G_{(1, p)} = \langle b \in G^+_0 : a \leq_* b \rangle^*_{G_1}.$$
%i.e. the smallest pure subgroup of $G_1$ containing $\{b \in G^+_0 : a \leq_* b \}$.
	\item For $t \in T$, $H_{(1, t)} := H_{(2, t)} \cap G_1$ and $I_{(1, t)} := I_{(2, t)} \cap G_1$ are pure in $G_1$.
	\item\label{convention_hatf_G1} For $\hat{f}^{2}_t$ as in Definition~\ref{def_G02_cohopfian}(\ref{hat_f}), $\hat{f}^{2}_t [H_{(1, t)}] \subseteq I_{(1, t)}$. We define:
	$$\hat{f}^{1}_t = \hat{f}^{2}_t \restriction H_{(1, t)},$$
and for $\bar{t}$ a finite sequence of members of $T$ we let:
	$$\hat{f}^{1}_{\bar{t}} = (\cdots \circ \hat{f}^{1}_{t_\ell}\circ \cdots).$$
	\item\label{item4} $\hat{f}^1_t \restriction H_{(1, t)} = \hat{f}^2_t \restriction H_{(1, t)}$ is into $I_{(1, t)} \leq G_1$ but it is not onto $I_{(1, t)}$.
	\item\label{item5} Assume $a = \sum_{\ell < k} q_\ell x_\ell \in G_0$, $k > 0$, $x_\ell \in X$, $q_\ell \in \mathbb{Q}^+$, $\bar{x} = (x_\ell : \ell < k) \in \mrm{seq}_k(X)$ and $p = p_a$. If $b \in G_{(1, p)}$, then there are $j >0$, $m>0$, and, for $i < j$, $\bar{y}^i$, $b_i$ and  $q'_i \in \mathbb{Q}^+$ such that the following conditions are verified:
	\begin{enumerate}[(a)]
	\item for $i < j$, $\bar{x} \leq^k_X \bar{y}^i$;
	\item $(b_i = \sum_{\ell < k} q_\ell y^i_\ell : i < j)$ is linearly independent;
	\item $b = \sum \{q'_i b_i : i < j \}$;
	\item for $i < j$, $m a \leq_* m b_i$.
	\end{enumerate}
\end{enumerate}
\end{lemma}

	\begin{proof} Item (1) is easy. We prove item (2). The RTL inclusion is clear by \ref{def_G1_cohopfian}(2). We prove the other implication. To this extent:
		\begin{enumerate}[$(*_1)$]
	\item $a \in G^+_0$, $p = p_a$ and we let $W_p := \{b \in G^+_0 : a \leq_* b\}$.
\end{enumerate}
We claim:
	\begin{enumerate}[$(*_2)$]
	\item $W_p$ is a linearly independent subset of $G_2$, as a $\mathbb{Q}$-vector space.
\end{enumerate}
[Why $(*_2)$? Let $k \geq 1$, $\bar{x} \in \mrm{seq}_k(X)$, $\bar{q} \in (\mathbb{Z}^+)^k$ and $a = \sum \{q_\ell x_\ell : \ell < k\}$ (recall that $a \in G^+_0$).
W.l.o.g. $\bar{x}$ is reasonable. Now, toward contradiction, assume that $n \geq 1$, $b_i \in W_p$, for $i < n$, $(b_i : i < n)$ is without repetitions and there are $q^i \in \mathbb{Q}^+$, for $i < n$, such that:
\begin{enumerate}[$(*_{2.1})$]
	\item $\sum \{q^i b_i : i < n\} = 0$. 
\end{enumerate} 
For each $i < n$, let $b_i = \sum \{q_\ell x_{(i, \ell)} : \ell < k\}$, where $\bar{x} \leq^k_X \bar{x}_i = (x_{(i, \ell)} : \ell < k)$. As $a \in G^+_1$, clearly $n > 1$, and by \ref{observation_Xtree}(\ref{ktree})
 there is $i_* < n$ such that $\bar{x}_{i_*}$ is $<^k_X$-maximal in $\{\bar{x}_i : i < n\}$. As $\bar{x}$ is reasonable, so is $\bar{x}_{i_*}$ and so $x_{(i_*, n-1)}$ appears exactly once in $(*_{2.1})$, so recalling $q^{n-1} \in \mathbb{Q}^+$ we get a contradiction, and so $(*_2)$ holds indeed.]
\begin{enumerate}[$(*_{3})$]
	\item Let $\mathcal{U} \subseteq X$ be such that:
	\begin{enumerate}[(a)]
	\item if $y \in \mathcal{U}$, then $y \notin \langle W_p\rangle^*_{G_2}$;
	\item $\mathcal{U} \cup W_p$ is linearly independent;
	\item under conditions (a), (b), $\mathcal{U}$ is maximal.
	\end{enumerate}
\end{enumerate} 
Clearly $\mathcal{U}$ is well-defined, and we have:
\begin{enumerate}[$(*_{4})$]
	\item 
	\begin{enumerate}[(a)]
	\item the disjoint union $\mathcal{U} \cup W_p$ is a basis of $G_2$, as a $\mathbb{Q}$-vector space;
	\item let $h \in \mrm{End}(G_2)$ be s.t. $h \restriction \mathcal{U}$ is the identity and $h(a) = 0$, for all $a \in W_p$.
	\end{enumerate}
\end{enumerate} 
Now we define:
\begin{enumerate}[$(*_{5})$]
	\item 
	\begin{enumerate}[(a)]
	\item $G'_1 := (\sum \{\mathbb{Q} y : y \in \mathcal{U}\}) + G_1$;
	\item $G''_1 = \sum \{\mathbb{Q}_p y : y \in \mathcal{U}\} + G_1$.
\end{enumerate}
\end{enumerate}
Also, we have:
\begin{enumerate}[$(*_{6})$]
	\item 
	\begin{enumerate}[(a)]
	\item $h \restriction G''_1 \in \mrm{End}(G''_1)$;
	\item if $d \in G''_1$ and $G''_1 \models p^{\infty} \, \vert \, d$, then $d = 0$;
	\item  $G'_1 = G''_2$.
\end{enumerate}
\end{enumerate} 
[Why $(*_{6})$)? Concerning clause (a), we just have to prove that if $b \in G^+_0$, $p' \in \mathbb{P}_b$, so $p' = p_d$, for some $d \leq_* b$, then, for every $m < \omega$, $h(p^{-m}_db) = p_d^{-m}h(b) \in G''_1$. 
Now, if $d = a$, then $p_d = p$ and $a \in W_p$, hence $h(b) = 0$, so fine. If on the other hand $d \neq a$, then $p_d \neq p$. Notice that the support of $h(b)$ is a subset of $\mathcal{U}$. Now, $\{b' \in G''_1 : \mrm{supp}(b') \subseteq \mathcal{U}\} = \bigoplus \{\mathbb{Q}_p x : x \in \mathcal{U}\}$, which is $p_d$-divisible, so clause (a) holds indeed. Finally clauses (b) and (c) hold by the definitions of $\mathbb{Q}_p$ and $G'_1$.]
\newline Now, let $c$ be any member of $G_{(1, p)}$. As $h \restriction G''_1 \in \mrm{End}(G''_1)$ clearly $h(c) \in G''_1$ and as $m < \omega$ implies $p^{-m} c \in G_{(1, p)}$ clearly $G''_1 \models p^\infty \vert \, h(c)$. By $(*_6)$(b), it follows that $h(c) = 0$, but this implies that $c$ belongs to the kernel of $h \restriction G_1$, which is $\langle W_p \rangle^*_{G_1}$. As $c$ was any member of $G_{(1, p)}$ we are done.
This concludes the proof of item (2).

\smallskip
\noindent Concerning item (3), notice:
$$H_{(1, t)} = \langle \mathbb{Z}x : x \in \mrm{dom}(f_t) \rangle^*_{G_1},$$
$$I_{(1, t)} = \langle \mathbb{Z}x : x \in \mrm{ran}(f_t) \rangle^*_{G_1}.$$
Item (4) is by item (2) and the following observation, if $f_t(x) = y$, then  we have $x \leq_* y$ (recall \ref{def_reasonable}(\ref{def_reasonable_item<})), and so $\mathbb{P}_x \subseteq \mathbb{P}_y$ (cf. \ref{lemma_Pb_infinite}(\ref{lemma_Pb_infinite_3})). Concerning item (5), assume that $0 < n < \omega$, $t \in T_n \setminus T_{< n}$, $x \in X_{n-1} \setminus X_{< n-1}$ and let $y = f_t(x) \in X_n \setminus X_{< n}$ (cf. Observation~\ref{observation_new_elements_are_moved}), notice that in particular $x <_* y$.
\noindent So $p_y$ is well-defined, since $y \in G^+_0$, and we have the following:
\begin{enumerate}[(a)]
	\item $G_1 \models p_y \not| \; x$, and so $H_{(1, t)} \models p_y \not| \; x$ (as $H_{(1, t)}$ is pure in $G_1$, cf. item (3));
	\item $G_1 \models \bigwedge_{m < \omega} p^m_y \, | \, y$.
\end{enumerate}
[Why (b)? By the definition of $G_1$. Why (b)? Recalling that $x <_* y$.]
\newline But then, since by item (4), $\hat{f}_t \restriction H_{(1, t)}$ is an embedding of $H_{(1, t)}$ into $I_{(1, t)}$ we have that $\hat{f}_t[H_{(1, t)}] \models p_y \not\!| \, f(x) \wedge  f(x) = y$. On the other hand, since $I_{(1, t)}$ is pure in $G_1$ (cf. (3) of this lemma) we have that for every $m < \omega$, $p^{-m}_y y \in I_{(t, 1)}$ (cf. ~\ref{obs_pure_TFAB}). Finally, item (6) is by clause (2) and unraveling definitions.
\end{proof}

	We now prove the main theorem of this section, namely \ref{main_th_Sec1}. Notice that in \ref{main_th_Sec1}(2) below we prove more than needed in order to show that the set of endorigid groups in $\mrm{TFAB}_\omega$ is complete co-analytic, as, in combination with \ref{main_th_Sec1}(1) and \ref{main_th_Sec1}(3), it would suffice to show that if $T$ is well-founded then there is an endomorphism of $G_1$ which is not multiplication by an integer, we show that in addition such an endomorphism can be taken to be $1$-to-$1$ and s.t. $G_1/f[G_1]$ is not torsion.

	\begin{theorem}\label{main_th_Sec1} Let $\mathfrak{m}(T) \in \mathrm{K}^{\mrm{ri}}_1(T)$.
	\begin{enumerate}[(1)]
	\item We can modify the construction so that $G_1[\mathfrak{m}(T)] = G_1[T]$ has domain $\omega$ and the function $T \mapsto G_1[T]$ is Borel (for $T$ a tree with domain $\omega$).
	\item\label{non-well-found_co} If $T$ is not well-founded, then $G_1[T] = G_1$ has a $1$-to-$1$ $f \in \mrm{End}(G_1)$ which is not multiplication by an integer and such that $G_1/f[G_1]$ is not torsion.
	\item If $T$ is well-founded, then $G_1[T]$ is endorigid.
\end{enumerate}
\end{theorem}

	\begin{proof} Item (1) is easy. We prove item (2). Let $(t_n : n < \omega)$ be a strictly increasing infinite branch of $T$. By Lemma~\ref{G0_is_tree}(\ref{G0_is_tre_item4}), $(\hat{f}^2_{t_n} : n < \omega)$ is increasing, by Definition~\ref{def_G02_cohopfian}(\ref{hat_f}), $\hat{f}^2_{t_n}$ embeds $H_{(2, t_n)}$ into $I_{(2, t_n)}$, thus $\hat{f}^2 = \bigcup_{n < \omega} \hat{f}^2_{t_n}$ is an embedding of $G_2$ into $\bigcup_{n < \omega} H_{(2, t_n)}$. Now, $(H_{(2, t_n)} : n < \omega)$ is a chain of pure subgroups of $G_2$ with limit $G_2$, because, recalling \ref{hyp_A2}(\ref{item_fs_sub_ft}), we have that:
	$$H_{(2, t_n)} = \mrm{dom}(f_{t_n}) \subseteq \mrm{dom}(f_{t_{n+1}}) \subseteq H_{(2, t_{n+1})}$$
and by \ref{hyp_A2}(c) we have that $\bigcup_{n < \omega} H_{(2, t_n)} = G_2$.
 Thus $\hat{f}^1 := \hat{f} \restriction G_1 = \bigcup_{n < \omega} \hat{f}^1_{t_n} = \bigcup_{n < \omega} \hat{f}^2_{t_n} \restriction H_{(1, t_n)}$ is an embedding of $G_1$ into $G_1$ (cf. Lemma~\ref{lemma_pre_main_th}(3), (5)), in fact we have that  $\mathrm{dom}(\hat{f}^1_{t_n}) = H_{(1, t_n)}$ (cf. Lemma~\ref{lemma_pre_main_th}(3), (5)) and $G_1 = \bigcup_{n < \omega} H_{(1, t_n)}$, where $(H_{(1, t_n)} : n < \omega)$ is chain of pure subgroups of $G_1$ with limit $G_1$. %Following the convention laid out in Definition~\ref{lemma_pre_main_th}(\ref{convention_hatf_G1}), we let $\hat{f} \restriction G_1 = \hat{f}$, since we are considering this also as a map on $G_1$.
Clearly $\hat{f}^1$ is not of the form $g \mapsto mg$ for some $m \in \mathbb{Z} \setminus \{ 0 \}$, since for every $x \in \mrm{dom}(f_t)$ we have $x \neq f_t(x)$ (cf. Obs.~\ref{observation_new_elements_are_moved}(2)).  We claim that $G_1/\hat{f}^1[G_1]$ is not torsion. To this extent, first of all notice that $X_0 \neq \emptyset$ (by Definition~\ref{hyp_A2}(\ref{X0_not_empty})) and $X_0 \cap \mathrm{ran}(f_{t_n}) = \emptyset$ (by Definition~\ref{hyp_A2}(\ref{X0_cap_ft})). Thus, we have the following:
$$\mathrm{ran}(\hat{f}^1) \subseteq G^2_{X \setminus X_0} := \sum \{ \mathbb{Q}x : x \in X \setminus X_0 \} = \langle X \setminus X_0 \rangle^*_{G_2}.$$
%So $G_1/\mrm{ran}(\hat{f}) \leq G_2/G^2_{X \setminus X_0}$. 
Now, let $x \in X_0$, then $x \in G_1 \setminus \mrm{ran}(\hat{f}^1)$, moreover, for $q \in \mathbb{Q} \setminus \{ 0 \}$:
$$q x \notin G^2_{X \setminus X_0} \text{ and so } q x \notin \mrm{ran}(\hat{f}^1),$$
and so in particular, for every $0 < n < \omega$ we have that $nx \notin \mrm{ran}(\hat{f}^1)$, hence $n(x/(\mrm{ran}(\hat{f}^1)) \neq 0$. This concludes the proof of item (2).

\medskip
\noindent We now prove item (3). To this extent, suppose that $(T, <_T)$ is well-founded and, letting $G_1 = G_1[T]$, suppose that $\pi \in \mathrm{End}(G_1)$. We shall show that there is $m \in \mathbb{Z}$ such that, for every $a \in G_1$, $\pi(a) = ma$, i.e., $G_1$ is endorigid. We recall that the equivalence relation $E_1$ (used below) was defined in \ref{def_reasonable}(\ref{def_E_k_endo}).
\noindent
\newline \underline{Case 1}. The set $Y = \{x/E_1 : \text{ for some } y \in x/E_1, \pi(y) \notin \mathbb{Q}y\}$ is infinite.
\begin{enumerate}[$(*_1)$]
\item Choose $x_i$, $n_i$, for $i < \omega$, such that:
\begin{enumerate}
	\item $n_i$ is increasing with $i$;
	\item $x_i \in X_{n_{i+1}} \setminus X_{n_i}$;
	\item $\pi(x_i) \notin \mathbb{Q}x_i$, $\mrm{supp}(\pi(x_i)) \subseteq X_{n_{i + 1}}$;
	\item $X_{n_i} \cap x_i/E_1 = \emptyset$;
	\item $(x_i/E_1 : i < \omega)$ are pairwise distinct (this actually follows).
\end{enumerate}
\end{enumerate}
Note that, for $i < \omega$, we have:
\begin{enumerate}[$(*_2)$]
	\item $\mrm{supp}(\pi(x_i)) \subseteq x_i/E_1$, hence $\mrm{supp}(\pi(x_i)) \subseteq X_{n_{i+1}} \setminus X_{n_i}$.
\end{enumerate}
[Why? We apply \ref{lemma_pre_main_th}(\ref{item5}) with $(x_i, \pi(x_i), 1, (1))$ here standing for $(a, b, k, (q_\ell : \ell < k))$ there, so in particular $p = p_{x_i}$. In order to be able to apply \ref{lemma_pre_main_th}(\ref{item5}) we need that $b = \pi(a) \in G_{(1, p)}$, but this is clear in our case as $\pi \in \mrm{End}(G_1)$ and $p = p_{x_i}$. But then applying \ref{lemma_pre_main_th}(\ref{item5}) and writing $b = \pi(x_i)$ as there we get what we need.]

\smallskip
\noindent
For $r < \omega$, $(\mrm{supp}(x_\ell) : \ell \leq r)$ is a sequence of non-empty sets and $\mrm{supp}(x_\ell) \subseteq X_{n_{\ell+1}} \setminus X_{n_\ell}$, so it is a sequence of pairwise disjoint \mbox{non-empty sets.  Now, for $r < \omega$, let:}
$$x^+_r = \sum_{\ell \leq r} x_\ell, \; p_r = p_{x^+_r} \text{ and }\bar{x}_r = (x_\ell : \ell \leq r).$$

\smallskip
\noindent
As $\pi \in \mrm{End}(G_1)$, clearly $\pi(x^+_r) \in G_{(1, p_r)}$, hence by~\ref{lemma_pre_main_th}(\ref{item5}) applied to $x^+_r$, $\pi(x^+_r)$ here standing for $a, b$ there we can find $j_r, m_r >0$, and, for $j < j_r$, $\bar{y}^{(r, j)}$, $b^r_j$, $q^r_j \in \mathbb{Q}^+$ such that the following holds:
\begin{enumerate}[$(*_{3})$]
	\item
	\begin{enumerate}[(a)]
	\item for $j < j_r$, $\bar{x}_r \leq^{r+1}_X \bar{y}^{(r, j)}$;
	\item $(b^r_j = \sum_{\ell \leq r} y^{(r, j)}_\ell : j < j_r)$ is linearly independent;
	\item $\pi(x^+_r) = \sum \{q^r_j b^r_j : j < j_r \}$;
	\item for $j < j_r$, $m_r x^+_r \leq_* m_r b^r_j$ (and $m_r b^r_j \in G^+_0$).
\end{enumerate}
\end{enumerate}
\begin{enumerate}[$(*_{3.1})$]
	\item We define $f^1_{()}$ as the identity on $X$, hence, for $j < j_r$, TFAE:
	\begin{enumerate}[$(\cdot_1)$]
	\item $\bar{y}^{(r, j)} = \bar{x}_r$;
	\item $f^1_{()}(\bar{x}_r) = \bar{y}^{(r, j)}$;
	\item for all $0 < n < \omega$ and $\bar{t} \in T^n$, $f^1_{\bar{t}}(\bar{x}_r) \neq \bar{y}^{(r, j)}$.
	\end{enumerate}
\end{enumerate}
As $\bar{x}_r \leq^{r+1}_X \bar{y}^{(r, j)}$ we can apply \ref{unique_bart} and find a finite sequence $\bar{t}^r_j \in T^{< \omega}$ s.t.:
\begin{enumerate}[$(*_{4})$]
	%\item
%	\begin{enumerate}[(A)]
%	\item 
	\item if $\bar{x}_r <^{r+1}_X \bar{y}^{(r, j)}$, then: 
	\begin{enumerate}[(a)]
	\item $f^{1}_{\bar{t}^{r}_j}(\bar{x}_r) = \bar{y}^{(r, j)}$;
	\item $\hat{f}^{1}_{\bar{t}^r_j}(x^+_r) = b^r_j$;
	\item for $\ell \leq r$ and $j < j_r$ we have $f^{1}_{\bar{t}^{r}_j}(x_\ell) \neq x_\ell$ and $\mrm{lg}(\bar{t}^{r}_j) > 0$;	
	\item $\bar{t}^{r}_j = (t^r_{(j, \ell)} : \ell < \mrm{lg}(\bar{t}^r_j))$;
%	\item abusing our notation we let $f_{\bar{t}^r_j \restriction 0}$ be $\mrm{id}_{X_{n(r, j)}}$ for:
%	$$n(r, j) = \mrm{min}\{n \in (n_j, n_{j+1}) : x_r \in X_n \};$$
	\item if $j < j_r$ and $\ell < \mrm{lg}(\bar{t}^{r}_j)$, then:
	$$t <_T t^r_{(j, \ell)} \; \Rightarrow \; \bigvee_{m \leq r} f_{\bar{t}^r_j \restriction \ell}(x_m) \notin \mrm{dom}(f_t);$$
	\item in (e) this is equivalent to the following condition:
	$$t <_T t^r_{(j, \ell)} \; \Rightarrow \; f_{\bar{t}^r_j \restriction \ell}(x_r) \notin \mrm{dom}(f_t);$$
	\item $(\bar{t}^r_j : j < j_r)$ is without repetitions;
	\item $(f_{\bar{t}^r_j}(x_r) : j < j_r)$ is without repetitions.
\end{enumerate}
	%\item if $\bar{x}_r = \bar{y}^{(r, j)}$, then $\bar{t}^r_j = ()$ and we abuse notation as in (A)(e).
%\end{enumerate}
\end{enumerate}
Why $(*_{4})$? Concerning (e), recall \ref{unique_bart}. Concerning (f)-(h), recalling \ref{observation_Xtree}(\ref{item_10}), note that if $\mrm{lg}(\bar{t}^r_j) > 0$, then $\mrm{dom}(f_{\bar{t}^r_j})$ is $X_n$ for some $n < \omega$; so $x_r \in \mrm{dom}(f_{\bar{t}^r_j})$ and:
\begin{enumerate}[$(*_{4.1})$]
	\item $\ell \leq r \; \Rightarrow \; x_\ell \in \mrm{dom}(f_{\bar{t}^r_j}).$
\end{enumerate}
This concludes the proof of $(*_{4})$.
\begin{enumerate}[$(*_{5})$]
	\item For $r < \omega$ and $\ell \leq r$ we have $\pi(x_\ell) = \sum \{y^{(r, j)}_\ell : j < j_r\}$.
\end{enumerate}
[Why? Because $\{y^{(r, j)}_\ell : j < j_r\} \subseteq x_\ell/E_1$ and $(x_\ell/E_1 : \ell \leq r)$ is a sequence of pairwise disjoint sets.]
\newline Now, by induction on $r < \omega$, we choose $i_r$ and $y_r$ such that:
\begin{enumerate}[$(*_{6})$]
	\item 
	\begin{enumerate}[(a)]
	\item $i_r < j_r$, $y_r = f_{\bar{t}^r_{i_r}}(x_r) \neq x_r$, hence $\mrm{lg}(\bar{t}^r_{i_r}) > 0$;
	\item if $r > 0$, then $f_{\bar{t}^r_{i_r}}(x_{r-1}) = y_{r-1}$.
	\end{enumerate}
\end{enumerate}
Why $(*_{6})$ is possible? For $r = 0$, recall that $\pi(x_r) \notin \mathbb{Q}x_r$. For $r \geq 1$, 
%firstly, we use $(*_{5})$ to deduce that $(f_{\bar{t}^r_{j}}(x_{r-1}) : j < j_r)$ is well-defined and it sum is $\pi(x_{r-1})$. So 
by $(*_{5})$:
$$\sum \{f_{\bar{t}^r_{j}}(x_{r-1}) : j < j_r \} = \pi(x_{r-1}) = \sum \{f_{\bar{t}^{r-1}_{j}}(x_{r-1}) : j < j_{r-1}\}.$$
Now, by $(*_4)$(h) the sum in the RHS is without repetitions and of course $f_{\bar{t}^{r-1}_{i_{r-1}}}(x_{r-1})$ appears in it, hence it belongs to the support on the LHS, so for some $i_r < j_r$:
$$f_{\bar{t}^{r}_{i_{r}}}(x_{r-1}) = f_{\bar{t}^{r-1}_{i_{r-1}}}(x_{r-1}) = y_{r-1}.$$
As $f_{\bar{t}^{r}_{i_{r}}}(x_{r-1}) \neq x_{r-1}$, clearly $\mrm{lg}(\bar{t}^r_{i_r}) > 0$ and so $x_r \neq y_r$. This proves $(*_{6})$.
\newline Now, $f_{\bar{t}^{r-1}_{i_{r-1}}}(x_{r-1}) = f_{\bar{t}^{r}_{i_{r}}}(x_{r-1})$ and $\bar{t}^{r-1}_{i_{r-1}}$ satisfies $(*_4)$(e), hence by \ref{observation_Xtree}(\ref{item_12}) we have $\mrm{lg}(\bar{t}^{r-1}_{i_{r-1}}) = \mrm{lg}(\bar{t}^{r}_{i_{r}})$ and $\ell < \mrm{lg}(\bar{t}^{r}_{i_{r}})$ implies $t^{r-1}_{(i_{r-1}, \ell)} \leq_T t^{r}_{(i_{r}, \ell)}$. So $(\mrm{lg}(\bar{t}^r_{i_r}) : r < \omega)$ is constant, say constantly $k$, and if $\ell < k$, then $(t^r_{(i_r, \ell)} : r < \omega)$ is a $\leq_T$-sequence. But $x_r \notin X_{n_r}$ and so $t^r_{(i_r, \ell)} \notin T_n$, hence $(t^r_{(i_r, \ell)} : r < \omega)$ is $<_T$-increasing, and so we reach a contradiction. This concludes the proof of Case 1.

\smallskip
\noindent
 \underline{Case 2}. The set $Y = \{x/E_1 : \text{ for some } y \in x/E_1, \pi(y) \notin \mathbb{Q}y\}$ is finite and $\neq \emptyset$.
\newline Choose $x_0 \in X$ such that $\pi(x_0) \notin \mathbb{Q}x_0$. Let $n < \omega$ be such that $x_0 \in X_n$ and choose $x_1$ such that:
	\begin{enumerate}[$(\oplus_1)$]
	\item
	\begin{enumerate}
	\item $x_1 \in X \setminus \bigcup\{y/E_1 : y \in X_n\}$;
	\item $\pi(x_1) \in \mathbb{Q}x_1$.
	\end{enumerate}
\end{enumerate}
[Why possible? By the assumption in Case 2.]
\newline Notice now that:
\begin{enumerate}[$(\oplus_2)$]
	\item For $\ell \in \{1, 2\}$, $\mrm{supp}(x_\ell) \subseteq x_\ell/E_1$.
\end{enumerate}
\begin{enumerate}[$(\oplus_3)$]
	\item By \ref{lemma_pre_main_th}(\ref{item5}), there are $(\bar{t}_j, q_j : j < j_*)$ such that:
	\begin{enumerate}[(a)]
	\item $\bar{t}_j$ ($j < j_*$) are with no repetitions and $q_j \in \mathbb{Q}^+$;
	\item $\pi(x_0 + x_1) = \sum \{q_j f_{\bar{t}_j}(x_0 + x_1) : j < j_*\}$.
	\end{enumerate}
\end{enumerate}
\begin{enumerate}[$(\oplus_4)$]
	\item We have:
	\begin{enumerate}[(a)]
	\item $\pi(x_0) = \sum \{q_j f_{\bar{t}_j}(x_0) : j < j_*\}$;
	\item $\pi(x_1) = \sum \{q_j f_{\bar{t}_j}(x_1) : j < j_*\}$.
	\end{enumerate}
\end{enumerate}
[Why? $\pi(x_0) - \sum \{q_j f_{\bar{t}_j}(x_0) : j < j_*\} = - \pi(x_1) + \sum \{q_j f_{\bar{t}_j}(x_1) : j < j_*\}$. Now the LHS has support $\subseteq x_0/E_1$ and RHS has support $\subseteq x_1/E_1$. As $x_0/E_1 \cap x_1/E_1 = \emptyset$, both the LHS and the RHS are $0$, and so we are done.]
\newline However $\pi(x_1) \in \mathbb{Q}x_1$ by choice, and so:
\begin{enumerate}[$(\oplus_5)$]
	\item 
	\begin{enumerate}[(a)]
	\item for some $j < j_*$, $\bar{t}_j = ()$, w.l.o.g. for $j = 0$;
	\item for $0 < j < j_*$, $\mrm{lg}(\bar{t}_j) > 0$ (by $(\oplus_3)$(a)).
	\end{enumerate}
\end{enumerate}
\begin{enumerate}[$(\oplus_6)$]
	\item For $i = 0, 1$, let $\mathcal{E}_i$ be the following equivalence relation on $j_*$:
	$$\{(j_1, j_2) : f_{\bar{t}_{j_1}}(x_i) = f_{\bar{t}_{j_2}}(x_i) \}.$$
\end{enumerate}
\begin{enumerate}[$(\oplus_7)$]
	\item $0/\mathcal{E}_1 = \{0\}$ and if $0 < j < j_*$, then:
	\begin{enumerate}
	\item  $\sum \{q_\ell : \ell \in j/\mathcal{E}_1 \} = 0$;
	\item  $\sum \{q_\ell f_{\bar{t}_\ell}(x_1) : \ell \in j/\mathcal{E}_1\} = 0$.
\end{enumerate}
\end{enumerate}
[Why? Note that:
$$\sum \{q_\ell f_{\bar{t}_\ell}(x_1) : \ell \in j/\mathcal{E}_1\} = \sum \{q_\ell : \ell \in j/\mathcal{E}_1 \} f_{\bar{t}_j}(x_1).$$
So if $\sum \{q_\ell : \ell \in j/\mathcal{E}_1 \} \neq 0$, then $f_{\bar{t}_j}(x_1)$ belongs to the support of the RHS of $(\oplus_4)(b)$ but the support of this object is $\{x_1\}$ (by $(\oplus_1)$(b)) and $x_1 \neq f_{\bar{t}_j}(x_1)$, as $\bar{t}_j \neq ()$, together we reach a contradiction, and so we have $(\oplus_7)$(a)(b).]
\begin{enumerate}[$(\oplus_8)$]
	\item $E_1$ refines $E_0$.
\end{enumerate}
[Why? Assume that $j_1, j_2 < j_*$ and $j_1 \mathcal{E}_1 j_2$, this means that $f_{\bar{t}_{j_1}}(x_1) = f_{\bar{t}_{j_2}}(x_1)$. By \ref{hyp_A2}(\ref{hyp_A2_last}), as $x_1 \notin X_n$ we have that $X_n \subseteq \mrm{dom}(f_{\bar{t}_{j_1}}) \cap \mrm{dom}(f_{\bar{t}_{j_2}})$ and $f_{\bar{t}_{j_1}} \restriction X_n = f_{\bar{t}_{j_2}} \restriction X_n$. As $x_0 \in X_n$ we get that $f_{\bar{t}_{j_1}}(x_0) = f_{\bar{t}_{j_2}}(x_0)$, which means $j_1 \mathcal{E}_0 j_2$, as desired.]
\begin{enumerate}[$(\oplus_9)$]
	\item $0/E_0 = \{0\}$ and if $0 < j < j_*$, then:
	\begin{enumerate}
	\item $\sum \{q_\ell : \ell \in j/\mathcal{E}_0 \} = 0$;
	\item $\sum \{q_\ell f_{\bar{t}_0}(x_\ell): \ell \in j/\mathcal{E}_0 \} = 0$.
\end{enumerate}	
\end{enumerate}
[Why? By $(\oplus_7)$+$(\oplus_8)$, recalling \ref{hyp_A2}(\ref{hyp_A2_last}).]
\begin{enumerate}[$(\oplus_{10})$]
	\item $\pi(x_0) = q_0x_0$ (follows by $(\oplus_{9})$(b)).
\end{enumerate}
But $(\oplus_{10})$ contradicts our choice of $x_0$, as $\pi(x_0) \notin \mathbb{Q}x_0$.

\smallskip
\noindent
\underline{Case 3}. The set $Y = \{x/E_1 : \text{ for some } y \in x/E_1, \pi(y) \notin \mathbb{Q}y\}$ is empty.
\newline For $x \in X$, let $\pi(x) = q_x x$. Now, first of all we claim:
\begin{enumerate}[$(\star_{0.1})$]
	\item If $a \in G^+_1$, $\pi(a) = 0$ and $x/E_1 \cap \mrm{supp}(a) = \emptyset$, then $\pi(x) = 0$.
\end{enumerate}
[Why? Let $p = p_{x+a}$. Firstly, notice that $\pi(x+a) = \pi(x) + \pi(a) = q_x x$. Secondly, recalling that $x/E_1 \cap \mrm{supp}(a) = \emptyset$, notice that by \ref{lemma_pre_main_th}(\ref{G1p_pure_co-hop}) we have that $x \notin G_{(1, p)}$, but this contradicts that $x+a \in G_{(1, p)}$, as $\pi(x+a) = q_x x$.]
\begin{enumerate}[$(\star_{0.2})$]
	\item If $\pi$ is not $1$-to-$1$, then $\pi$ of the form $a \mapsto 0$, for all $a \in G_1$.
\end{enumerate}
[Why? Let $a \in G^+_1$ be such that $\pi(a) = 0$. If $y \in X \setminus \mrm{supp}(a)$, then we get that $\pi(y) = 0$, by applying $(\star_{0.1})$ to $(a, y)$. If $y \in \mrm{supp}(a)$, choose $x \in X \setminus \mrm{supp}(a)$ and apply $(\star_{0.1})$ to $(x, y)$.]
\begin{enumerate}[$(\star_{0.3})$]
	\item W.l.o.g. $\pi$ is $1$-to-$1$.
\end{enumerate}
[Why? Otherwise, by $(\star_{0.2})$, $\pi$ is multiplication by an integer and so we are done.] 
\begin{enumerate}[$(\star_1)$]
	\item $(q_x : x \in X)$ is constant.
\end{enumerate}
Why $(\star_1)$? Choose $x_0, x_1 \in X$ such that $q_{x_0} \neq q_{x_1}$ and, if possible, they are both $\neq 0$. Let $n < \omega$ be such that $x_0, x_1 \in X_n$ and choose a $<_X$-minimal $x_2 \in X_{n+1} \setminus X_n$, possible by \ref{observation_Xtree}(\ref{roots_item}).
Let $a = x_0 + x_1 + x_2$, $p = p_a$ (cf. \ref{def_G1_cohopfian}) and $\bar{x} = (x_0, x_1, x_2)$. As $a \in G_{(1, p)}$ and $\pi \in \mrm{End}(G_1)$, clearly $\pi(a) = b \in G_{(1, p)}$ and so by \ref{lemma_pre_main_th}(\ref{item5}) there are $j < \omega$ and, for $i < j$, $\bar{y}^i \in \mrm{seq}_3(X)$ and $q^i \in \mathbb{Q}^+$ such that $\bar{x} \leq^3_X \bar{y}^i$ and:
\begin{enumerate}[$(\star_{1.1})$]
	\item $b = \sum_{i < j} q^i (\sum_{\ell < 3} y^i_\ell)$.
\end{enumerate}
Notice that by $(\star_{0.2})$ we have $j > 0$ and w.l.o.g. we can assume that for $i < j-1$ we have $\bar{y}^{j-1} \not\leq^3_X \bar{y}^i$ and also that $\bar{x}$ is reasonable (so the $\bar{y}^i$'s are also reasonable).
Also:
\begin{enumerate}[$(\star_{1.2})$]
	\item $b = \pi(a) = q_{x_0} x_0 + q_{x_1} x_1 + q_{x_2} x_2$.
\end{enumerate} 
As $i < j-1$, implies $\bar{y}^{j-1} \not\leq^3_X \bar{y}^i$, clearly 
$y^{j-1}_2 \notin \{y^i_\ell : i < j-1, \ell \leq 2\} \cup \{y^{j-1}_0, y^{j-1}_1\}$ (by \ref{observation_Xtree}(\ref{observation_Xtree_item10}) and $\bar{y}^{j-1} \in \mrm{seq}_3(X)$), and so $y^{j-1}_2$ appears exactly once in the RHS of equation $(\star_{1.1})$, and so it appears in LHS of $(\star_{1.1})$, so $y^{j-1}_2 \in \mrm{supp}(b) \subseteq \{x_0, x_1, x_2\}$. But $x_2 \notin x_0/E_1 \cup x_1/E_1$, as $x_0, x_1 \in X_n$ and $x_2 \in X_{n+1} \setminus X_n$ is $<_X$-minimal. On the other hand, clearly $y^i_\ell \in x_\ell/E_1$ for $\ell \leq 2$ and $i < j$. Hence, necessarily, $y^{j-1}_2 = x_2$. Finally, as $x_2$ is $<_X$-minimal and for some $\bar{t} \in T^{< \omega}$, $f_{\bar{t}}(\bar{x}) = \bar{y}^{j-1}$, necessarily, $f_{\bar{t}}(x_2) = y^{j-1}_2$, so clearly $\bar{t} = ()$. Hence, $\bar{y}^{j-1} = \bar{x}$ and of course $\bar{x} \leq^3_X \bar{y}$ implies $f_{\bar{t}}(\bar{x}) \leq^3_X \bar{y}$.
Thus, by the statement after $(\star_{1.1})$, $j = 1$ and $\bar{y}^0 = \bar{x}$. So we have:

\begin{enumerate}[$(\star_{1.3})$]
	\item $q_{x_0} x_0 + q_{x_1} x_1 + q_{x_2} x_2 = q^0(y^0_0 + y^0_1 + y^0_2) = q^0(x_0 + x_1 + x_2)$.
\end{enumerate}
Thus, $q_{x_0} = q^0 = q_{x_1}$, contradicting our assumption that $q_{x_0} \neq q_{x_1}$. 
\begin{enumerate}[$(\star_2)$]
	\item Let $q_x = q_*$ for $x \in X$ (recalling $(\star_1)$).
\end{enumerate}
\begin{enumerate}[$(\star_3)$]
	\item $q_*$ is an integer.
\end{enumerate}
Why $(\star_3)$? Let $q_* = \frac{m}{n}$, with $m, n \in \mathbb{Z}^+$, $m$ and $n$ coprimes. Suppose that there is a prime $p$ such that $p \, \vert \, n$. Then we easily reach a contradiction noticing that:
\begin{enumerate}[$(\cdot)$]
	\item if $x \in X$ is $<_1$-minimal and $r$ is a prime different from $p_x$, then $r \not\vert \; x$;
	\item there are $<_1$-minimal $x, y \in X$ such that $x \neq y$.
\end{enumerate}
It follows that $n = 1$ and so $(*_3)$ holds.

\smallskip
\noindent Hence, our proof is complete, as Cases 1 and 2 are contradictory, while in Case 3 we showed that the arbitrary $\pi \in \mrm{End}(G_1)$ is indeed multiplication by an integer.
\end{proof}

	\begin{remark} Notice that, in the proof of \ref{main_th_Sec1}, Cases 2 and 3 do not use the assumption that $T$ is well-founded and so for an arbitrary tree $T$ (as in \ref{hyp_A1}) and $\pi \in \mrm{End}(G_1[T])$ we have:
\begin{enumerate}[(a)]
\item Case 1 happens if only if $T$ is not well-founded;
\item Case 2 never happens;
\item if Case 3 happens, then $\pi$ is multiplication by an integer.
\end{enumerate}
\end{remark}

\end{document}